\documentclass{article}
\usepackage[utf8]{inputenc}
\usepackage{fullpage}
\usepackage{hyperref}
\usepackage{xcolor}
\usepackage{amsmath,amsfonts,amssymb}
\usepackage{graphicx}%
\usepackage{subcaption}

\usepackage[margin=1.0in]{geometry}

\usepackage{accents}

\newcommand{\C}{\mathbb{C}}

\usepackage{amsthm}
\usepackage{mathrsfs}
\usepackage[shortlabels]{enumitem}
\newtheorem{theorem}{Theorem}[section]
\newtheorem{lemma}[theorem]{Lemma}
\newtheorem{proposition}[theorem]{Proposition}
\newtheorem{corollary}[theorem]{Corollary}
\newtheorem{remark}[theorem]{Remark}
\newtheorem{hypothesis}[theorem]{Hypothesis}

\newcommand{\tb}[1]{\textcolor{blue}{#1}}

\newcommand{\eps}{\varepsilon}
\newcommand{\E}{\mathbb E}\newcommand{\Pm}{\mathbb P}
\newcommand{\Rm}{\mathbb R}\newcommand{\Nm}{\mathbb N}
\newcommand{\dint}{\displaystyle\int}

\newcommand{\mQ}{\mathcal Q}
\newcommand{\tps}{\mathrm{t}}
\newcommand{\aver}[1]{\langle {#1} \rangle}
\newcommand{\dsum}{\displaystyle\sum}
\newcommand{\sff}{{\mathsf f}}
\newcommand{\sk}{{\mathsf k}}
\newcommand{\sm}{{\mathsf m}}
\newcommand{\sn}{{\mathsf n}}
\newcommand{\sC}{{\mathsf C}}
\newcommand{\sL}{{\mathsf L}}
\newcommand{\sM}{{\mathsf M}}
\newcommand{\sR}{{\mathsf R}}
\newcommand{\sS}{{\mathsf S}}
\newcommand{\sB}{{\mathsf B}}
\newcommand{\sth}{{\vartheta}}

\title{Complex Gaussianity of long-distance random wave processes }
\author{Guillaume Bal \thanks{Departments of Statistics and Mathematics and Committee on Computational and Applied Mathematics, University of Chicago, Chicago, IL 60637; guillaumebal@uchicago.edu} \and Anjali Nair \thanks{Department of Statistics and Committee on Computational and Applied Mathematics, University of Chicago, Chicago, IL 60637; anjalinair@uchicago.edu}}
\date{\today}

\begin{document}

\maketitle

\begin{abstract}
   Interference of randomly scattered classical waves naturally leads to familiar speckle patterns, where the wave intensity follows an exponential distribution while the wave field itself is described by a circularly symmetric complex normal distribution. In the It\^o-Schr\"odinger paraxial model of wave beam propagation, we demonstrate how a deterministic incident beam transitions to such a fully developed speckle pattern over long distances in the so-called scintillation (weak-coupling) regime.
\end{abstract}


\section{Introduction}
\label{sec:intro}

In a well-known paradigm of long-distance wave propagation in random media, the wave field is described as a random superposition of plane waves with Gaussian distributed amplitudes and uniformly distributed phases \cite{goodman1976some},\cite[Chapter 11]{sheng1990scattering}. Such a wave field is then {\em circularly symmetric complex normal (Gaussian)}, with independent zero-mean normal real and imaginary parts. As a consequence, the wave energy density is distributed according to an exponential law and the {\em scintillation index}, defined as the normalized variance of the energy density with respect to its squared mean, is unity. Wave fields at nearby points are also often observed to be essentially independent. This regime provides a model for the {\em speckle formation} observed in many experiments of classical wave propagation in heterogeneous media \cite{andrews2001laser,carminati2021principles,furutsu1972statistical,goodman1976some,sheng1990scattering}. While fairly intuitive and  well accepted in the physical literature \cite{furutsu1972statistical,valley1976application,yakushkin1978moments},\cite[Chapter 20]{ishimaru1978wave}, this conjecture is not entirely supported by any mathematical derivation. The main objective of this paper is to provide a complete derivation in the  {\em diffusive regime of the scintillation scaling} for an It\^o-Schr\"odinger paraxial wave model.

While a complete understanding of wave fields propagating in reasonably arbitrary  random media remains essentially out of reach, much progress has been made in the setting of beam propagation. In this setting, time-harmonic waves propagate primarily along a privileged direction while scattering from interactions with the underlying medium occurs only in the transverse {\color{black} and forward} directions. This is the {\em paraxial regime} of wave propagation, which has been derived mathematically from time-harmonic scalar wave models in a number of situations \cite{bailly1996parabolic, garnier2009coupled}. {\color{black} In the homogeneous setting, the paraxial approximation applies when the wavelength ($\lambda$) is much smaller than the beam width $(w_0)$ and the propagation distance is at most of order $\frac{w_0^2}{\lambda}$, called the Rayleigh length.} 
{\color{black} Further assuming that correlation length of the underlying medium is larger than the wavelength while smaller than the overall distance of propagation, the paraxial regime may be modeled by a random wave process satisfying an {\em It\^o-Schr\"odinger equation}}. This approximation was justified rigorously  \cite{bailly1996parabolic, garnier2009coupled} and the theory of the random wave process developed in \cite{dawson1984random}. 
A major technical advantage of the It\^o-Schr\"odinger model is the availability of closed-form partial differential equations for the statistical moments of the random wave process.

A direct analysis of the large-distance behavior of solutions to the It\^o-Schr\"odinger model also seems out of range. In fact, different asymptotic regimes may be considered leading to different statistical limits for the wave field. For instance, in the so-called spot dancing regime \cite{dawson1984random,furutsu1973spot}, the wave fields are not circularly Gaussian asymptotically but are rather shown to satisfy a Rice-Nakagami distribution for Gaussian incident conditions for the propagating beam. We will consider specifically the {\em scintillation scaling}, also referred to as the weak-coupling scaling, in which essentially all limiting models of wave propagation such as, e.g., radiative transfer or Fokker-Planck, are derived \cite{BKR-KRM-10}. It consists in assuming that the effects of the fluctuations are weak $\epsilon\ll1$ at the microscopic scale but integrated over long (adimensionalized) distances of propagation $z$ at least of order $\epsilon^{-1}$. In the scintillation (or weak-coupling) scaling, we consider the {\em kinetic} regime, where the distance of propagation $z\epsilon\approx 1$, and the {\em diffusive} regime, where $z\epsilon \gg 1$. Under natural assumptions on the power spectrum of the random medium and appropriate assumptions on the width of the propagating beam in the plane $z=0$, we show that in this latter regime displaying fully developed speckle, the wave process is indeed circular Gaussian with a scintillation function (normalized variance of the process) uniformly equal to unity. We also obtain in this regime that the expectation of the energy density (the irradiance) solves a diffusion equation in an appropriate set of variables. If we identify time with distance of propagation $z$ along the beam, then the mean distance in the transverse variables $\langle x\rangle \approx z^{\frac32}$ follows a very anomalous diffusion simply reflecting a faster-than-linear beam spreading.

\medskip

The field of wave propagation in (random) heterogeneous media has received considerable attention over the past decades. Meaningful macroscopic descriptions of wave fields appear for several weak-coupling, long-distance propagation regimes. In such regimes, the wave field itself is in general highly oscillatory and the main object of interest involves higher-order correlations of the field. In the radiative transfer regime, {\color{black} obtained when the medium correlation length and the typical wavelength of the wave field are comparable,} the main object of interest is {\color{black} the evolution of a} phase-space energy density. Its convergence to the deterministic solution of a linear Boltzmann equation is proved rigorously using diagramatic expansions in \cite{erdHos2000linear,LS-ARMA-07,Spohn}; see also \cite{RPK-WM,B-WM-05} for a formal derivation based on two-scale asymptotic expansions and \cite{sheng1990scattering} for formal ladder diagram expansions. {\color{black} The diffusive scaling in the radiative transfer regime of the random Schr\"{o}dinger equation  has been treated in \cite{erdHos2008quantum,hernandez2024quantum}.}  When the correlation length of the random medium is significantly larger than the typical wavelength, then the kinetic limit takes the form of a Fokker-Planck equation \cite{BKR-liouv} instead. Several results of convergence are presented in the review papers \cite{BKR-KRM-10,garnier2018multiscale}. These limiting deterministic models hold in sufficiently high dimension.  Wave propagation in {\color{black} one-dimensional or} randomly layered media is different, even in the weak-coupling regime because of Anderson localization. We refer to \cite{fouque2007wave} for results in this setting. 

A direct macroscopic analysis of the wave field, rather than phase-space energy densities, may be performed in the paraxial approximation \cite{bal2011asymptotics}. It is shown there that the wave field, written in dual Fourier variables and after appropriate phase compensation modeling propagation in a homogeneous medium, may be decomposed in the weak-coupling regime as a known deterministic component plus a mean-zero circular Gaussian variable. The derivation of the result also relies on a fair amount of combinatorial techniques to capture the interaction of the wave field with the underlying heterogeneous environment. In the It\^o-Schr\"odinger approximation of the paraxial regime, these results have recently been improved in \cite{gu2021gaussian} to provide a refined estimate of the statistical dependence of wave field at nearby wavenumbers. The analysis in the latter paper leverages the martingale structure of the random wave process developed in \cite{dawson1984random}. 

The results in \cite{bal2011asymptotics,gu2021gaussian} are obtained in the aforementioned kinetic regime of the scintillation scaling and for relatively narrow beams. They hold for wave fields in the Fourier domain and, while they provide information on the statistical stability of a properly averaged phase space energy density, it is not immediate to devise results for wave field in the physical variables where speckle is observed. It is in fact not clear that such results hold for narrow beams, or that they extend to the long-propagation setting of the diffusive regime.

\medskip

An analysis of broad beams in the kinetic regime of the scintillation scaling of the time-harmonic It\^o-Schr\"odinger model is presented in a series of papers \cite{garnier2014scintillation,garnier2016fourth,garnier2018noninvasive,garnier2022scintillation}, where it is shown that the statistical moments of order up to four of the random wave process in the physical variables are indeed consistent with the circular Gaussian conjecture. This is based on a careful analysis of the closed-form partial differential equations satisfied by the statistical moments of the random wave process and the asymptotic expansion of their solutions in the scintillation regime. 

The main technical component of our derivation revisits these partial differential equations for arbitrary-order statistical moments. In particular, we show that in Fourier variables, the solutions to these equations with arbitrary initial measures with bounded variation can be written, up to a negligible component, as a functional of first and second statistical moments as for complex Gaussian variables. This control allows us to pass to the diffusive (long-propagation distance) regime and after inverse Fourier transformation, to obtain error estimates in the physical variables in the uniform sense. These elements allow us to show convergence of finite dimensional distributions of the random wave process to a circular Gaussian limit. We also obtain stochastic continuity and tightness results showing that the random wave process converges in distribution to its limit as distributions over spaces of H\"older-continuous functions.

We note also that a steady-state solution of the It\^o-Schr\"odinger wave process is analyzed in \cite{fouque1998forward}.  When augmented with appropriate boundary conditions in the lateral variables, that paper shows that the wave field is circular complex normal. One of our main results is indeed to demonstrate that an initial deterministic beam profile converges to such a limit over long distances of propagation. Additional results on moments of the solution to the It\^o-Schr\"odinger wave equation may be found in {\color{black}\cite{bal2004self, bal2025long,bal2010dynamics,papanicolaou2007self}. Also see~\cite{bal2024long, bal2025splitting} for a direct analysis of statistical moments of the paraxial approximation in the scintillation regime and its role in justifying numerical schemes when $\eps=1$.}

\medskip

The rest of the paper is structured as follows. Section \ref{sec:process} introduces the random wave process solution of the It\^o-Schr\"odinger equation and its appropriate rescalings in the kinetic and diffusive sub-regimes of the scintillation regime. It also presents the main convergence results of the paper. The wave process is uniquely characterized by its first and second statistical moments in the scintillation regime. Such moments are computed explicitly in the regimes of interest in section \ref{sec:firstmoments}. The main technical part of the paper is a careful analysis of arbitrary statistical moments of the wave process in the scintillation regime. We analyze such moments in section \ref{sec:highermoments}. The proof of the main results on tightness and convergence properties as well as remarks and possible extensions are collected in section \ref{sec:final}.

\section{Random wave process and main results}
\label{sec:process}
\paragraph{It\^{o}-Schr\"{o}dinger wave process.}
We consider the setting of a time-harmonic wave beam propagating along a main axis $z\in [0,\infty)$ with lateral variables denoted by $x\in\Rm^d$ (with $d=2$ in the physical setting). 
{\color{black}
The starting model of scalar wave beam propagation in heterogeneous media is the following time-harmonic Helmholtz (wave) equation with varying propagation speed:
\[
   \Big( \partial^2_z + \Delta_x + \frac{\omega_0^2}{c^2(z,x)}\Big) p(z,x) =  p_0(x) \delta'_0(z)
\]
where $p_0(x)$ is the profile of the incident beam at $z=0$. Writing the central frequency  $\omega_0=c k_0$ for $c$ an effective speed, $k_0$ the carrier wave number of the source and  $n(z,x)=\frac{c_0}{c(z,x)}$ the refractive index where $c_0$ the speed of light in vacuum, the above equation with $n^2(z,x)=1+\nu(z,x)$ for $\nu(z,x)$ a random coefficient with variance $\sigma^2$ (say at $(0,0)$), is recast as
\[
   \Big( \partial^2_z + \Delta_x + k_0^2 (1+\nu(z,x))\Big) p(z,x) = p_0(x) \delta'_0(z).
\]
Different regimes of wave propagation appear for different scalings of the random coefficient $\nu(z,x)$ \cite{bal2015limiting,BKR-KRM-10,garnier2018multiscale,sheng1990scattering}. The paraxial approximation aims at considering high-frequency waves propagating over long distances along the main axis $z$, and then rescaling the random fluctuations so they contribute to leading order, which is modeled as
\[
  z \to \frac{z}{\theta},\qquad x \to x,\qquad k_0\to \frac{k_0}\theta,\qquad \sigma^2\to \theta^3 \sigma^2, \qquad p_0\to p_0,
\]
for some adimensionalized parameters $0<\theta\ll1$. The rapid oscillations along the main axis are compensated as follows
\[
  p(z,x) = e^{i\frac{ k_0 z}{\theta}} u(\theta z,x),\quad u(z,x) = e^{-i\frac{k_0}\theta \frac z\theta} p (\frac z\theta,x).
\]
With the approximation $|\theta^2\partial^2_z u|\ll1$, we formally obtain the following paraxial equation for the envelope $u(z,x)$:
\[
  \Big(2ik_0 \partial_z + \Delta_x + \frac {k_0^2}{\theta^{\frac12}} \nu(\frac z\theta,x) \Big)u = 0,\qquad u(0,x)=u_0(x)
\]
with incident beam profile $u_0(x)=p_0(x)$. For a process $\nu(z,x)$ with sufficiently short correlation (in $z$), a formal functional central limit approximation leads to 
\[
  \frac 1{\theta^{\frac12}} \nu(\frac z\theta,x) dz \approx B(dz,x)
\]
with $B(z,x)$ denoting a mean zero Gaussian process over $[0,\infty)\times \mathbb{R}^d$, with statistics given by 
\begin{equation}\label{eq:B}
    \mathbb{E}[B(z,x)B(z',y)]=\min(z,z') R(x-y).
\end{equation}
Here, $R(x-y)=\int_{-\infty}^\infty\mathbb{E}\nu(0,x)\nu(z,y)\mathrm{d}z$ denotes the lateral spatial covariance of the fluctuations in the medium. We collect the assumptions we make on $R$ in a paragraph below.

Formally passing to a limit $\theta\to0$ in the above equation for $u$ yields the It\^o-Schr\"{o}dinger equation
\begin{equation}\label{eq:ustrato}
 du = \frac{i}{2k_0}\Delta_x u dz + \frac{ ik_0}2 u \circ dB,\qquad u(0,x)=u_0(x).
\end{equation}
Here, $u\circ dB$ is product in the Stratonovich sense, reflecting the property that the random medium $\nu(z,x)$ is anticipating toward neither  positive nor negative values of $z$. For a rigorous derivation of the It\^o-Schr\"{o}dinger equation starting from the Helmholtz model, we refer the reader to~\cite{bailly1996parabolic, garnier2009coupled}.

Recasting product in~\eqref{eq:ustrato} in It\^o form to display the independence of $u(z,\cdot)$ and $B(dz,\cdot)$, we obtain 
\begin{equation}\label{eqn:u_SDE}
    \mathrm{d}u=\frac{i}{2k_0}\Delta_xu\mathrm{d}z-\frac{k_0^2R(0)}{8}u\mathrm{d}z+\frac{ik_0}{2}u\mathrm{d}B,\quad (z,x)\in[0,\infty)\times\mathbb{R}^d\,.
\end{equation}
The product $udB$ is written in the It\^o sense, i.e., such that $\E udB=0$ with $\E$ expectation associated to $\Pm$, the probability measure on a probability space where the Gaussian process is defined. This is the starting point in this paper.
}

We assume a source (incident boundary condition) of the form
\begin{equation*}
    u(0,x)=u_0(x),
\end{equation*}
with $u_0$ square integrable and such that its Fourier transform is the density of a measure with finite total variation. 

From the theory developed in \cite{dawson1984random}, we obtain that the process $u(z,x;\omega)$ is defined in $\sL^2(\mathbb{R}^d)$ uniformly in $z\ge 0$  $\mathbb{P}$-a.s., where $\omega$ is a realization of $B$ over a sample space $\Omega$.

\paragraph{Rescaled process in scintillation regime.}

The kinetic regime of the scintillation scaling, or weak-coupling regime \cite{garnier2016fourth,gu2021gaussian}, corresponds to a scaling of the form
\begin{equation*}
    z \to \frac{z}{\eps},\quad R \to \eps R(x),\quad 0< \eps < 1.
\end{equation*}
Under this scaling, we observe that the random wave process for $u^\eps$ takes the form
\begin{equation}\label{eq:uSDEk}
    \mathrm{d}u^\eps=\frac{i}{2k_0\eps}\Delta_xu^\eps\mathrm{d}z-\frac{k_0^2R(0)}{8}u^\eps\mathrm{d}z+\frac{ik_0}{2}u^\eps\mathrm{d}B.
\end{equation}
Distances $z=\mathcal{O}(1)$ correspond to an intermediate distance regime where the cumulative effects of random perturbations are of order $\mathcal{O}(1)$. We will refer to this as the kinetic scaling under scintillation regime. 

In order to analyze the long-distance propagation statistics under the scintillation regime, we introduce a parameter $\eta=\eta(\eps)$ such that  the rescaling from \eqref{eqn:u_SDE} now takes the form:
\begin{equation*}
z \to \frac{\eta z}{\eps},\quad R\to\frac{\eps}{\eta^{3}}R(x)\,.  
\end{equation*}
For $\eta\ll1$, specifically $\eta=(\ln|\ln\eps|)^{-1}$, this corresponds to a long distance propagation regime where the cumulative effects of randomness is $\frac{\eta}\eps \frac\eps{\eta^3}\gg 1$. For reasons that will become clear soon, we will refer to this as the diffusive regime. In this setting, we find
\begin{equation}\label{eq:uSDEd}
      \mathrm{d}u^\eps=\frac{i\eta}{2k_0\eps}\Delta_xu^\eps\mathrm{d}z-\frac{k_0^2R(0)}{8\eta^2}u^\eps\mathrm{d}z+\frac{ik_0}{2\eta}u^\eps\mathrm{d}B\,.  
\end{equation}
To analyze wave fields in the diffusive regime, we will assume that $R(x)$ is sufficiently smooth near $x=0$ and that $\nabla^2R(0)$ is a negative definite tensor with $R(x)$ maximal at $x=0$. 

We do not write the dependence of the wave field in $\eta$. The kinetic regime corresponds to $\eta(\eps)=1$ while in the diffusive regime, we have  $\eta(\eps)=(\ln|\ln\eps|)^{-1}$.

The incident conditions $u(z=0)$ have not been rescaled yet. A boundary condition $u(0,x)=u_0(x)$ independent of $\eps$ corresponds to what we will refer to as a narrow beam. These are the beams analyzed in \cite{gu2021gaussian,bal2011asymptotics}. Our results apply to broader beams, as in the analysis in \cite{garnier2014scintillation, garnier2016fourth}.

{\color{black} We consider sufficiently wide incident beam profiles of the form 
\begin{equation}\label{eq:u0}
    u^\eps_0(x)= \sff(\eps^\beta x),\quad \sff\in\mathcal{S}(\mathbb{R}^d),\ \  \beta\ge 1\,.
\end{equation}
Here, $\mathcal{S}(\mathbb{R}^d)$ is the space of Schwartz functions. We can also extend our analysis to include incident beam profiles given as finite superposition of plane waves modulated by wide sources:
\begin{equation}\label{eq:u0plane}
    u_0^\eps(x)=\sum_{\sm=1}^\sM \sff_\sm(\eps^\beta x)e^{i\sk_\sm\cdot x},\quad \sff_\sm\in\mathcal{S}(\mathbb{R}^d), \ \ \beta\ge 1\,,
\end{equation}
and $\Rm^d\ni \sk_\sm\neq \sk_{\sn}$ when $\sm\neq \sn$. These sources can be viewed as modulated rapidly oscillating plane waves. We consider only finitely many such superpositions as a continuous superposition could correspond to narrow incident beams for which our results presumably do not apply.
}

In both cases, what is important for our derivation is that the Fourier transform of $u_0^\eps$ is given by $\sum_{\sm=1}^\sM\eps^{-d\beta}\,\hat{\sff}_\sm(\eps^{-\beta}(\xi-\sk_\sm))$ in the latter case and $\eps^{-d\beta}\,\hat\sff(\eps^{-\beta}\xi)$ in the former case. These functions are integrable uniformly in $\eps$ so that $\widehat{u_0^\eps}(\xi)d\xi$ is a signed measure with uniformly finite total variation.

While a richer asymptotic limit is obtained in the diffusive regime for $\beta=1$, we also present the simpler case of very wide beam profile $\beta>1$ as a model of one (or $\sM$) incident plane wave(s).

{
\color{black} 
Note that while our incident beam is broad with a width of order $\eps^{-\beta}$ for $\beta\geq1$, the overall distance of propagation along the main axis after which we probe the signal is of order $\frac{\eta}{\eps\theta}\gg \frac{1}{\eps^\beta}$.

}


\begin{hypothesis}[Covariance function $R(x)$]\label{hyp:R} \rm
We assume throughout the paper that $R(x)=R(-x)\in \sL^1(\mathbb{R}^d)\cap \sL^\infty(\mathbb{R}^d)$. This implies that $\hat{R}(\xi)\in \sL^1(\mathbb{R}^d)\cap \sL^\infty(\mathbb{R}^d)$ as well. As a power spectrum, we have $\hat R(\xi)\geq0$ although all arguments below would apply after replacing $\hat R(\xi)$ by $|\hat R(\xi)|$. 

We also assume the existence of a radially symmetric $\hat \sR(\xi)\in \sL^1(\Rm^d)$ such that $\hat R(\xi) \leq \hat \sR(\xi)=\hat \sR(|\xi|)$. 

We next assume that for each $e\in {\mathbb S}^{d-1}$ and $\tau\in\Rm^d$, then $s\mapsto R(\tau+se)$ is integrable. 

Finally, in lateral dimension $d\geq3$, we assume that $\aver{\xi}^{d-2}\hat R(\xi)\in \sL^\infty(\Rm^d)$, where $\aver{\xi}:=\sqrt{1+|\xi|^2}$.

These assumptions are needed in the kinetic regime $\eta(\eps)=1$. In the diffusive regime with $\eta^{-1}=\ln\ln\eps^{-1}$, we further {\color{black} observe that $R(x)$ is maximal at $x=0$ and assume it is }sufficiently smooth there that the Hessian $\Gamma=\nabla^2 R(0)$ is defined and negative definite. 
\end{hypothesis}

The above hypotheses are sufficient to obtain convergence of finite dimensional distributions. Additional natural continuity requirements are imposed in Theorem \ref{thm:tightness} to deduce stochastic continuity and tightness properties of $u^\eps$.

\paragraph{Spatially rescaled random vector.}
We know from the results in \cite{dawson1984random} that {\color{black} for a realization $\omega$ of $B$ over sample space $\Omega$,} $u^\eps(z,x;\omega)$ with condition at $z=0$ given by \eqref{eq:u0plane} is defined as a continuous process in $z$ with values in $\sL^2(\Rm^d\times\Omega)$ and hence for a given $z$ for almost all $(x,\omega)$. Moreover, restricting the definition of $\eps$ to a countable sequence, for instance $\eps^{-1}\in\Nm^*$, we may assume $u^\eps(z,x;\omega)$ thus defined for all such values of $\eps$. We drop the dependence in $\omega$. 

For $z>0$ and $r\in\Rm^d$ fixed, let us define the rescaled process
\begin{equation}\label{eq:phieps}
   \phi^\eps(z,r,x)=u^\eps(z,\eps^{-\beta}r+\eta x).
\end{equation}
Let $X=(x_1,\cdots,x_N)$ be a collection of $N$ points in $\Rm^d$. For almost all choices of $X$, we define the random vector
\begin{equation}\label{eq:Phieps}
   \Phi^\eps(z,r,X) = (\phi^\eps(z,r,x_1),\cdots,\phi^\eps(z,r,x_N))
   =(u^\eps(z,\eps^{-\beta}r+\eta x_1),\cdots,u^\eps(z,\eps^{-\beta}r+\eta x_N)).
\end{equation}

Our objective is to establish the convergence of the finite dimensional distributions $\Phi^\eps$ and of the process $\phi^\eps$ as $\eps\to0$. We recall that $\eta(\eps)=1$ in the kinetic regime while $\eta^{-1}=\ln\ln\eps^{-1}$ in the diffusive regime. In particular, we will establish that $\Phi^\eps-\E\Phi^\eps$ is in law approximately circularly symmetric complex Gaussian in the kinetic regime as $\eps\to0$ while $\Phi^\eps$ {\em itself} is in law approximately circularly symmetric complex Gaussian as $\eps\to0$ in the diffusive regime. The latter case corresponds to the setting of fully developed speckle. We recall that a process is circularly symmetric complex Gaussian if its finite dimensional distributions $Z=(Z_1,\ldots,Z_n)$ are circularly symmetric complex Gaussian, i.e., when $Z=\Re Z+i\Im Z$ with the real part $\Re Z$ and imaginary part $\Im Z$ independent zero-mean Gaussian processes. The moments characterizing such processes are given explicitly in Remark \ref{rem:GSR} below.

\paragraph{Main results.} We assume that $R$ satisfies the assumptions in Hypothesis \ref{hyp:R} and first that $\eta(\eps)=1$.
\begin{theorem}[Kinetic regime]\label{thm:kinetic}
    The mean zero random vector $\Phi^\eps-\mathbb{E}[\Phi^\eps]\boldsymbol{\Rightarrow}\tilde{\Phi}$ in distribution as $\eps\to 0$, where $\tilde{\Phi}=(\tilde{\phi}_1,\ldots,\tilde{\phi}_N)$ is a circularly symmetric Gaussian random vector characterized by
    \begin{equation}\label{eq:mtsZ}
        \mathbb{E}[\tilde{\phi}_j\tilde{\phi}_l]=0\qquad \mbox{ and } \qquad       \mathbb{E}[\tilde{\phi}_j\tilde{\phi}^\ast_l]=\widetilde{M}_{1,1}(z,r,x_j,x_l)\,.
    \end{equation}
    Here, $\widetilde{M}_{1,1}$ is the limiting centered second moment given by \eqref{eqn:plane_wave_mod}-\eqref{eqn:mu_2_plane_wave_super_pos} below in both cases $\beta=1$ and $\beta>1$.
\end{theorem}
\begin{remark}[Kinetic regime for smooth incident beams]\label{rem:cv} \rm When the incident beam profile is given by \eqref{eq:u0}, the mean vector $\mathbb{E}[\Phi^\eps]$ converges as $\eps\to0$. This is not the case in general when \eqref{eq:u0plane}
 holds. {\color{black}This is because the first moment $\mathbb{E}[\phi^\eps]$ has a limit as $\eps\to0$ when the incident beam is of the form \eqref{eq:u0} but remains oscillatory for plane-wave incident beams in \eqref{eq:u0plane} in the kinetic regime.} Then, $\Phi^\eps\boldsymbol{\Rightarrow}\Phi$ in distribution as $\eps\to 0$, where $\Phi=({\phi}_1,\ldots,{\phi}_N)$ is a random vector such that $\mathbb{E}[\phi_j]=M_{1,0}(z,r,x_j)$ and $\mathbf{Z}=\Phi-\mathbb{E}[\Phi]$ is a circularly symmetric Gaussian random vector characterized by
        \begin{equation}\label{eq:mtsZ_raw}
        \mathbb{E}[Z_jZ_l]=0\qquad \mbox{ and } \qquad       \mathbb{E}[Z_jZ^\ast_l]=M_{1,1}(z,r,x_j,x_l)-M_{1,0}(z,r,x_j)M_{0,1}(z,r,x_l),
    \end{equation}
where $Z_j=\phi_j-\mathbb{E}[\phi_j]$. Here, $M_{1,1}$ is given by \eqref{eqn:mu_2_kinetic_beta>1} when $\beta>1$ and by \eqref{eqn:mu_2_kinetic_beta=1} when $\beta=1$. $M_{1,0}$ is given by~\eqref{eqn:mu_1_limit} and $M_{0,1}=M_{1,0}^*$.    
\end{remark}
Assume further that $R$ is sufficiently smooth and maximal at $x=0$ with $\Gamma=\nabla^2 R(0)<0$. Then we have for $\eta^{-1}(\eps)=\ln\ln\eps^{-1}$ the convergence result:
\begin{theorem}[Diffusive regime]\label{thm:diffusive}
        The random vector  $\Phi^\eps\boldsymbol{\Rightarrow}\Phi$ in distribution as $\eps\to 0$, where  $\Phi$ is a circularly symmetric Gaussian random vector with elements $\{\phi_j\}_{j=1}^N$ such that
    \begin{equation}\label{eq:mtphi}
        \mathbb{E}[\phi_j\phi_l]=0,\quad \mathbb{E}[\phi_j\phi^\ast_l]=M_{1,1}(z,r,x_j,x_l)\,.
    \end{equation}
  Here, $M_{1,1}$ is given below by~\eqref{eqn:mu_2_diffusive_beta>1} when $\beta>1$ and~\eqref{eqn:mu_2_diffusive_beta=1} when $\beta=1$. 
\end{theorem}
These theorems are proved in section \ref{sec:final}. The explicit expressions for the moments $\widetilde{M}_{1,1}$, $M_{1,1}$, and $\mathbb{E}[\phi^\eps]$ are given in section \ref{sec:firstmoments} in the kinetic and diffusive regimes. 

The result of Theorem \ref{thm:diffusive} shows that in the diffusive regime, the random wave field $\phi^\eps$ has finite dimensional distributions that are indeed circularly symmetric normal asymptotically as $\eps\to0$. Moreover, we obtain the following immediate corollaries:
\begin{corollary}[Scintillation in the diffusive regime]\label{cor:scint}
Define the intensity
\begin{equation}\label{eqn:I_eps}
    I^\eps(z,r,x)=|\phi^\eps(z,r,x)|^2 = |u^\eps(z,\eps^{-\beta} r +\eta x)|^2\,.
\end{equation}
    Then $I^\eps\boldsymbol{\Rightarrow}I$ in distribution as $\eps\to 0$ where $I(z,r)$ at (z,r) fixed is distributed according to an exponential law with $\mathbb{E}[I](z,r)$ given by $\sum_{\sm=1}^\sM|\,\sff_\sm(r)|^2$ (or $|\,\sff(r)|^2$ for the incident beam \eqref{eq:u0}) when $\beta>1$ and by $\mathbb{E}[I](z,r)=I_2(z,r)$, the solution to a diffusion equation \eqref{eqn:I2_diffusion} when $\beta=1$ independent of $x$. In particular, the scintillation index
    \begin{equation}\label{eq:scintindex}
    \sS(z,r):=\frac{\mathbb{E}[I^2]-\mathbb{E}[I]^2}{\mathbb{E}[I]^2}=1,\quad\forall z>0, \ r\in\mathbb{R}^d\,.
    \end{equation}
\end{corollary}
The above corollary fully characterizes the scintillation regime with a scintillation index uniformly equal to $1$ corresponding to a fully developed speckle pattern. The diffusion equation \eqref{eqn:I2_diffusion} obtained when $\beta=1$ is naturally written in the `time' variable $z^3$ and lateral `spatial' variables $r$, resulting in a very anomalous diffusion regime with $r\sim z^{\frac32}$ for large $z$ {\color{black} corresponding to anomalous} beam spreading.
\begin{corollary}[Self-averaging in the diffusive regime]\label{cor:selfaver} Consider the average intensity over a region $D$
    \begin{equation}
        I^{\eps}_D(z,r,x)=\frac{1}{|D|}\int_{D}I^\eps(z,r+r',x)\mathrm{d}r'\,,
    \end{equation}
    where $D$ is a cube centered at the origin in $\mathbb{R}^d$ with length $a_\eps\gg \eps^\beta \eta$. Then $I^{\eps}_D(z,r,x) \boldsymbol{\Rightarrow} I_{D}(z,r,x)$ in probability as $\eps\to 0$ where the deterministic limit $I_D$ is given by
    \begin{equation}
        I_{D}(z,r,x)= \left\{ \begin{array}{cl} \dfrac{1}{|D|}\dint_{D} \E I(z,r+r')\mathrm{d}r',  & a_\eps=a>0 \\
        \E I(z,r), & a_\eps\to0, \end{array}\right.
    \end{equation}
    with $I$ as in Corollary \ref{cor:scint}.
\end{corollary}

The corollaries are also proved in section \ref{sec:final}.  These results display the multi-scale structure of the wave field. While it remains statistically unstable at the scale $(z,r)$, it becomes self-averaging when spatially averaged over domains that are much larger than $\eta$ (at the microscopic scale $x$ or much larger than $\eps^\beta\eta$ at the macroscopic scale $r$).

The above results apply to finite dimensional distributions of $\phi^\eps$. With minor additional constraints on $R(x)$ and the incident beam profile $u_0^\eps(x)$, we can in fact prove that for each fixed $(r,z)$, $\phi^\eps$ has a H\"older continuous version and the family of measures it generates is tight. More precisely, we have the following result. 

\begin{theorem}[Stochastic continuity and tightness]\label{thm:tightness}
    Assume $\aver{k}^{2\alpha_R}\hat R(k) \in \sL^1(\Rm^d)$ for $\alpha_R\in(0,1]$ in the kinetic regime and $\alpha_R=1$ in the diffusive regime. 
    Assume an incident condition of the form \eqref{eq:u0plane} (or \eqref{eq:u0}) with a bound $\|\aver{k}^{2\alpha}\hat \sff_\sm(k)\|\leq C$ (or $\|\aver{k}^{2\alpha}\hat \sff(k)\|\leq C$) for $\alpha_R\geq \alpha>0$. 

    Let $\alpha_0=\alpha$ in the kinetic regime and $0<\alpha_0<\alpha=1$ in the diffusive regime.
    For $C(z,\alpha_0)$ independent of $\eps\in(0,1]$, we have
    \begin{equation}\label{eq:tightness}
       \sup_{s\in [0,z]} \E |\phi^\eps(s,r,x+h) - \phi^\eps(s,r,x)|^{2n} \leq C(z,\alpha_0) |h|^{2\alpha_0 n},\quad \mbox{ for all} \ h\in \sB(0,1)\subset \Rm^d.
    \end{equation}
    Choosing $n$ large enough so that $2\alpha_0 n\geq d+2\alpha_- n$ for $0<\alpha_-<\alpha_0$ shows that for each $z>0$, a version of $\phi^\eps(z)$ is H\"older continuous in $C^{0,\alpha_-}(\Rm^d)$ and the process $\phi^\eps(z)$ is tight on $C^{0,\alpha_-}(\Rm^d)$. The latter continuity and tightness properties also hold for the centered process $\phi^\eps(z)-\E\phi^\eps(z)$.
\end{theorem}
The proof is given in section \ref{sec:final}.
\begin{corollary}[Convergence of processes]\label{cor:conv}
    Under the hypotheses of the preceding theorem, the processes whose finite dimensional distributions are shown to converge in Theorems \ref{thm:kinetic} and \ref{thm:diffusive}  (and Remark \ref{rem:cv}) in the kinetic and diffusive regimes converge in distribution as probability measures on $C^{0,\alpha_-}(\Rm^d)$. 
\end{corollary}
The proof is an immediate consequence of the fact that tight processes converge in the above sense when their finite dimensional distributions converge \cite{billingsley2017probability,kunita1997stochastic}. Under appropriate smoothness conditions on $\hat R$ and $\hat \sff_\sm$ (or $\hat\sff$), we may choose $\alpha_R=\alpha=1$ with $\alpha_-\in [0,1)$ arbitrary.

As a direct extension of the above results, we also readily obtain that a vector of processes $\phi^\eps(z,r_j,x)$ for different values of $r_j$ for $1\leq j\leq N$ converges in the same sense to their independent limits $\phi(z,r_j,x)$. It is quite possible that the processes satisfies joint continuity properties in $(z,x)$ rather than separately in $z$ as shown in \cite{dawson1984random} when $\eps=1$ and $x$ as shown in the above theorem combining the above theory with the martingale techniques in \cite{dawson1984random,gu2021gaussian}. This is not considered here.

\paragraph{Main steps of the derivation.} The main technical part of the derivation is a detailed analysis of the statistical moments of $u^\eps$ and $\phi^\eps$ as $\eps\to0$. For $X=(x_1,\cdots,x_p)\in \Rm^{pd}$ and $Y=(y_1,\cdots,y_q)\in\Rm^{qd}$, we define the $p+q$th moment of $u^\eps$  and $\phi^\eps$ as
\begin{equation}\label{eq:mupq}
\mu_{p,q}^\eps(z,X,Y)=\mathbb{E}\Big[\prod_{j=1}^pu^\eps(z,x_j)\prod_{l=1}^q{u^{\eps*}}(z,y_l)\Big],
    \quad m_{p,q}^\eps(z,r,X,Y) = \E \Big[\prod_{j=1}^p \phi^\eps(z,r,x_j)\prod_{l=1}^q \phi^{\eps\ast}(z,r,y_l)\Big].
\end{equation}
It is well known that such moments satisfy closed-form equations. They are given explicitly in \eqref{eq:mupqpde} below. This is a specificity of the It\^o-Schr\"odinger regime of wave propagation. The above terms allow us to express arbitrary moments of the random vector $\Phi^\eps$ in \eqref{eq:Phieps} as well. Denote by $v\in\Rm^{(p+q)d}$ the dual (Fourier) variables to $(X,Y)$, by $\hat\mu^\eps_{p,q}(z,v)$ the Fourier transform of $\mu_{p,q}^\eps(z,X,Y)$ in \eqref{eq:hatmupq}, and finally by $\psi^\eps_{p,q}(z,v)$ a phase compensated wave field such that
\[
 \hat\mu^\eps_{p,q}(z,v) = \Pi^\eps_{p,q}(z,v)  \psi^\eps_{p,q}(z,v), \quad \Pi^\eps_{p,q}(z,v) = e^{-\frac{iz\eta}{2k_0\eps}  \, v^{\mathrm t} \Theta v}.
\]
Here, $\Theta$ is a diagonal matrix of $p$ ones followed by $q$ minus ones; see \eqref{eq:Phipq}. The phase compensation ensures that $\psi^\eps_{p,q}$ would be constant in $z$ if propagation occurred in a homogeneous medium with $R(x)\equiv0$.

The moments $\psi^\eps_{p,q}(z,v)$ of the compensated wave field satisfy closed form equations of the form
\[
  (\partial_z - L^\eps_{p,q}) \psi^\eps_{p,q}=0,\quad L^\eps_{p,q} = \sum_j L^\eps_j
\]
with each operator $L^\eps_j$ of the form
\[
   L^\eps_j \rho (z,v) = \frac{c_j}{\eta^2} \dint_{\Rm^d}  \hat R(k) e^{\frac{iz\eta}{2k_0\eps}  (v,k)^{\mathrm t} \check A_j  (v,k)} \rho(z,v- A_j k) dk,
\]
for appropriate matrices $(A_j,\check A_j)$ that are described in detail in section \ref{sec:highermoments}. The highly oscillatory phases generate a number of cancellations as $\eps\to0$, which we leverage to obtain the following result. Let $U^\eps_{p,q}(z)$ be the solution operator of the evolution equation $(\partial_z - L^\eps_{p,q})\psi=0$ mapping a boundary condition at $z=0$ to a solution at a given $z>0$. We will obtain that $U^\eps_{p,q}$ is a bounded operator on the Banach space  ${\mathcal M}_B(\Rm^{pd+qd})$ of complex-valued, finite, signed Borelian measures with norm given by their total variation (TV). Any finite signed measure may be decomposed as its real and imaginary part, and for each such part, $\nu=\nu_+-\nu_-$ for $\nu_\pm$ positive signed measures so that the total variation $\|\nu\|:=|\nu|(\Rm^{pd+qd})=\nu_+(\Rm^{pd+qd})+\nu_-(\Rm^{pd+qd})$. The total variation of the complex-valued measure is then the sum of the total variations of its real and imaginary components.
The operators $L^\eps_j$ above are clearly bounded in operator norm as operators acting on ${\mathcal M}_B(\Rm^{pd+qd})$. We also denote by $\|\cdot\|$ such an operator norm. 

Throughout the paper, we use the notation $\|\cdot\|$ exclusively as either the total variation norm of a bounded measure or as the uniform norm of operators acting on spaces of bounded measures. We use the norms $\|\cdot\|_\infty$ and $\|\cdot\|_1$ as the uniform (supremum) norm and the $\sL^1$-norm of bounded and integrable measurable functions, respectively.

We will show in Theorem \ref{thm:Ueps} below that $U^\eps_{p,q}= N^\eps_{p,q} + E^\eps_{p,q}$ where $N^\eps_{p,q}$ corresponds to the $(p,q)$ moment of a circular complex Gaussian variable and where 
\[
  \| E^\eps_{p,q} \| \ll c(p,q,z) \eps^{\frac13},
\]
for the choice $\eta^{-1}=\ln\ln\eps^{-1}$ and for some $c(p,q,z)$ bounded on compact domains. This shows that in the Fourier domain and after phase compensation, the $(p,q)$ moments $\psi^\eps_{p,q}$ are approximately equal to moments coming from a circular complex normal variable.

An advantage of the Banach space ${\mathcal M}_B(\Rm^{pd+qd})$ is that upon inverse Fourier transform, the error term $E^\eps_{p,q}$ translates into a controlled error in the uniform sense in the physical variables. For incident beams of the form \eqref{eq:u0}, this allows us to write 
\begin{eqnarray} \nonumber
      \mu_{p,q}^{\eps }(z,X,Y) \! &=&\! \!\mathscr{F}(\mu^\eps_{1,0}(z,x_1),\cdots,\mu^\eps_{1,0}(z,x_p),\mu^\eps_{0,1}(z,y_1),\cdots,\mu^\eps_{0,1}(z,y_q),\mu^\eps_{1,1}(z,x_1,y_1),\cdots,\mu^\eps_{1,1}(z,x_p,y_q))\\ & + &  \!\!{\mathcal O}_{\|\cdot\|_\infty}(\eps^{\frac13}c(p,q,z)) 
      \label{eq:mupqlimit}
\end{eqnarray}
where $\mathscr{F}$ defined in \eqref{eq:F} below is the continuous functional describing $(p,q)$ moments of a complex circular Gaussian variable in terms of its first- and second- moments. This also shows that $\phi^\eps(z,X,Y)$ has moments approximately given by those of a complex Gaussian variable. Since such moments characterize their distribution, we deduce that $\Phi^\eps$ converges in distribution to its normal limit provided the right-hand side in \eqref{eq:mupqlimit} converges. Since $\mathscr{F}$ is continuous, this is a consequence of the convergence of first- and second- moments to their limits. This convergence, and the constraint that $\|\hat\phi^\eps_{p,q}(0,v)\|$ be of order ${\mathcal O}(1)$ in total variation dictates our choice of incident beam profiles \eqref{eq:u0} and \eqref{eq:u0plane}. 

For incident profiles in \eqref{eq:u0}, we will obtain convergence of first- and second- moments in both the kinetic and the diffusive regimes.  For incident beam profiles \eqref{eq:u0plane}, the first and second moments converge in the diffusive regime but remain oscillatory in the kinetic regime as $\eps\to0$. What we obtain instead is that the centered second moments converge. This allows us to still express the $p+q$th moments of the centered variable $\tilde{u}^\eps=u^\eps-\mathbb{E}[u^\eps]$ as 
\begin{eqnarray}\label{eqn:mu_bar_pq}
   \tilde{\mu}^\eps_{p,q}(z,X,Y)= \mathbb{E}[\prod\limits_{j=1}^p\tilde{u}^\eps(z,x_j)\prod\limits_{j=1}^q{\tilde{u}^{\eps\ast}(z,y_l)}]=\widetilde{\mathscr{F}}(\tilde{\mu}^\eps_{1,1}(z,x_1,y_1),\cdots,\tilde{\mu}^\eps_{1,1}(z,x_p,y_q))+ {\mathcal O}_{\|\cdot\|_\infty}(\eps^{\frac13}),
   \end{eqnarray}
where $\widetilde{\mathscr{F}}$ defined in \eqref{eqn:bar_F} below, is a continuous functional which describes the moments of a complex circular Gaussian random variable in terms of its centered second moments $\tilde{\mu}^\eps_{1,1}(z,x,y)$.

First and second moments $\mu^\eps_{0,1}, \mu^\eps_{1,0}, \mu^\eps_{1,1}$ and hence centered second moments $\tilde\mu^\eps_{1,1}$, satisfy closed form expressions. These and their limiting behavior are presented in detail in section \ref{sec:firstmoments}. They fully characterize the limiting distributions expressed in Theorems \ref{thm:kinetic}, \ref{thm:diffusive}, Remark \ref{rem:cv}, and Corollary \ref{cor:conv}.

The final step of the derivation consists in proving results of stochastic continuity and tightness for the process $\phi^\eps$. This is done by looking at moments of $\phi^\eps(z,r,x+h)-\phi^\eps(z,r,x)$ and showing that they satisfy equations similar to those for $\mu^\eps_{p,q}$, except that these equations involve source terms that are bounded in TV norm in the dual variables.

\section{Limiting expressions for first and second statistical moments}
\label{sec:firstmoments}

Using It\^o's formula for Hilbert space-valued processes \cite{garnier2016fourth,miyahara1982stochastic}, we obtain closed-form partial differential equations for the moments of copies of $u^\eps$. These equations are not solvable explicitly except for first and second moments. This section recalls such equations and analyzes their solutions in the scintillation regime.

\subsection{First moment}
The first moment may be obtained by realizing that $\E udB=0$ and taking mathematical expectation in \eqref{eq:uSDEd}. Defining $\mu_{1,0}^\eps(z,x)=\mathbb{E}[u^\eps(z,x)]$, we have
\begin{equation}\label{eq:mu1}
    \partial_z\mu_{1,0}^\eps=\frac{i\eta}{2k_0\eps} \Delta_x \mu_{1,0}^\eps-\frac{k_0^2}{8\eta^2}R(0)\mu_{1,0}^\eps,\quad \mu_{1,0}^\eps(0,x)=u_0^\eps(x)\,.
\end{equation}
Defining the Fourier transform
\begin{equation*}
  \hat{\mu}_{1,0}^\eps(z,\xi)=\int_{\Rm^d}
  \mu_{1,0}^\eps(z,x)e^{-i\xi\cdot x}\mathrm{d}x\,,
\end{equation*}
we obtain
\begin{equation*}
    \partial_z\hat{\mu}_{1,0}^\eps=\Big(-\frac{i\eta}{2k_0\eps}|\xi|^2-\frac{k_0^2}{8\eta^2}R(0)\Big)\hat{\mu}_{1,0}^\eps,\quad \hat{\mu}_{1,0}^\eps(0,\xi)=\hat{u}_0^\eps(\xi)\,.
\end{equation*}
This can be solved explicitly, which after an inverse Fourier transform gives 
\begin{equation*}
    \mu_{1,0}^\eps(z,x)=e^{-\frac{k_0^2}{8\eta^2}R(0)z}\int_{\Rm^d}\hat{u}_0^\eps(\xi)e^{-\frac{iz\eta}{2k_0\eps}|\xi|^2}e^{i\xi\cdot x} \frac{\mathrm{d}\xi}{(2\pi)^d}\,.
\end{equation*}

When $u_0^\eps$ is of the form in~\eqref{eq:u0}, after a change of variables $\eps^{-\beta}\xi\to \xi$, we obtain 
\begin{equation*}
     m^\eps_{1,0} (z,r,x) = \mu_{1,0}^\eps(z,\eps^{-\beta} r + \eta x)=e^{-\frac{k_0^2}{8\eta^2}R(0)z}\int_{\Rm^d}\hat{\sff}(\xi)e^{-\frac{iz\eta}{2k_0}\eps^{2\beta-1}|\xi|^2}e^{i(r+\eps^\beta\eta x)\cdot \xi} \frac{\mathrm{d}\xi}{(2\pi)^d},
\end{equation*}
with $m_{1,0}^\eps(z,r,x)=\mathbb{E}\phi^\eps(z,r,x)$. Let $M_{1,0}(z,r,x)=\lim_{\eps\to 0}m_{1,0}^\eps(z,r,x)$. Since $\hat{\sff}(\xi)$ is integrable, we obtain from the Lebesgue dominated convergence theorem that in the limit $\eps\to0$,
\begin{equation}\label{eqn:mu_1_limit}
M_{1,0}(z,r,x)=
    \begin{cases}
        e^{-\frac{k_0^2}{8}R(0)z} \,\sff (r),\quad &\eta(\eps)=1\\
        0,\quad &\eta(\eps)\to 0\,.
    \end{cases}
\end{equation}
The moment $M_{0,1}=M_{1,0}^*$ is obtained by complex conjugation.

When the incident beam is of the form \eqref{eq:u0plane}, we still obtain that $M_{1,0}(z,r,x)=0$ when $\eta\ll 1$. However, when $\eta=1$, 
\begin{eqnarray}
    m^\eps_{1,0}(z,r,x) 
    =e^{-\frac{k_0^2}{8}R(0)z}\sum\limits_{\sm=1}^\sM
    \int_{\Rm^d} \hat{\sff}_\sm(\xi)e^{-\frac{iz}{2k_0\eps}|\eps^\beta\xi+\sk_\sm|^2}e^{i\xi\cdot (r+\eps^\beta x)}e^{i\sk_\sm\cdot (\eps^{-\beta}r+x)} \frac{\mathrm{d}\xi}{(2\pi)^d},
\end{eqnarray}
which remains highly oscillatory as $\eps\ll 1$. Nevertheless, as will be shown in section \ref{subsec:second_moment}, the centered second moment still converges in the limit $\eps\to0$.
\subsection{Second moment}\label{subsec:second_moment}
Let us now consider the $p=q=1$ moment $\mu_{1,1}^\eps(z,x,y)=\mathbb{E}[u^\eps(z,x){u^{\eps*}}(z,y)]$. Using It\^o's formula for the Hilbert space-valued process in \eqref{eq:uSDEd}, we obtain
\begin{equation}
    \partial_z\mu_{1,1}^\eps=\frac{i\eta}{2k_0\eps}(\Delta_x-\Delta_y)\mu_{1,1}^\eps+\frac{k_0^2}{4\eta^2} \big( R(x-y)-R(0) \big)\mu_{1,1}^\eps,\quad \mu_{1,1}^\eps(0,x,y)=u_0^\eps(x){u_0^{\eps*}}(y)\,.
\end{equation}
Introduce the variables $\tau=y-x, r=\frac{1}{2}(x+y)$ so that we obtain for $\breve{\mu}_{1,1}^\eps(z,r,\tau)=\mu_{1,1}^\eps(z,x,y)$ 
that
\begin{equation*}
    \partial_z\breve{\mu}_{1,1}^\eps=-\frac{i\eta}{k_0\eps}\partial_r\cdot\partial_\tau\breve{\mu}_{1,1}^\eps+\frac{k_0^2}{4\eta^2}\big( R(\tau)-R(0)\big)\breve{\mu}_{1,1}^\eps\,.
\end{equation*}
Defining the Fourier transform w.r.t. $r$ as 
$\check{\mu}_{1,1}^\eps(z,\xi,\tau)=\int_{\Rm^d}\breve{\mu}_{1,1}^\eps(z,r,\tau)e^{-i\xi\cdot r}\mathrm{d}r$, we obtain
\begin{equation*}
    \big( \partial_z-\frac{\eta}{k_0\eps}\xi\cdot\partial_\tau \big)\check{\mu}_{1,1}^\eps=\frac{k_0^2}{4\eta^2} \big( R(\tau)-R(0) \big)\check{\mu}_{1,1}^\eps\,.
\end{equation*}
This can be solved explicitly as
\begin{equation*}
\check{\mu}_{1,1}^\eps(z,\xi,\tau)=\check{\mu}_{1,1}^\eps(0,\xi,\tau+\frac{\eta\xi z}{k_0\eps}) \ e^{\frac{k_0^2z}{4\eta^2}\int_{0}^1Q(\tau+\frac{\eta\xi s z}{k_0\eps})\mathrm{d}s}\, = \check{\mu}_{1,1}^\eps(0,\xi,\tau+\frac{\eta\xi z}{k_0\eps}) \ \mQ(\tau,\frac{\eta \xi}\eps),
\end{equation*}
where we have defined terms that we will be using repeatedly as:
\begin{equation}\label{eqn:Q}
    0\geq Q(x):=R(x)-R(0)\qquad \mbox{ and } \qquad \mQ(\tau,\alpha) := \exp \Big( \frac{k_0^2z}{4\eta^2}\int_0^1 Q\big(\tau+\alpha\frac{sz}{k_0}\big) ds\Big).
\end{equation}
Transforming the above equality back to spatial coordinates gives 
\begin{equation}\label{eqn:mu_2}
\breve\mu_{1,1}^\eps(z,r,\tau)=\int_{\Rm^{2d}} e^{i\xi\cdot(r-r')}\breve{\mu}_{1,1}^\eps(z=0,r',\tau+\frac{\eta\xi z}{k_0\eps}) \, \mQ(\tau,\frac{\eta\xi}{\eps})\,  \frac{\mathrm{d}\xi\mathrm{d}r'}{(2\pi)^d},
\end{equation}
where 
\begin{equation*}
    \breve{\mu}_{1,1}^\eps(z=0,r,\tau)=u_0^\eps\big(r-\frac{\tau}{2}\big){u_0^{\eps*}}\big(r+\frac{\tau}{2}\big)=\int\limits_{\mathbb{R}^{2d}}\hat{u}_0^\eps(\xi){\hat{u}_0^{\eps*}}(\zeta)e^{ir\cdot(\xi-\zeta)}e^{-i\frac{\tau}{2}\cdot(\xi+\zeta)}\frac{\mathrm{d}\xi\mathrm{d}\zeta}{(2\pi)^{2d}}\,.
\end{equation*}
Using \eqref{eq:mupq}, the corresponding moment of $\phi^\eps$, $m^\eps_{1,1}(z,r,x,y)= \mu^\eps_{1,1}(z,\eps^{-\beta} r+\eta x, \eps^{-\beta}r+\eta y)$ is given by
\begin{equation}\label{eq:meps11}
    m^\eps_{1,1}(z,r,x,y)=\int\limits_{\mathbb{R}^{2d}} e^{i\eps^{-\beta}r\cdot(\xi-\zeta)}e^{i\eta(\xi\cdot x-\zeta\cdot y)}e^{-\frac{iz\eta}{2k_0\eps}(|\xi|^2-|\zeta|^2)}\hat{u}_0^\eps(\xi){\hat{u}_0^{\eps*}}(\zeta)\mQ\big(\eta(y-x),\frac{\eta(\xi-\zeta)}{\eps}\big)\frac{\mathrm{d}\xi\mathrm{d}\zeta}{(2\pi)^{2d}}\,.
\end{equation}
When the source is a superposition of plane wave modulated sources~\eqref{eq:u0plane}, we have that
\begin{equation}\label{eqn:m_11_plane_wave}
    m^\eps_{1,1}(z,r,x,y)=\sum\limits_{\sm=1}^\sM\sum\limits_{\sn=1}^\sM\mathcal{I}_{\sm,\sn}^\eps(z,r,x,y)\,,
\end{equation}
\begin{equation}\label{eqn:I_mn} 
  \begin{aligned}
    \mathcal{I}_{\sm,\sn}^\eps(z,r,x,y)
    = & \int_{\mathbb{R}^{2d}}
    \hat{\sff}_\sm(\xi)\, \hat{\sff}_{\sn}^*(\zeta)
    e^{i(\eps^\beta\xi+\sk_\sm)\cdot(\eps^{-\beta}r+\eta x)}
    e^{-i(\eps^\beta\zeta+\sk_{\sn})\cdot(\eps^{-\beta}r+\eta y)}
    \\ &\quad \times  
    e^{-\frac{iz\eta}{2k_0\eps}(|\eps^\beta\xi+\sk_\sm|^2-|\eps^\beta\zeta+\sk_{\sn}|^2)} \mQ\big(\eta(y-x), \frac\eta\eps(\eps^\beta(\xi-\zeta) + \sk_\sm-\sk_{\sn})\big)
    \frac{\mathrm{d}\xi \mathrm{d}\zeta}{(2\pi)^{2d}}\,,
\end{aligned}
\end{equation}
where we have used a change of variables $\eps^{-\beta}(\xi-\sk_{\sm})\to\xi, \eps^{-\beta}(\zeta-\sk_{\sn})\to\zeta$.  In particular, when $u_0^\eps$ is of the form~\eqref{eq:u0}, we have
\begin{equation}\label{eqn:meps11_f0}
    m^\eps_{1,1}(z,r,x,y)=\int\limits_{\mathbb{R}^{2d}} 
    e^{ir\cdot(\xi-\zeta)}
    e^{i\eta\eps^\beta(\xi\cdot x-\zeta\cdot y)}
    e^{-\frac{iz\eta\eps^{2\beta-1}}{2k_0}(|\xi|^2-|\zeta|^2)}\hat{\sff}(\xi)\hat{\sff}^\ast(\zeta)\mQ\big(\eta(y-x),\eta\eps^{\beta-1}(\xi-\zeta)\big)\frac{\mathrm{d}\xi\mathrm{d}\zeta}{(2\pi)^{2d}}\,.
\end{equation}
We now analyze the limit of $m^\eps_{1,1}$ as $\eps\to0$. Depending on $\eta(\eps)$ and the form of the incident beam at $z=0$, we need to consider several scenarios.

\paragraph{Kinetic regime $\eta(\eps)=1$.}
For incident beams of the form \eqref{eq:u0}, we now show that the limit $M_{1,1}(z,r,x,y):=\lim_{\eps\to 0}m^\eps_{1,1}(z,r,x,y)$ exists.

For $\beta>1$, as $\lim_{\eps\to 0}\int_{0}^1Q(y-x+\frac{\eps^{\beta-1}(\xi-\zeta) s z}{k_0})\mathrm{d}s=Q(y-x)$ pointwise by assumption on $R(x)$, we obtain, for instance by the Lebesgue dominated theorem, that 
\begin{equation}\label{eqn:mu_2_kinetic_beta>1}
      M_{1,1}(z,r,x,y)=|\sff(r)|^2\mQ(y-x,0) = |\sff(r)|^2e^{\frac{k_0^2z}{4}Q(y-x)}\,.
\end{equation}
Similarly, when $\beta=1$, we obtain
\begin{equation}\label{eqn:mu_2_kinetic_beta=1}
    M_{1,1}(z,r,x,y)=\int_{\Rm^{2d}} e^{i\xi\cdot(r-r')}|\sff(r')|^2
    \mQ(y-x,\xi)
     \frac{\mathrm{d}\xi\mathrm{d}r'}{(2\pi)^d}\,.
\end{equation}

For incident beams as defined in \eqref{eq:u0plane} and as seen from \eqref{eqn:m_11_plane_wave}-\eqref{eqn:I_mn}, several terms remain highly oscillatory as $\eps\to 0$.    However, the central second moment
\begin{eqnarray*}
        \tilde{m}^\eps_{1,1}(z,r,x,y)=m^\eps_{1,1}(z,r,x,y)-m^\eps_{1,0}(z,r,x)m^\eps_{0,1}(z,r,y)
\end{eqnarray*}
converges as $\eps\to 0$. To see this, note that
\begin{equation*}
        \tilde{m}^\eps_{1,1}(z,r,x,y)=\sum_{\sm=1}^\sM\sum_{\sn=1}^\sM{\tilde{\mathcal{I}}}^\eps_{\sm,\sn}(z,r,x,y)\,,
\end{equation*}
\begin{equation*}
        \begin{aligned}
         {\tilde{\mathcal{I}}}^\eps_{\sm,\sn}(z,r,x,y)&=\int_{\mathbb{R}^{2d}}\hat{\sff}_\sm(\xi)\, \hat{\sff}_{\sn}^*(\zeta)
    e^{i(\eps^\beta\xi+\sk_\sm)\cdot(\eps^{-\beta}r+ x)}
    e^{-i(\eps^\beta\zeta+\sk_{\sn})\cdot(\eps^{-\beta}r+ y)}e^{-\frac{iz}{2k_0\eps}(|\eps^\beta\xi+\sk_\sm|^2-|\eps^\beta\zeta+\sk_{\sn}|^2)}
    \\ &\quad \times  
    \big(\mQ\big(y-x,\eps^{\beta-1}(\xi-\zeta)+\frac{\sk_\sm-\sk_\sn}{\eps}\big)
    -e^{-\frac{k_0^2z}{4}R(0)}\big)
     \frac{\mathrm{d}\xi \mathrm{d}\zeta}{(2\pi)^{2d}}\,.   
        \end{aligned}
\end{equation*}
When $\sm=\sn$, we verify that pointwise, $\lim_{\eps\to 0}{\tilde{\mathcal{I}}}^\eps_{\sm,\sm}(z,r,x,y)={\tilde{\mathcal{I}}}_{\sm,\sm}(z,r,x,y)$ with
\begin{equation}\label{eqn:plane_wave_mod}
     {\tilde{\mathcal{I}}}_{\sm,\sm}(z,r,x,y):=\begin{cases}
        |\sff_\sm(r)|^2e^{-i\sk_\sm\cdot(y-x)}(e^{\frac{k_0^2z}{4}Q(y-x)}-e^{-\frac{k_0^2z}{4}R(0)}), &\beta>1,\\
        e^{-i\sk_\sm\cdot(y-x)}\dint_{\Rm^{2d}} \!\!\! e^{i\xi\cdot(r-r'-\frac{\sk_\sm z}{k_0})}|\sff_\sm(r')|^2
        \big(
        \mQ(y-x,\xi)
        -e^{-\frac{k_0^2z}{4}R(0)}\big) \frac{\mathrm{d}\xi\mathrm{d}r'}{(2\pi)^d}, &\beta=1.
    \end{cases}
\end{equation}
When $\sm\neq \sn$,
\begin{eqnarray*}
        |\tilde{\mathcal{I}}^\eps_{\sm,\sn}|(z,r,x,y)\le \int_{\mathbb{R}^{2d}}|\hat{\sff}_\sm(\xi)||\hat{\sff}_{\sn}^*(\zeta)| \Big|
        \mQ(y-x,\eps^{\beta-1}(\xi-\zeta)+\frac{\sk_\sm-\sk_{\sn}}{\eps})
        -e^{-\frac{k_0^2z}{4}R(0)}\Big|
       \frac{\mathrm{d}\xi \mathrm{d}\zeta}{(2\pi)^{2d}}\xrightarrow{\eps\to 0} 0\,,
\end{eqnarray*}
as for each $(\xi,\zeta,z)$ fixed and $\sk_\sm\not= \sk_{\sn}$ (and then using Lebesgue's dominated convergence), we have 
\[ \lim_{\eps\to 0}\int_{0}^1\! Q(\tau+\frac{\eps^{\beta-1}(\xi-\zeta) s z}{k_0}+\frac{(\sk_\sm-\sk_{\sn}) s z}{k_0\eps})\mathrm{d}s=  \tb{-} R(0),\quad \mbox{ or } \quad \lim_{|\tau'|\to\infty} \mQ(\tau,\tau') = e^{-\frac{k_0^2z}{4}R(0)}.\] 
This means that only the terms with $\sm=\sn$ contribute in the limit $\eps\to 0$, and we have 
\begin{eqnarray}\label{eqn:mu_2_plane_wave_super_pos}
         \widetilde{M}_{1,1}(z,r,x,y)=\lim_{\eps\to 0}\tilde{m}^\eps_{1,1}(z,r,x,y)=\sum\limits_{\sm=1}^\sM\tilde{\mathcal{I}}_{\sm,\sm}(z,r,x,y)\,,
\end{eqnarray}
where $\tilde{\mathcal{I}}_{\sm,\sm}$ is given by~\eqref{eqn:plane_wave_mod}. The centered second moment due to a broad incident beam $u_0^\eps(x)=\sff(\eps^\beta x)$ is also given by~\eqref{eqn:mu_2_plane_wave_super_pos}  (with $\sk_\sm=0$ identically). This result also shows that when $\eps\ll 1$, cross-correlations due to plane wave modulated incident beams with different transverse wavenumbers become negligible, and the limiting centered second moment is simply the sum of centered second moments due to individual sources. 
\paragraph{Diffusive regime with $\eta^{-1}=\ln\ln\eps^{-1}$.}
Since $R$ is sufficiently smooth and maximum at $x=0$, we write $Q(x)=\frac{1}{2}x^\tps\Gamma x+o(|x|^2)$ with $\Gamma=\nabla^2R(0)$. 
When $\beta>1$, this gives us the limit 
  \begin{equation*}
      \lim_{\eps\to 0}\frac{k_0^2z}{4\eta^2}\int_{0}^1Q(\eta(y-x)+\frac{\eps^{\beta-1}\eta(\xi-\zeta) s z}{k_0})\mathrm{d}s=\frac{k_0^2z}{8}(y-x)^\tps\Gamma(y-x)\,,
  \end{equation*}
and hence for sources of the form \eqref{eq:u0}, setting $\eta^{-1}=\ln\ln\eps^{-1}$ in~\eqref{eqn:meps11_f0} gives the limit
\begin{equation}\label{eqn:mu_2_diffusive_beta>1}
      M_{1,1}(z,r,x,y)=|\,\sff(r)|^2e^{\frac{k_0^2z}{8}(y-x)^\tps\Gamma(y-x)}\,.
\end{equation}
This shows that the correlation function of the limiting random vector $\phi$ decays exponentially as $z|x-y|^2$ increases.

When $\beta=1$, we obtain instead
  \begin{equation*}
       \lim_{\eps\to 0}\frac{k_0^2z}{4\eta^2}\int_{0}^1Q(\eta(y-x)+\frac{\eps^{\beta-1}\eta(\xi-\zeta) s z}{k_0})\mathrm{d}s=\frac{k_0^2z}{8}\int_{0}^1\big(y-x+\frac{(\xi-\zeta) s z}{k_0}\big)^\tps\Gamma\big(y-x+\frac{(\xi-\zeta) s z}{k_0}\big)\mathrm{d}s,
  \end{equation*}
so that
\begin{equation}\label{eqn:mu_2_beta_10}
    \begin{array}{rcl}
       M_{1,1}(z,r,x,y)&=&\dint_{\Rm^{2d}} e^{i\xi\cdot(r-r')}|\,\sff(r')|^2\exp\Big(\frac{k_0^2z}{8}\int_{0}^1\big(y-x+\frac{\xi s z}{k_0}\big)^\tps \Gamma\big(y-x+\frac{\xi s z}{k_0}\big)\mathrm{d}s\Big) \dfrac{\mathrm{d}\xi\mathrm{d}r'}{(2\pi)^d}\,
\\[4mm]
     &=& e^{\frac{k_0^2}{32} z \tau^\tps\Gamma\tau}e^{i\xi_0\cdot r} 
      \dint_{\Rm^{d}}  \Big( \dint_{\Rm^d} e^{i\xi\cdot(r-r')} e^{\frac{z^3}{24}\xi^\tps\Gamma \xi} \dfrac{d\xi}{(2\pi)^d}  \Big) [e^{-i\xi_0\cdot r'}|\sff(r')|^2] dr'
       \end{array}
\end{equation}
with $\tau=y-x$ and $\xi_0=-\frac{3k_0\tau}{2z}$ after evaluating the integrals in $s$ and completing squares in $(\xi-\xi_0)$. Introduce the Green's function of the following parabolic equation
\begin{equation}\label{eq:GFM11}
    (\partial_t + \frac1{24} \nabla_r\cdot \Gamma \nabla_r) G(t,r)=0,\quad G(0,r)=\delta_0(r).
\end{equation}
Then we deduce that 
\begin{equation}\label{eqn:mu_2_diffusive_beta=1}
    M_{1,1}(z,r,x,y)=  e^{\frac{k_0^2}{32} z (y-x)^\tps\Gamma(y-x)}e^{-i\frac{3k_0}{2z}(y-x)\cdot r} \  G(z^3, r) * [e^{i\frac{3k_0}{2z}(y-x)\cdot r} |\sff|^2(r)].
\end{equation}
Here, we write $f(r)*g(r)$ for the convolution $(f*g)(r)$.
This expression holds for $z>0$ while $M_{1,1}(0,r,x,y)=|\,\sff(r)|^2$. We thus obtain that $M_{1,1}$ is decaying exponentially as $z|x-y|^2$ increases and that it is otherwise given by {\color{black}an unusual diffusion propagator $G(t,\cdot)$ with diffusion tensor $-\Gamma/24$ evaluated at $t=z^3$} and applied to a $z-$dependent function. Note that for large $z$ such that $z|y-x|^2$ remains of order ${\mathcal O}(1)$, the phase $e^{-i\frac{3k_0}{2z}(y-x)\cdot r}$ is increasingly negligible.

When $x=y$, we obtain that $M_{1,1}(z,r,x,x)=I_2(z,r)$ the mean intensity solution of the following diffusion equation
\begin{equation}\label{eqn:I2_diffusion}
 \partial_{z^3}I_2+\frac1{24}\nabla_r\cdot\Gamma\nabla_rI_2=0, \qquad I_2(0,r)= \left\{\begin{array}{cl} |\sff(r)|^2\ & \ \mbox{ when} \,u_0^\eps \, \mbox{ is given by } \, \eqref{eq:u0} \\  \dsum_{\sm=1}^\sM|\sff_\sm(r)|^2\ &  \ \mbox{ when} \,u_0^\eps \, \mbox{ is given by } \, \eqref{eq:u0plane}.\end{array} \right.
\end{equation}
We recall that $\Gamma$ is negative-definite.
We also verify that when the incident beam is a superposition of modulated plane waves \eqref{eq:u0plane},
\begin{eqnarray}\label{eq:m11diffplane}
    M_{1,1}(z,r,x,y)=\sum\limits_{m=1}^M\lim_{\eps\to 0}\mathcal{I}^\eps_{\sm,\sm}(z,r,x,y)\,,
\end{eqnarray}
with $\lim_{\eps\to 0}\mathcal{I}^\eps_{\sm,\sm}$ is given by ~\eqref{eqn:mu_2_diffusive_beta>1} when $\beta>1$ and~\eqref{eqn:mu_2_diffusive_beta=1} when $\beta=1$ with $\sff\equiv \sff_\sm$. 
 Indeed, for $\sm\not=\sn$, we obtain as in the kinetic regime that $|\tilde{\mathcal{I}}_{\sm,\sn}|\to0$. Since the first moments also converge to $0$, we deduce that $\mathcal{I}_{\sm,\sn}=0$ in the limit $\eps\to0$. It remains to realize that $\sk_\sm$ does not appear directly in the limit $\eta\to0$ to conclude that \eqref{eq:m11diffplane} holds.

We summarize the above results as follows.
\begin{lemma}[First and second moment limits]\label{lem:limmts} 
  Let $m^\eps_{1,0}(z,r,x)=\E \phi^\eps(z,r,x)$, $m^\eps_{0,1}(z,r,x)=m^{\eps*}_{1,0}(z,r,x)$, $m^\eps_{1,1}(z,r,x,y) = \E \phi^\eps(z,r,x) \phi^{\eps*}(z,r,y)$ and $\tilde m^\eps_{1,1}(z,r,x,y)=m^\eps_{1,1}(z,r,x,y)-m^\eps_{1,0}(z,r,x)m^{\eps*}_{1,0}(z,r,y)$.

  For an incident condition \eqref{eq:u0}, $m^\eps_{1,0}$ converges to $M_{1,0}$ given in \eqref{eqn:mu_1_limit} in the kinetic and diffusive regimes. $m^\eps_{1,1}$ converges to $M_{1,1}$, which in the kinetic regime is given in \eqref{eqn:mu_2_kinetic_beta>1} when $\beta>1$ and in \eqref{eqn:mu_2_kinetic_beta=1} when $\beta=1$. In the diffusive regime, $M_{1,1}$ is given by \eqref{eqn:mu_2_diffusive_beta>1} when $\beta>1$ and \eqref{eqn:mu_2_diffusive_beta=1} when $\beta=1$.
  
  For an incident condition \eqref{eq:u0plane}, $\tilde m^\eps_{1,1}$ converges to $\widetilde M_{1,1}$ given in \eqref{eqn:plane_wave_mod}-\eqref{eqn:mu_2_plane_wave_super_pos} in the kinetic regime. In the diffusive regime, $m^\eps_{1,0}$ converges to $0$ while $m^\eps_{1,1}$ converges to $M_{1,1}$ given by \eqref{eqn:mu_2_diffusive_beta>1},\eqref{eq:m11diffplane} when $\beta>1$ and by \eqref{eqn:mu_2_diffusive_beta=1},\eqref{eq:m11diffplane}  when $\beta=1$.
\end{lemma}
All convergence results have been obtained point-wise. More uniform estimates may be obtained under mild assumptions on $R(x)$ and the incident beam conditions. Such estimates will be obtained in the proof of Theorem \ref{thm:tightness}.

\section{Higher order moments and the Gaussian summation rule}\label{sec:highermoments}
This section presents the main technical part of this paper. For $X=(x_1,\cdots,x_p)$ and $Y=(y_1,\cdots,y_q)$, recall that the $p+q$th moment of $u^\eps$ is defined in \eqref{eq:mupq}. Using It\^o's formula once more for Hilbert-valued processes \cite{garnier2016fourth}, we deduce from \eqref{eq:uSDEd} that
\begin{equation}\label{eq:mupqpde}
\begin{aligned}
        \partial_z\mu_{p,q}^\eps&= {\mathcal L}^\eps_{p,q}\mu_{p,q}^\eps,\qquad  {\mathcal L}^\eps_{p,q} :=  \frac{i\eta}{2k_0\eps}\big(\sum_{j=1}^p\Delta_{x_j}-\sum_{l=1}^q\Delta_{y_j}\big)+\frac{k_0^2}{4\eta^2}\mathcal{U}_{p,q}(X,Y)\\
        \mu_{p,q}^\eps(0,X,Y)&=\prod_{j=1}^pu_0^\eps(x_j)\prod_{l=1}^q{u_0^{\eps*}}(y_l)\,.
\end{aligned}
\end{equation}
The potential $\mathcal{U}_{p,q}$ is given by
\begin{equation}\label{eq:Upot}
    \begin{aligned}
        \mathcal{U}_{p,q}(X,Y)&=\sum_{j=1}^p\sum_{l=1}^qR(x_j-y_l)-\sum_{1\le j< j'\le p}R(x_j-x_{j'})-\sum_{1\le l< l'\le q}R(y_l-y_{l'})-\frac{p+q}{2}R(0)\,.
    \end{aligned}
\end{equation}
The limiting behavior of these moments is analyzed in the Fourier variables. Define the (complex symmetrized) Fourier transform of $\mu^\eps_{p,q}$ as
\begin{equation}\label{eq:hatmupq}
    \hat{\mu}^\eps_{p,q}(z,v)=\int_{\Rm^{(p+q)d}} \mu^\eps_{p,q}(z,X,Y)e^{-i(\sum_{j=1}^p\xi_j\cdot x_j-\sum_{l=1}^q\zeta_l\cdot y_l)}\mathrm{d}x_1\cdots\mathrm{d}x_p\mathrm{d}y_1\cdots\mathrm{d}y_q\,,
\end{equation}
where $v=(\xi_1,\cdots,\xi_p,\zeta_1,\cdots,\zeta_q)$ denotes the dual variables to $(X,Y)$. Then $\hat{\mu}_{p,q}^\eps$ satisfies 
\begin{equation}\label{eq:mupqF}   
\begin{aligned}
        {\partial_z\hat{\mu}^\eps_{p,q}}&=-\frac{i\eta}{2k_0\eps}\Big(\sum_{j=1}^p|\xi_j|^2-\sum_{l=1}^q|\zeta_l|^2\Big)\hat{\mu}^\eps_{p,q}+\frac{k_0^2}{4\eta^2}\int_{\Rm^d}\hat{R}(k)\Big[\sum_{j=1}^p\sum_{l=1}^q\hat{\mu}^\eps_{p,q}(\xi_j-k,\zeta_l-k)\\
    & \ \ \ \ -\sum_{1\le j< j'\le p} \!\!\! \hat{\mu}^\eps_{p,q}(\xi_j-k,\xi_{j'}+k)-\sum_{1\le l< l'\le q}\hat{\mu}^\eps_{p,q}(\zeta_l-k,\zeta_{l'}+k)-\frac{p+q}{2}\hat{\mu}^\eps_{p,q}\Big] \dfrac{\mathrm{d}k}{(2\pi)^d}\,,\\
        \hat{\mu}^\eps_{p,q}(0,v)&=\prod_{j=1}^p\hat{u}^\eps_0(\xi_j)\prod_{l=1}^q\hat{u}^{\eps\ast}_0(\zeta_l)\,.
    \end{aligned}
\end{equation}
\subsection{Moments of compensated wave field and integro-differential equations}
It is convenient to recast \eqref{eq:mupqF} as an integro-differential equation with a highly oscillatory kernel. We first need to introduce some notation.
Let $\mathbb{I}_d$ be the $d\times d$ identity matrix and $\Lambda_0$ denote a $(p+q)\times (p+q)$ diagonal matrix. We define the $d(p+q)\times d(p+q)$ matrix $\Theta$ as
\begin{equation*}
    \Theta=\Lambda_0\otimes \mathbb{I}_d\,,\quad  \{\Lambda_0\}_{m,m}=\begin{cases}
        1,\quad m=1,\cdots p\\
        -1,\quad m=p+1,\cdots, q\,.
    \end{cases}
\end{equation*}
This makes $\Theta$ a block diagonal matrix of $p$ ones followed by $q$ negative ones. We also define a number of $d(p+q)\times d$ matrices as follows. Let $\lambda_{1}, \lambda_2$ below denote $(p+q)\times 1$ vectors. For $j=1,\cdots,p, l=1,\cdots,q$, we define $A_{j,l}$ as:
\begin{equation*}
    A_{j,l}=\lambda_1\otimes \mathbb{I}_d\,,\qquad   \{\lambda_1\}_m=\begin{cases}
        1,\quad m=j, p+l\\
        0,\quad \text{otherwise}\,.
\end{cases}
\end{equation*}
In other words, blocks $j$ and $p+l$ of $A_{j,l}$ are identity and the rest vanish. Finally, for $1\le j< j'\le p$ and $p+1\le j< j'\le p+q$, define $B_{j,j'}$ as:
\begin{equation*}
    B_{j,j'}=\lambda_2\otimes \mathbb{I}_d\,,\qquad \{\lambda_2\}_m=\begin{cases}
        1,\quad m=j\\
        -1,\quad m=j'\\
        0,\quad \text{otherwise}\,.
\end{cases}
\end{equation*}
This means that the $j$th block of $B_{j,j'}$ is $\mathbb{I}_d$ and the $j'$th block is $-\mathbb{I}_d$ with the rest vanishing.  Under this operator formulation,~\eqref{eq:mupqF} takes the form
\begin{equation*}
\begin{aligned}
    \partial_z\hat{\mu}_{p,q}^\eps&=-\frac{i\eta}{2\eps k_0}v^\tps\Theta v\hat{\mu}_{p,q}^\eps+\frac{k_0^2}{4\eta^2}\int_{\Rm^d}\hat{R}(k)\Big[\sum_{j=1}^p\sum_{l=1}^q\hat{\mu}^\eps_{p,q}(v-A_{j,l}k)\\
    &-\sum_{1\le j< j'\le p}\hat{\mu}^\eps_{p,q}(v-B_{j,j'}k)-\sum_{p+1\le l< l'\le p+q}\hat{\mu}^\eps_{p,q}(v-B_{l,l'}k)-\frac{p+q}{2}\hat{\mu}^\eps_{p,q}\Big] \dfrac{\mathrm{d}k}{(2\pi)^d}\,.
\end{aligned}
   \end{equation*}

Next we introduce the phase compensated wave field moments $\psi_{p,q}^\eps$ such that
\begin{equation}\label{eqn:phase_correc}
    \hat{\mu}_{p,q}^\eps=\Pi^\eps_{p,q}(z,v)\psi_{p,q}^\eps\,,
\end{equation}
where $\Pi^\eps_{p,q}$ captures propagation in free space and is  given explicitly by
\begin{equation}\label{eq:Phipq}
    \Pi^\eps_{p,q}(z,v)=e^{-\frac{iz\eta}{2k_0\eps}(\sum_{j=1}^p|\xi_j|^2-\sum_{l=1}^q|\zeta_l|^2)} = e^{-\frac{iz\eta}{2k_0\eps} v^\tps\Theta v }.
\end{equation}
Using~\eqref{eq:mupqF} and the operator formulation above allows us to obtain that
\begin{equation*}
    \begin{aligned}
      {\partial_z\psi_{p,q}^\eps}&=\frac{k_0^2}{4\eta^2}e^{\frac{iz\eta}{2k_0\eps}v^\tps\Theta v}\int_{\Rm^d}\hat{R}(k)\Big[\sum_{j=1}^p\sum_{l=1}^q\psi_{p,q}^\eps(v-A_{j,l}k)e^{-\frac{iz\eta}{2k_0\eps}(v-A_{j,l}k)^\tps\Theta(v-A_{j,l}k)}\\
    &-\sum_{1\le j< j'\le p}\psi_{p,q}^\eps(v-B_{j,j'}k)e^{-\frac{iz\eta}{2k_0\eps}(v-B_{j,j'}k)^\tps\Theta(v-B_{j,j'}k)}\\
    &-\sum_{p+1\le l< l'\le p+q}\psi_{p,q}^\eps(v-B_{l,l'}k)e^{-\frac{iz\eta}{2k_0\eps}(v-B_{l,l'}k)^\tps\Theta(v-B_{l,l'}k)}-\frac{p+q}{2}\psi_{p,q}^\eps(v)e^{-\frac{iz\eta}{2k_0\eps}v^\tps\Theta v}\Big]\frac{\mathrm{d}k} {(2\pi)^d}\,,\\
       \psi_{p,q}^\eps(0,v)&=\hat{\mu}_{p,q}(0,v)\,.   
    \end{aligned}
\end{equation*}
This translates into the evolution equation
\begin{equation}\label{eqn:psi_evolution}
   \partial_z\psi_{p,q}^\eps=L^\eps_{p,q}\psi_{p,q}^\eps,\qquad \psi_{p,q}^\eps(0,v)=\hat{\mu}_{p,q}(0,v)\,,
\end{equation}
where the operator $L^\eps_{p,q}$ is given by
\begin{equation}\label{eqn:L_expansion}
     {L}^{\eps}_{p,q}=\frac{p+q}{2}L_{\eta}+\sum_{j=1}^p\sum_{l=1}^q{L}^{\eps,1}_{j,l}+\sum_{1\le j< j'\le p}L^{\eps ,2}_{j,j'}+\sum_{p+1\le l< l'\le p+q}L^{\eps ,2}_{l,l'} = : \frac{p+q}{2}L_{\eta} + L^{\eps,1} + L^{\eps,2}\,,
 \end{equation}
  with obvious notation for $L^{\eps,\kappa}$ for $\kappa=1,2$ and  
 \begin{equation}\label{eqn:L_pq_def}
     \begin{aligned}
          L_{\eta}&=-\frac{k_0^2}{4\eta^2}\int_{\Rm^d}\hat{R}(k) \frac{\mathrm{d}k} {(2\pi)^d}=-\frac{C_0}{\eta^2},\qquad C_0=\frac{k_0^2}{4}R(0)\,,\\
         {L}^{\eps ,1}_{j,l}\rho(v)&=\frac{k_0^2}{4\eta^2}\int_{\Rm^d}\hat{R}(k)\rho(v-A_{j,l}k)e^{\frac{iz\eta}{k_0\eps}k^\tps A_{j,l}^\tps\Theta v}\frac{\mathrm{d}k} {(2\pi)^d}\,,\\
         L^{\eps ,2}_{j,j'}\rho(v)&=-\frac{k_0^2}{4\eta^2}\int_{\Rm^d}\hat{R}(k)\rho(v-B_{j,j'}k)e^{\frac{iz\eta}{2k_0\eps}(2k^\tps B_{j,j'}^\tps\Theta v-k^\tps B_{j,j'}^\tps\Theta B_{j,j'}k)}\frac{\mathrm{d}k} {(2\pi)^d}\,.
     \end{aligned}
 \end{equation}
Here we have used the fact that 
 \begin{equation*}
     \begin{aligned}
         (v-A_{j,l}k)^\tps\Theta(v-A_{j,l}k)&=v^\tps\Theta v-2k^\tps A_{j,l}^\tps\Theta v\, \quad \mbox{ since } \quad A_{j,l}^\tps\Theta A_{j,l}=0.
     \end{aligned}
 \end{equation*}
 Now define $U^\eps(z)$ as the solution operator of \eqref{eqn:psi_evolution}, which solves the evolution equation
 \begin{equation}\label{eqn:sol_operator_evolution}
     \partial_zU^\eps=L^\eps_{p,q}U^\eps,\qquad U^\eps(0)=\mathbb{I}\,.
 \end{equation}
For easier presentation, we will drop the $\eps$ dependence of the operators in~\eqref{eqn:L_pq_def} and~\eqref{eqn:sol_operator_evolution} and retain it only in the solution operator $U^\eps$. 
We recall that $\mathcal{M}_B(\mathbb{R}^{pd+qd})$ denotes the Banach space of finite signed measures on $\mathbb{R}^{pd+qd}$ equipped with the total variation norm $\|\cdot\|$ and $\mathbb{I}$ above denotes identity on that space.
\begin{lemma}\label{lemma:L_pq_bound}
   The operator $L_{p,q}$ is bounded on $\mathcal{M}_B(\mathbb{R}^{pd+qd})$ with operator norm uniform in $z$:
    \begin{equation*}
        \|L_{p,q}\|\le \frac{C_0(p+q)^2}{2\eta^2}\,.
    \end{equation*}
    \begin{proof}
       Note that for any $\rho\in\mathcal{M}_B(\mathbb{R}^{pd+qd})$,
       \begin{equation*}
           \|L_{p,q}\rho\|\le \frac{p+q}{2}\|L_{\eta}\rho\|+\sum_{j=1}^p\sum_{l=1}^q\|L^1_{j,l}\rho\|+\sum_{1\le j< j'\le p}\|L^2_{j,j'}\rho\|+\sum_{p+1\le l< l'\le p+q}\|L^{2}_{l,l'}\rho\|\,,
       \end{equation*}
       where $\|\cdot\|$ denotes the total variation. For $1\le j\le p, 1\le l\le q$,
       \begin{equation*}
           \|L_{j,l}^1\rho\|=\Big\|\frac{k_0^2}{4\eta^2}\int_{\Rm^d} \hat{R}(k)\rho(v-A_{j,l}k)e^{\frac{iz\eta}{k_0\eps}k^\tps A_{j,l}^\tps\Theta v}\frac{\mathrm{d}k} {(2\pi)^d}\Big\|\le \frac{k_0^2}{4\eta^2}\int_{\Rm^d}\hat{R}(k)\|\rho(v-A_{j,l}k)\|\frac{\mathrm{d}k} {(2\pi)^d}=\frac{C_0}{\eta^2}\|\rho\|\,.
       \end{equation*}
       A similar argument applies to $L^2_{j,j'}$ as well, which completes the proof.       
    \end{proof}
\end{lemma}
We now verify commutation properties that are crucial in our derivation. Denote $[A,B]=AB-BA$ the standard commutator.
We observe that 
\begin{equation}\label{eq:commute}
  [L^1_{j,l}(z_1), L^1_{j',l'}(z_2)]=0
\end{equation}
for any non-negative $z_1$ and $z_2$ provided that $j\neq j'$ and  $l\neq l'$. This is a consequence of
\begin{equation*}
    k^\tps A_{j,l}^\tps \Theta v=k\cdot(\xi_j-\zeta_l)\,,
\end{equation*}
so $L^1_{j,l}$ and $L^1_{j',l'}$ involve orthogonal directions when $j\neq j', l\neq l'$.  {\color{black} 
This is used on a number of occasions to show that $[U_1(z),U_2(z)]=0$ when $\partial_z U_k=L_k(z)U_k$ for $k=1,2$ and $[L_1(z_1),L_2(z_2)]=0$.
}
When either one of the indices is repeated, i.e, $j=j'$ or $l=l'$, $L^1_{j,l}(z_1)$ and $L^1_{j',l'}(z_2)$ do not commute necessarily when $z_1\neq z_2$. The effect of such a coupling will be shown to be negligible when $\eps\ll 1$ due to the rapidly oscillatory phases; see Proposition \ref{prop:linear_error} below. 

Similarly, note that
\begin{equation}\label{eqn:quad_k_B}
    k^\tps B_{j,j'}^\tps\Theta v-\frac{1}{2}k^\tps B_{j,j'}^\tps\Theta B_{j,j'}k=\begin{cases}
        k\cdot(\xi_j-\xi_{j'}-k),\quad 1\le j<j'\le p\\
        -k\cdot(\zeta_j-\zeta_{j'}-k),\quad p+1\le j<j'\le p+q\,.       
    \end{cases}
\end{equation}
We will show that the contribution of these terms is small when $\eps\ll 1$ in Proposition \ref{prop:quad_error} below. We may also verify that $[L^\kappa_{j,l}(z), L^{\kappa'}_{j',l'}(z)]=0$ generally at the same value $z$ but this property is not used or useful in our derivation. %
\subsection{Approximation of higher moments in dual variables}
{\color{black} We now introduce the combinatorial notation that is needed in the approximation of the solution operator $U^\eps(z)$ by \eqref{eqn:N_def} below.}
Define $\Lambda=\Lambda_{p,q}$ as the (unordered) set of cardinality $pq$ of (ordered) pairs $\Lambda=\{(j,l),\ 1\le j\le p, 1\le l\le q\}$. We will denote by $\gamma=(\gamma_1,\gamma_2)$ any element in $\Lambda$.

We need to consider partitions of pairs in $\Lambda$ as follows. We denote by $1\leq \kappa\leq K$ for
\[
  K = \sum_{m=1}^{p\wedge q} \binom{p}{m} \binom{q}{m} \, m! 
\]
any choice of $m=m(\kappa)$ indices in $1\leq j\leq p$ followed by any (ordered) choice of $m$ indices in $1\leq l\leq q$. We denote by $\Lambda_{\kappa}$ the (unordered) set of $m=m(\kappa)$  pairs thus selected. For two different pairs $\gamma, \gamma'\in \Lambda_{\kappa}$, we have that $\gamma_j\not=\gamma'_j$ for both $j=1,2$. As a consequence, the operators $L^1_\gamma(z)$ and $L^1_{\gamma'}(z')$ commute thanks to \eqref{eq:commute}.

For each choice of $\Lambda_{\kappa}$, we define $\Lambda_{\kappa}'$ as the $(p-m(\kappa))(q-m(\kappa))$ pairs in $\Lambda$ such that for $\gamma\in \Lambda_{\kappa}$ and $\gamma'\in \Lambda_{\kappa}'$ we always have $\gamma_j\not=\gamma_j'$ for $j=1,2$. For such pairs, the operators $L^1_\gamma(z)$ and $L^1_{\gamma'}(z')$ also commute for all $z,z'$. Note that $\Lambda_{\kappa}'=\emptyset$ when $m=p\wedge q$.

We finally denote by $\bar\Lambda_{\kappa}$ the set of pairs of cardinality $m(\kappa)(p+q-m(\kappa)-1)$ such that 
\[
  \Lambda = \Lambda_{\kappa} + \bar \Lambda_{\kappa} + \Lambda'_{\kappa}.
\]
We observe that $\bar\Lambda_{\kappa}$ is composed of pairs $\gamma'$ such that there is a pair $\gamma\in \Lambda_{\kappa}$ where $\gamma$ and $\gamma'$ share one component but not both. In this case, $L^1_\gamma(z)$ and $L^1_{\gamma'}(z')$ {\em do not} commute for some $(z,z')$.

As an illustration, we obtain for $p=q=3$ that for one of the indices $1\leq \kappa\leq K=9+18+6=33$, then $\Lambda_{\kappa}=\{(1,1),(2,2)\}$ so that $m(\kappa)=2$ while $\Lambda'_{\kappa}=\{(3,3)\}$ and $\bar\Lambda_{\kappa}=\{(1,2),(1,3),(2,1),(2,3),(3,1),(3,2)\}$.

For each pair $\gamma\in\Lambda$, we define the solution operator $U^\eps_\gamma(z)$ as
 \begin{equation}\label{eqn:U_gamma}
       \partial_zU^\eps_{\gamma}=L^{1}_{\gamma}U^\eps_{\gamma},\quad U^\eps_{\gamma}(0)=\mathbb{I}\,.
\end{equation}
We also define $U_\eta(z)$ as either the function $U_\eta(z)=e^{-\frac{C_0z}{\eta^2}}$ or the operator of multiplication by this exponential damping. This operator solves the equation $(\partial_z-L_{\eta})U_\eta=0$ while $U_\eta(0)={\mathbb I}$.

The main objective of this section is to show that $U^\eps$ is well approximated by:
\begin{equation}\label{eqn:N_def}
          N_{p,q}^\eps=U_\eta^{\frac{p+q}{2}}  \sum_{\kappa=1}^K  {\prod_{\gamma\in\Lambda_{\kappa}} (U^{\eps}_\gamma-\mathbb{I})} + U_\eta^{\frac{p+q}{2}}.
\end{equation}
The operator $N_{p,q}^\eps$ describes the evolution of the moments of a complex normal variable as we will see later. Note that by construction, all operators in the above product commute thanks to \eqref{eq:commute} so that ordering is not necessary. The operators $L^\kappa$ for $\kappa=1,2$ denote the sum of operators $L^\kappa_{\cdot,\cdot}$ as in \eqref{eqn:L_expansion}. 

\begin{theorem}\label{thm:Ueps}
      The solution operator to~\eqref{eqn:sol_operator_evolution} admits the decomposition
      \begin{equation}\label{eqn:U=N+E}
          U^\eps=N_{p,q}^\eps+E_{p,q}^\eps\,,
      \end{equation}
     where $N_{p,q}^\eps$ is given in \eqref{eqn:N_def} and  the error $E_{p,q}^\eps$ solves the evolution equation
     \begin{eqnarray}\label{eqn:E} \partial_zE_{p,q}^\eps&=&L_{p,q}E_{p,q}^\eps+\mathcal{E}_{p,q}^{\eps},\qquad E_{p,q}^\eps(0)=0, \\ 
      \label{eqn:error_source}
       \mathcal{E}_{p,q}^{\eps}&=&U_{\eta}^{\frac{p+q}{2}} \sum_{\kappa=1}^K\sum_{\gamma'\in\bar\Lambda_{\kappa}}L_{\gamma'}^1\prod_{\gamma\in \Lambda_{\kappa}}(U^{\eps}_{\gamma}-\mathbb{I}) + L^2N_{p,q}^\eps\,.
     \end{eqnarray}
   Moreover, for $c(p,q)$ a constant independent of $z,\eps$ and $\mathfrak{C}_{1}$ and $\mathfrak{C}_{2}$ defined in \eqref{eq:mc1} and \eqref{eq:mc2}, we have:
   \begin{equation}\label{eqn:E_bound}
   \sup_{0\le z'\le z}\|E_{p,q}^\eps(z')\|\le c(p,q)\langle z\rangle^2\eta^{-6}\Big(\mathfrak{C}_1[\frac{C_2\eps}{\eta},{\hat{\sR}},\hat{R},d]+\mathfrak{C}_2[\frac{k_0\eps}{\eta},\hat{R},d]\Big)e^{\frac{(p+q)^2C_0z}{2\eta^2}} \leq \eps^{\frac 13}
   \end{equation}
   for any $0<\eps \leq \eps_0(p,q,z)$ uniformly bounded on compact sets.
\end{theorem}   
   The bound $\eps^{\frac13}$ can be replaced by $\eps^\alpha$ for any $\alpha<1$ in lateral dimension $d\geq2$ and any $\alpha<\frac12$ in dimension $d=1$. The proof of this theorem concludes with that of Corollary \ref{cor:boundE} below.
   \begin{proof}
       We verify that $N_{p,q}^\eps(0)=\mathbb{I}$ so that $E_{p,q}^\eps(0)=0$.
       Substituting \eqref{eqn:U=N+E} into the equation \eqref{eqn:sol_operator_evolution} gives 
       \begin{equation*}
           \partial_zE_{p,q}^\eps=L_{p,q}E_{p,q}^\eps+\mathcal{E}_{p,q}^{\eps},\quad \mathcal{E}_{p,q}^{\eps}=L_{p,q}N_{p,q}^\eps-\partial_zN_{p,q}^\eps\,.
       \end{equation*}
       Using the definition of $N_{p,q}^\eps$ in~\eqref{eqn:N_def} and the equation for $U^\eps$ in~\eqref{eqn:U_gamma} gives 
       \begin{equation}\label{eqn:N_z_1}
           \begin{aligned}
               \partial_zN_{p,q}^\eps&=\frac{p+q}{2}L_{\eta}N_{p,q}^\eps+U_\eta^{\frac{p+q}{2}}
               \sum_{\kappa=1}^K \sum_{\gamma\in\Lambda_{\kappa}}               
               L^1_\gamma U^\eps_\gamma\prod_{\gamma\neq \gamma'\in\Lambda_{\kappa}}(U^\eps_{\gamma'}-\mathbb{I})\\
               &=\frac{p+q}{2}L_{\eta}N_{p,q}^\eps+U_\eta^{\frac{p+q}{2}}\sum_{\kappa=1}^K \sum_{\gamma\in\Lambda_{\kappa}}
               L^1_\gamma\prod_{\gamma'\in\Lambda_{\kappa}}(U^\eps_{\gamma'}-\mathbb{I})
               +U_\eta^{\frac{p+q}{2}} \sum_{\kappa=1}^K \sum_{\gamma\in\Lambda_{\kappa}}
                 L^1_\gamma\prod_{\gamma\neq \gamma'\in\Lambda_{\kappa}}(U^\eps_{\gamma'}-\mathbb{I})\\
               &=\frac{p+q}{2}L_{\eta}N_{p,q}^\eps+U_\eta^{\frac{p+q}{2}}\sum_{\kappa=1}^K  \sum_{\gamma\in\Lambda_{\kappa}}
               L^1_\gamma\prod_{\gamma'\in\Lambda_{\kappa}}(U^\eps_{\gamma'}-\mathbb{I})
               \\
               &\quad +U_\eta^{\frac{p+q}{2}}\sum_{m(\kappa)\geq2} \sum_{\gamma\in\Lambda_{\kappa}}
               L^1_\gamma \prod_{\gamma\neq \gamma'\in\Lambda_{\kappa}}(U^\eps_{\gamma'}-\mathbb{I})+
               U_\eta^{\frac{p+q}{2}}\sum_{m(\kappa)=1}\sum_{\gamma\in\Lambda_{\kappa}}L^1_{\gamma}.
           \end{aligned}
       \end{equation}
       Here, $\sum_{m(\kappa)\geq2}$ means summation over all $1\leq \kappa\leq K$ such that $m(\kappa)\geq2$.
       Now, applying $L_{p,q}$ on $N_{p,q}^\eps$ gives 
       \begin{equation}\label{eqn:LN_1}
       \begin{aligned}
                      L_{p,q}N_{p,q}^\eps&=\frac{p+q}{2}L_{\eta}N_{p,q}^\eps+L^2N_{p,q}^\eps+U_\eta^{\frac{p+q}{2}} \sum_{\kappa=1}^K 
                       L^1\prod_{\gamma\in\Lambda_{\kappa}}(U_{\gamma}^{\eps}-\mathbb{I})+U_\eta^{\frac{p+q}{2}}L^1.       \end{aligned}
       \end{equation}
       The first and last terms in~\eqref{eqn:N_z_1} cancel with the first and last terms in~\eqref{eqn:LN_1} respectively since by construction of the sets $\Lambda_{\kappa}$, $L^1=\sum_{m(\kappa)=1}\sum_{\gamma\in\Lambda_{\kappa}}L^1_{\gamma}$. Splitting $L^1$ for any fixed value of $1\leq \kappa\leq K$ as 
       \begin{equation*}
           L^1=\sum_{\gamma\in\Lambda}L^1_\gamma=\sum_{\gamma\in\Lambda_{\kappa}}L^1_\gamma+\sum_{\gamma\in\Lambda'_{\kappa}}L^1_\gamma+\sum_{\gamma\in\bar{\Lambda}_k}L^1_\gamma
       \end{equation*}
       and substituting in the third term of~\eqref{eqn:LN_1} gives for $-(\partial_z-L_{p,q})N_{p,q}^\eps=(\partial_z-L_{p,q})E_{p,q}^\eps$ the source term:
       \begin{equation*}
       \begin{aligned}
            \mathcal{E}_{p,q}^{\eps}&=L^2N_{p,q}^\eps+U_\eta^{\frac{p+q}{2}}\sum_{\kappa=1}^{K} 
             \Big( \sum_{\gamma\in\bar\Lambda_{\kappa}} + \sum_{\gamma\in\Lambda'_{\kappa}} \Big)
             L^1_\gamma\prod_{\gamma'\in\Lambda_{\kappa}}(U_{\gamma'}^{\eps}-\mathbb{I})
          -U_\eta^{\frac{p+q}{2}}\sum_{m(\kappa)\geq2} \sum_{\gamma\in\Lambda_{\kappa}}L^1_\gamma\prod_{\gamma\neq \gamma'\in\Lambda_{\kappa}}(U^\eps_{\gamma'}-\mathbb{I})\,.
       \end{aligned}
\end{equation*}
To show that the source $\mathcal{E}_{p,q}^{\eps}$ is given by \eqref{eqn:error_source}, it remains to show that the sums over $\Lambda'_\kappa$ and the last term in the above expressions cancel out, or in other words
\begin{equation*}
           \begin{aligned}
               \sum_{\kappa=1}^{K}  \sum_{\gamma\in\Lambda'_{\kappa}}L^1_\gamma\prod_{\gamma'\in\Lambda_{\kappa}}(U_{\gamma'}^{\eps}-\mathbb{I}) =\sum_{m(\kappa')\geq2} \sum_{\gamma\in\Lambda_{\kappa'}}L^1_\gamma\prod_{\gamma\neq\gamma'\in\Lambda_{\kappa'}}(U_{\gamma'}^{\eps}-\mathbb{I})\,.
           \end{aligned}
\end{equation*}
To see this, observe that pairs in $\Lambda_{\kappa}$ and $\Lambda_{\kappa}'$ are totally disjoint. As a consequence, the operators involved in any of the above products commute. For $\kappa$ fixed and $\gamma\in\Lambda'_{\kappa}$, consider the set of $m(\kappa)+1$ pairs $\gamma \cup  \Lambda_{\kappa}$. There is by construction a unique element $1\leq m(\kappa')\leq K$ such that $\Lambda_{\kappa'}=\gamma \cup  \Lambda_{\kappa}$ with $m(\kappa')\geq2$.  This shows that the sum in the right-hand-side for $m(\kappa)=m\leq p\wedge q-1$ and left-hand sides for $m(\kappa')=m(\kappa)+1$ agree. Since $\Lambda_{\kappa}'=\emptyset$ for $m(\kappa)=p\wedge q$, this concludes the derivation.
 
The source term ${\mathcal E}^\eps_{p,q}$ is written as a sum of terms of the form $L^2N_{p,q}^\eps$ and  $L^1_{\gamma'}(U^\eps_{\gamma}-\mathbb{I}))$ for $\gamma\in\Lambda_{\kappa}$ while $\gamma'\in \bar\Lambda_{\kappa}$. We will show next that each term is small in an appropriate sense and conclude the proof of \eqref{eqn:E_bound}  in Corollary \ref{cor:boundE} below. 
   \end{proof}
  We first state the following bound on $U^{\eps}_{\gamma}$:
  \begin{lemma}\label{lemma:U_gamma_bound}
      The operator norm of the solution operator $U^\eps_\gamma$ to the evolution equation~\eqref{eqn:U_gamma} is bounded as
      \begin{equation}\label{eqn:U_gamma_bound}
          \|U_\gamma^{\eps}(z)\|\le e^{\frac{C_0z}{\eta^2}}\,,
      \end{equation}
      for all $\gamma\in\Lambda$. Moreover,  for every $1\leq \kappa\leq K$ with $m=m(\kappa)$,
      \begin{eqnarray}
          \label{eqn:U_gamma_prod_bound}  \|\prod_{\gamma\in\Lambda_{\kappa}}U_\eta(U^\eps_\gamma-\mathbb{I})(z)\| &  \le & (1+e^{-\frac{C_0z}{\eta^2}})^m, \\     
          \label{eqn:U_gamma_z_bound} \|\partial_z\prod_{\gamma\in\Lambda_{\kappa}}U_\eta(U^\eps_\gamma-\mathbb{I})(z)\| & \le & \frac{2mC_0}{\eta^2}(1+e^{-\frac{C_0z}{\eta^2}})^m.
      \end{eqnarray}
\end{lemma}
\begin{proof}
          The proof follows a similar argument as in Lemma~\ref{lemma:L_pq_bound}. Observing that 
          $\|L_\gamma\|\le \frac{C_0}{\eta^2}$,
          an application of Gr\"{o}nwall's inequality gives the required bounds in~\eqref{eqn:U_gamma_bound} and~\eqref{eqn:U_gamma_prod_bound}. Note that
          \begin{equation*} \partial_z\prod_{\gamma\in\Lambda_{\kappa}}U_\eta(U^\eps_\gamma-\mathbb{I})(z)=mL_{\eta}\prod_{\gamma\in\Lambda_{\kappa}}U_\eta(U^\eps_\gamma-\mathbb{I})(z)+\sum_{\gamma\in\Lambda_{\kappa}}U_{\eta} L_\gamma U^\eps_\gamma\prod_{\gamma\neq\gamma'\in\Lambda_{\kappa}}U_\eta(U^\eps_{\gamma'}-\mathbb{I})(z)\,.
          \end{equation*}
          We obtain the bound in~\eqref{eqn:U_gamma_z_bound} by using~\eqref{eqn:U_gamma_prod_bound} and observing that
          \begin{equation*}
              \|\sum_{\gamma\in\Lambda_{\kappa}}U_{\eta} L_\gamma U^\eps_\gamma\prod_{\gamma\neq\gamma'\in\Lambda_{\kappa}}U_\eta(U^\eps_{\gamma'}-\mathbb{I})(z)\|\le \sum_{\gamma\in\Lambda_{\kappa}}\frac{C_0}{\eta^2}(1+e^{-\frac{C_0z}{\eta^2}})^{m-1}\,.
          \end{equation*}
     This concludes the proof of the result.  \end{proof} 
  We now establish the following bound for $N_{p,q}^\eps$: 
  \begin{corollary}\label{corro:N_bound}
      The operator $N_{p,q}^\eps$ given by~\eqref{eqn:N_def} satisfies
        \begin{equation}
        \begin{aligned}
       \|N_{p,q}^\eps(z)\|&\le 2(p\vee q)!3^{p\wedge q},\quad   \|\partial_zN_{p,q}^\eps(z)\|\le \frac{C_0(p+q)}{\eta^2}\|N_{p,q}^\eps(z)\|.
        \end{aligned}
       \end{equation}
\end{corollary}

\begin{proof}
     From \eqref{eqn:U_gamma_bound}, we have that $\|U^\eps_\gamma-\mathbb{I}\|\le 1+e^{\frac{C_0z}{\eta^2}}$ and from the definition of $N_{p,q}^\eps$ that       
     \begin{equation*}
       \begin{aligned}
           \|N_{p,q}^\eps(z)\|&\le \|{U_\eta}^{\frac{p+q}{2}}\|\sum_{\kappa=1}^{K}\prod_{\gamma\in\Lambda_{\kappa}}\|U^\eps_\gamma-\mathbb{I}\|+\|{U_\eta}^{\frac{p+q}{2}}\|\,.           
       \end{aligned}
       \end{equation*}
       Using  that $U_{\eta}=e^{-\frac{C_0z}{\eta^2}}$ and observing that $\sum_{m(\kappa)=m} = \binom{p}{m} \binom{q}{m} m!$, we obtain from \eqref{eqn:U_gamma_prod_bound} that   \begin{equation}\label{eq:bdNeps}
           \begin{aligned}
               \|N_{p,q}^\eps(z)\|&\le e^{-\frac{(p+q)C_0z}{2\eta^2}}\Big[1+\sum_{\kappa=1}^{K}(1+e^{\frac{C_0z}{\eta^2}})^m\Big]\  \le e^{-\frac{(p+q)C_0z}{2\eta^2}}\Big[1+\sum_{m=1}^{p\wedge q}\binom{p}{m} \binom{q}{m} m!(1+e^{\frac{C_0z}{\eta^2}})^m\Big]\\
               &\le 1+\sum_{m=1}^{p\wedge q}2^m\binom{p}{m} \binom{q}{m} m!\le 1+(p\vee q)!\sum_{m=1}^{p\wedge q}2^m\binom{p}{m}  \le 1+(p\vee q)!3^{p\wedge q}\,.
           \end{aligned}
       \end{equation}
       Now, using the expression for $\partial_zN_{p,q}^\eps$ from~\eqref{eqn:N_z_1} and the estimate in~\eqref{eqn:U_gamma_bound}, we deduce:
       \begin{equation*}
           \begin{aligned}
                           \|\partial_zN_{p,q}^\eps(z)\|&\le \frac{(p+q)C_0}{2\eta^2}\|N_{p,q}^\eps(z)\|+e^{-\frac{(p+q)C_0z}{2\eta^2}}\sum_{\kappa=1}^{K}\frac{mC_0}{\eta^2}e^{\frac{C_0z}{\eta^2}}(1+e^{\frac{C_0z}{\eta^2}})^{m-1}\\
               &\le\frac{(p+q)C_0}{2\eta^2}\|N_{p,q}^\eps(z)\|+\frac{C_0}{\eta^2}\sum_{m=1}^{p\wedge q}m2^m
               \binom{p}{m}\binom{q}{m}
               m!\,.
           \end{aligned}
       \end{equation*}
       This proves the lemma.
   \end{proof}
 The proof of the estimate \eqref{eqn:E_bound} relies on a number of lemmas we now state and prove.
\begin{lemma}\label{lemma:linear_int}
    Let $f:\mathbb{R}^d\to[0,\infty)$ such that $f(\xi)=f(|\xi|)\in \sL^1(\mathbb{R}^d)\cap \sL^\infty (\mathbb{R}^d)$. Also let $g:\mathbb{R}^d\to[0,\infty)$ such that $g(\zeta)\in \sL^1(\mathbb{R}^d)\cap \sL^\infty (\mathbb{R}^d)$. Then $\exists   \delta_0>0$ such that $\forall \delta\in(0,\delta_0]$,
    \begin{equation*}
        \sup_{w\in\mathbb{R}^d}\int_{\Rm^{2d}}
        f(\xi)g(\zeta)\Big(1\wedge\frac{\delta}{|\xi\cdot(\zeta+w)|}\Big) d\xi d\zeta\le 
          {\mathfrak C}_1[\delta,f,g,d]\,,
    \end{equation*}
    with
    \begin{equation}\label{eq:mc1}
        {\mathfrak C}_1[\delta,f,g,d] = C (\|f\|_1+\|f\|_\infty)(\|g\|_1+\|g\|_\infty) \delta \left\{ \begin{array}{cl} |\ln\delta|^2 & d=1, \\ |\ln\delta| & d\ge2,  \end{array}\right.
    \end{equation}
    where $C$ is a constant independent of $f,g,d,\delta$.
\end{lemma}
\begin{proof}
We need to bound
\[ I_0 := \dint_{\Rm^{2d}} g(\zeta-w) f(\xi) \frac{\delta d\xi d\zeta}{\delta+|\xi\cdot\zeta|} = \dint_{\Rm^d} g(\zeta-w) \dint_{\Rm^d} \frac{f(|\xi|) \delta |\xi|^{d-1} d|\xi| d\hat\xi}{\delta + |\xi||\zeta| |\hat\xi\cdot\hat\zeta|} d\zeta.\]
We consider the case $d\geq2$ first. We next choose variables $\xi$ such that $\hat\xi\cdot\hat\zeta=\hat\xi_1=\cos\theta$ and obtain
\[
\dint_{\mathbb{S}^{d-1}} \frac{\delta d\hat\xi}{\delta+|\xi||\zeta| |\hat\xi_1|} = |\mathbb{S}^{d-2}| \dint_0^\pi \frac{\delta (\sin\theta)^{d-2} d\theta}{\delta+|\xi||\zeta||\cos\theta| }
\leq C \frac{\delta}{|\xi||\zeta|} \ln\Big( 1+\frac{|\xi||\zeta|}{\delta}\Big),
\]
for $\mathbb{S}^{d-2}$ the unit sphere in $\Rm^{d-1}$ and using a change of coordinates $u=\cos\theta$. This formula holds for $d\geq3$ with $|\mathbb{S}^{d-2}|(\sin\theta)^{d-2}$ replaced by $2$ when $d=2$ for the same right-hand-side bound.

We thus need to estimate
\[ I_0\leq \dint_{\Rm^d} g(\zeta-w) \dint_0^\infty \frac{\delta |\xi|^{d-1}  f(|\xi|)}{|\xi||\zeta|}\ln \Big(1+\frac{|\xi||\zeta|}{\delta}\Big) d|\xi| d\zeta \le \dint_{\Rm^{2d}} g(\zeta-w) f(\xi) \frac{\delta}{|\xi||\zeta|}\ln \Big( 1+\frac{|\xi||\zeta|}{\delta}\Big) d\xi d\zeta. \]
We compute
\[
  \dint_0^1 \frac{|\zeta|^{d-1}}{|\zeta|} \ln  \Big( 1+ \frac{|\xi||\zeta|}\delta\Big) d|\zeta| \leq \ln  \Big( 1+ \frac{|\xi| }\delta\Big). 
\]
The domain $|\xi|\leq1$ and $|\zeta|\leq1$ thus gives a contribution to the integral $I_0$ bounded by
\[ \|g\|_\infty \|f\|_\infty \int_0^1 \frac{\delta|\xi|^{d-1}}{|\xi|}\ln  \Big( 1+ \frac{|\xi| }\delta\Big) d|\xi| \leq C \|g\|_\infty \|f\|_\infty \delta |\ln\delta|. \]
The domain $|\zeta|\leq1$ and $|\xi|\geq1$ then gives a contribution to the integral $I_0$ bounded
\[
 \|g\|_\infty \dint_{|\xi|\geq1} f(\xi) \frac{\delta}{|\xi|} \ln  \Big( 1+ \frac{|\xi| }\delta\Big) d\xi  \leq C\|g\|_\infty \|f\|_1 \delta |\ln\delta|.
\]
Symmetrically, the domain $|\xi|\leq1$ and $|\zeta|\geq1$ gives a contribution bounded by $\|f\|_\infty \|g\|_1 \delta |\ln\delta|$.
It remains to analyze the domain $|\xi|\geq1$ and $|\zeta|\geq1$, which clearly provides a contribution bounded by $\|f\|_1\|g\|_1 \delta|\ln\delta|$ since $\frac\delta\alpha \ln (1+\frac\alpha\delta)$ is bounded by $C\delta|\ln\delta|$ uniformly in $\alpha\geq1$ for $\delta\leq1$.

When $d=1$, we may assume $\xi\geq0$ and $\zeta\geq0$ by symmetry and need to estimate
\[  \dint_0^\infty \dint_0^\infty g(\zeta-w) f(\xi) \frac{\delta}{\delta+\xi\zeta} d\xi d\zeta.\]
We compute
\[ \dint_0^1 \frac{\delta d\xi}{\delta+\xi\zeta} = \frac{\delta}{ \zeta} \ln \Big(1+\frac\zeta\delta\Big),\qquad 
\dint_0^1  \frac{\delta}{ \zeta} \ln \Big(1+\frac\zeta\delta\Big) d\zeta\leq \delta + \int_\delta^1 \frac{\delta}{ \zeta} \ln \Big(1+\frac\zeta\delta\Big) d\zeta\leq C \delta |\ln\delta|^2.\]
The domain $\xi\leq1$ and $\zeta\leq1$ thus gives a contribution $\|g\|_\infty\|f\|_\infty\delta |\ln\delta|^2$. The other contributions are treated as in the case $d\geq2$.
\end{proof}
We have the following lemma:
  \begin{lemma}\label{lemma:quad_inter}
    Let $f:\mathbb{R}^d\to[0,\infty)$ such that $f\in \sL^1(\mathbb{R}^d)\cap \sL^\infty(\mathbb{R}^d)$. We also assume that $\| \aver{\xi}^{d-2} f(\xi)\|_\infty\leq C$ in dimension $d\geq3$.  Then $\exists   \delta_0>0$ such that $\forall \delta\in(0,\delta_0]$, we have
    \begin{equation}\label{eq:lemma2}
        \sup_{w\in\mathbb{R}^d}\int_{\mathbb{R}^d}f(\xi+w)\min\Big\{1,\frac{\delta}{||\xi|^2-|w|^2|}\Big\}\mathrm{d}\xi\le 
         {\mathfrak C}_2[f,\delta,d],
    \end{equation}
    where for $C$ independent of $(f,\delta,d)$,
    \begin{equation}\label{eq:mc2}
        {\mathfrak C}_2[\delta,f,d] = \left\{\begin{array}{cl} 2\pi \|f\|_\infty \sqrt\delta & d=1, \\[1mm] C(\|\aver{\xi}^{d-2}f\|_\infty \delta \ln \delta^{-1}+\|f\|_1 \delta) & d\geq2.\end{array}\right.
    \end{equation}
\end{lemma}
\begin{proof}
    The right-hand side in \eqref{eq:lemma2} is bounded by (the supremum over $w$ of)
    \[  I_1:= \int_{\Rm^d} f(\xi+w) \frac{\delta d\xi}{||\xi|^2-|w|^2|+\delta} = \dint_{\Rm^d}f(|w|(\xi+\hat w)) \frac{\delta  |w|^d d\xi}{|w|^2 ||\xi|^2-1|+\delta}. \]

    We consider the case $d\geq2$ first.
    We treat the cases $|w|\leq1$ and $|w|\geq1$ separately.
    For $|w|\geq1$, define the set $D_w=\{\xi\in\Rm^d; |w|^2||\xi|^2-1|\leq 1\}$. The above right-hand-side integral $I_1$ over $|\xi|\in \Rm^d\backslash D_w$ is clearly bounded by $C\delta\|f\|_1$. So, it remains to estimate
    \[ \dint_{D_w}f(|w|(\xi+\hat w)) \frac{\delta  |w|^d d\xi}{|w|^2 ||\xi|^2-1|+\delta}. \]
    Up to some rotation in our choice of coordinates $\xi$, we may assume that $w=|w|e_1$. We then use, for $\xi=(\xi_1,\xi_2)$ with $\xi_1\in\Rm$ and $\xi_2\in\Rm^{d-1}$,
    \[ f(|w|(\xi_1+1,\xi_2)) |w|^{d-2} |(\xi_1+1,\xi_2)|^{d-2} \leq C\] to deduce that the above expression is bounded by 
    \[\int_{D_w} \frac{\delta |w|^2 d\xi_1 d\xi_2}{|(\xi_1+1,\xi_2)|^{d-2} (\delta+|w|^2 |\xi_1^2+|\xi_2|^2-1|)}. \]
    Now use $d\xi_2=|\xi_2|^{d-2}d|\xi_2|d\theta$ and the fact that $|\xi_2|^{d-2}\leq |(\xi_1+1,\xi_2)|^{d-2}$ to get the bound with $r=|\xi_2|$:
    \[ \dint_{\{|w|^2|\xi_1^2+r^2-1|\leq1\}} \frac{C\delta |w|^2 d\xi_1 dr}{\delta+|w|^2|\xi_1^2+r^2-1|} \leq C \dint_{\sqrt{1-|w|^{-2}}}^{\sqrt{1+|w|^{-2}}} \frac{\delta |w|^2 \rho d\rho}{\delta+|w|^2|\rho^2-1|} \]
    where we used $\rho$ for the radial variables of $(\xi_1,\xi_2)$, i.e., $\rho^2=r^2+\xi_1^2$. We thus obtain the bound
    \[C \dint_{1-|w|^{-2}}^{1+|w|^{-2}} \frac{\delta |w|^2  d\rho}{\delta+|w|^2|\rho-1|} = 2C \dint_0^{|w|^{-2}} \frac{\delta |w|^2d\rho}{\delta+|w|^2\rho} = 2C  \dint_0^{1} \frac{\delta d\rho}{\delta+\rho} \leq C' \delta |\ln\delta|. \]

    When $|w|\leq1$, we use the left-hand-side form of the integral $I_1$ to show that the integral over $|\xi|^2\geq2$ gives a contribution bounded by $\delta\|f\|_1$. On the rest of the domain, we bound $f$ by $\|f\|_\infty$ to obtain an upper bound of the form
    \[\dint_{|\xi|^2\leq2} \frac{\delta |\xi|^{d-1} d|\xi| d\hat\xi}{||\xi|^2-|w|^2|+\delta} \leq C\int_0^2 \frac{\delta r^{d-1}dr}{|r^2-|w|^2|+\delta} \leq C\dint_0^2\frac{\delta r^{d-1}dr}{|r-|w|||r+|w||+\delta}\leq C \dint_0^3 \frac{\delta dr}{r+\delta} \leq C \delta |\ln\delta|. \]

    When $d=1$, we have 
    \[\dint_{\Rm} \frac{f(\xi+w) \delta d\xi}{|\xi^2-|w|^2|+\delta} \leq \dint_0^\infty \frac{2\|f\|_\infty \delta d\xi}{|\xi^2-|w|^2| +\delta} \leq  \dint_0^\infty \frac{2\|f\|_\infty \delta d\xi}{(\xi-|w|)^2 +\delta}  \leq \dint_\Rm \frac{2\|f\|_\infty \delta d\xi}{\xi^2 +\delta}  = 2\pi \|f\|_\infty \sqrt\delta.\]
    All bounds are independent of $w$ and this concludes the proof of the lemma.
    \end{proof}

\begin{proposition}\label{prop:linear_error}
    Suppose $\gamma=(j,l)$ and $\gamma'=(j',l')$ such that either $j=j'$ or $l=l'$, but not both. Then: 
    \begin{equation*}
       \sup_{0\le z'\le z} \Big\|\int_{z'}^zL^1_{\gamma}(s)[U_{\eta}(U^\eps_{\gamma'}-\mathbb{I})](s)\mathrm{d}s\Big\|\le \ \frac{C_1\langle z\rangle}{\eta^4} \ {\mathfrak C}_1[\frac{C_2\eps}{\eta},{\hat \sR},\hat R,d]
    \end{equation*}
    where $C_1$ and $C_2$ are constants independent of $\eps$ and $z$ and ${\mathfrak C}_1$ is defined in \eqref{eq:mc1}.
\end{proposition}
\begin{proof}
Without loss of generality, let $\gamma=(j,l)$ and $\gamma'=(j,l'), l'\neq l$.  From~\eqref{eqn:U_gamma}, we find that
\begin{equation*}
            (U^\eps_{\gamma'}-\mathbb{I})(z)=\int_{0}^zL^1_{\gamma'}U^\eps_{\gamma'}(s)\mathrm{d}s\,.
\end{equation*}
For any $\rho\in\mathcal{M}_B(\mathbb{R}^{pd+qd})$, this gives us
\begin{equation*}
            \begin{aligned}
                \Big\|\int_{z'}^zL^1_{\gamma}(s)[U_{\eta}(U^\eps_{\gamma'}-\mathbb{I})](s)\mathrm{d}s\rho\Big\|&=  
                \Big\|\int_{z'}^zL^1_{\gamma}(s_1)U_{\eta}(s_1)\int_{0}^{s_1}L^{1}_{\gamma'}(s_2)
                U^\eps_{\gamma'}(s_2)\mathrm{d}s_2\mathrm{d}s_1\rho\Big\|\,.
            \end{aligned}
\end{equation*}
Using the definition of the operators $L^1_\gamma$ and $L^1_{\gamma'}$ from~\eqref{eqn:L_pq_def}, we obtain
        \begin{equation*}
            \begin{aligned}
                              & \Big\|\int_{z'}^zL^1_{\gamma}(s)[U_{\eta}(U^\eps_{\gamma'}-\mathbb{I})](s)\mathrm{d}s\rho\Big\|=  \Big(\frac{k_0^2}{4(2\pi)^d\eta^2}\Big)^2\Big\|\int_{z'}^z\int_{\Rm^d}\hat{R}(k_1)e^{\frac{is_1\eta}{k_0\eps}k_1^\tps A_{\gamma}^\tps\Theta v}U_{\eta}(s_1)\\
                &\times\int_{0}^{s_1}\int_{\Rm^d}\hat{R}(k_2)e^{\frac{is_2\eta}{k_0\eps}k_2^\tps A_{\gamma'}^\tps\Theta(v-A_{\gamma}k_1)}U^\eps_{\gamma'}(s_2,v-A_{\gamma}k_1-A_{\gamma'}k_2)\rho(v-A_{\gamma}k_1-A_{\gamma'}k_2)\mathrm{d}k_2\mathrm{d}s_2\mathrm{d}k_1\mathrm{d}s_1\Big\|\,.
            \end{aligned}
        \end{equation*}
Performing a change of variables $v-A_{\gamma}k_1-A_{\gamma'}k_2\to v$ and noting that $A_{\gamma}^\tps\Theta A_{\gamma}=A_{\gamma'}^\tps\Theta A_{\gamma'}=0$ and $A_{\gamma}^\tps\Theta A_{\gamma'}=\mathbb{I}_d$, we obtain, recalling that $U_\eta(z)=e^{L_\eta z}$  with $L_\eta=-\frac{C_0}{\eta^2}=-\frac{k_0^2 R(0)}{4\eta^2}$,
         \begin{equation*}
            \begin{aligned}
               & \Big(\frac{k_0^2}{4(2\pi)^d\eta^2}\Big)^{-2} \Big\|\int_{0}^zL^1_{\gamma}(s)[U_{\eta}(U^\eps_{\gamma'}-\mathbb{I})](s)\mathrm{d}s\rho\Big\| \\ 
               =&\Big\|\int_{z'}^z\int_{\Rm^d}\hat{R}(k_1)e^{\frac{is_1\eta}{k_0\eps}k_1\cdot (\xi_j-\zeta_l+k_2)}U_{\eta}(s_1)
                \int_{0}^{s_1}\int_{\Rm^d}\hat{R}(k_2)e^{\frac{is_2\eta}{k_0\eps}k_2^\tps A_{\gamma'}^\tps\Theta v}U^\eps_{\gamma'}(s_2,v)\rho(v)\mathrm{d}k_2\mathrm{d}k_1\mathrm{d}s_2\mathrm{d}s_1\Big\|\\
                \le&\sup_{w\in\mathbb{R}^d}\int_{\Rm^{2d}}\hat{\sR}(k_1)\hat{R}(k_2)\mathrm{d}k_1\mathrm{d}k_2 \int_{0}^{z'}\|U^\eps_{\gamma'}(s_2)\rho\|\mathrm{d}s_2 \Big|\int_{z'}^zU_{\eta}(s_1)e^{\frac{is_1\eta}{k_0\eps}k_1\cdot(k_2+w)}\mathrm{d}s_1\Big|\\
                &+\sup_{w\in\mathbb{R}^d}\int_{\Rm^{2d}}\hat{\sR}(k_1)\hat{R}(k_2)\mathrm{d}k_1\mathrm{d}k_2 \int_{z'}^{z}\|U^\eps_{\gamma'}(s_2)\rho\|\mathrm{d}s_2 \Big|\int_{s_2}^zU_{\eta}(s_1)e^{\frac{is_1\eta}{k_0\eps}k_1\cdot(k_2+w)}\mathrm{d}s_1\Big|\,.
            \end{aligned}
        \end{equation*}
Note that after performing an integration by parts, the term $|\int_{s}^zU_{\eta}(s_1)e^{\frac{is_1\eta}{k_0\eps}k_1\cdot(k_2+w)}\mathrm{d}s_1|$ is bounded by the minimum of $\int_{s}^zU_{\eta}(s_1)\mathrm{d}s_1$ and
        \begin{equation*}
            \begin{aligned}
                    \frac{k_0\eps}{\eta |k_1\cdot(k_2+w)|}\Big|e^{\frac{iz\eta k_1\cdot(k_2+w)}{k_0\eps}}U_{\eta}(z)-e^{\frac{is\eta k_1\cdot(k_2+w)}{k_0\eps}}U_{\eta}(s)-
                    \int_{s}^ze^{\frac{is_1\eta k_1\cdot(k_2+w)}{k_0\eps}}L_{\eta} U_\eta(s_1)\mathrm{d}s_1\Big|.                
            \end{aligned}
        \end{equation*}
This gives us
        \begin{equation*}
             \Big|\int_{s}^zU_{\eta}(s_1)e^{\frac{is_1\eta}{k_0\eps}k_1\cdot(k_2+w)}\mathrm{d}s_1\Big|\le U_{\eta}(s)\Big\{\frac{2\eta^2}{C_0}\wedge \frac{4k_0\eps}{\eta|k_1\cdot(k_2+w)|}\Big\}\le \frac{2}{C_0}U_{\eta}(s)\Big\{1\wedge \frac{2k_0C_0\eps}{\eta|k_1\cdot(k_2+w)|}\Big\}\,,
        \end{equation*}
as $\eta\le 1$, with $C_0=k_0^2R(0)/4$ as in~\eqref{eqn:L_pq_def}. Using \eqref{eqn:U_gamma_bound}, we have
        \begin{equation*}
            \int_{0}^{z'}\|U^\eps_{\gamma'}(s_2)\rho\|U_{\eta}(z')\mathrm{d}s_2+\int_{z'}^z\|U^\eps_{\gamma'}(s_2)\rho\|U_{\eta}(s_2)\mathrm{d}s_2\le 2\langle z\rangle\|\rho\|\,.
        \end{equation*}
Finally, this gives us the bound
        \begin{equation*}
            \begin{aligned}
                \Big\|\int_{z'}^zL^1_{\gamma}(s)[U_{\eta}(U^\eps_{\gamma'}-\mathbb{I})](s)\mathrm{d}s\Big\|&\le \frac{4}{C_0}\Big(\frac{k_0^2}{4(2\pi)^d\eta^2}\Big)^2\langle z\rangle\sup_{w\in\mathbb{R}^d}\int_{\Rm^{2d}}\hat{\sR}(k_1)\hat{R}(k_2)\Big\{1\wedge \frac{2k_0C_0\eps}{\eta|k_1\cdot(k_2+w)|}\Big\}\mathrm{d}k_1\mathrm{d}k_2\,.
            \end{aligned}
        \end{equation*}
An application of Lemma \ref{lemma:linear_int} with $\delta=2k_0C_0\eps/\eta$, $f=\hat{\sR}$ and $g=\hat{R}$ concludes the proof of the proposition.
\end{proof}
\begin{proposition}\label{prop:quad_error}
    In operator norm and with ${\mathfrak C}_2$ defined in \eqref{eq:mc2}, we have: 
    \begin{equation*}
       \sup_{0\le z'\le z} \Big\|\int_{z'}^zL^2_{j,j'}(s)\mathrm{d}s\Big\|\le\  \frac{ C \aver{z}}{\eta^2} \  {\mathfrak C}_2[k_0\frac\eps\eta,\hat R,d]
       ,\quad \forall j,j'\,.
    \end{equation*}
    \begin{proof}
Suppose $1\le j<j'\le p$. The case with $p+1\le j<j'\le p+q$ is dealt with in a similar manner. For any $\rho\in\mathcal{M}_B(\mathbb{R}^{pd+qd})$, using the definition of $L^2_{j,j'}$~\eqref{eqn:L_pq_def} gives
\begin{equation*}
    \begin{aligned}
         \Big\|\int_{z'}^zL^2_{j,j'}(s)\mathrm{d}s\rho\Big\|&=\frac{k_0^2}{4\eta^2}\Big\|\int_{z'}^z\int_{\Rm^d}\hat{R}(k)e^{\frac{is\eta}{2k_0\eps}(2k^\tps B_{j,j'}^\tps \Theta v-k^\tps B_{j,j'}^\tps\Theta B_{j,j'}k)}\rho(v-B_{j,j'}k)\frac{\mathrm{d}k\mathrm{d}s}{(2\pi)^d}\Big\|.
    \end{aligned}
\end{equation*}
Using the change of variables $v-B_{j,j'}k\to v$ and observing that $B_{j,j'}^\tps\Theta B_{j,j'}=2\mathbb{I}_d$ and ${\color{black} k^\tps} B_{j,j'}^\tps\Theta v=k\cdot(\xi_j-\xi_{j'})$ {\color{black}from~\eqref{eqn:quad_k_B}}, we obtain 
\begin{equation*}
    \begin{aligned}
        \Big \|\int_{z'}^zL^2_{j,j'}(s)\rho\mathrm{d}s\Big\|&=\frac{k_0^2}{4\eta^2}\Big\| \int_{\Rm^d}\hat{R}(k)\int_{z'}^ze^{\frac{is\eta}{k_0\eps}k\cdot(\xi_j-\xi_{j'}+k)}\rho(v) \frac{\mathrm{d}s \mathrm{d}k}{(2\pi)^d}\Big\|\\
        &\le\frac{k_0^2}{4\eta^2}\|\rho\|\sup_{w\in\mathbb{R}^d}\int_{\Rm^d}\hat{R}(k)\Big|\int_{z'}^ze^{\frac{is\eta}{k_0\eps}k\cdot(k+w)}\mathrm{d}s\Big|\dfrac{\mathrm{d}k}{(2\pi)^d}.
    \end{aligned}
\end{equation*}
Note that
\begin{equation*}
   \Big|\int_{z'}^ze^{\frac{is\eta}{k_0\eps}k\cdot(k+w)}\mathrm{d}s\Big|\le\begin{cases}
                    z-z',\quad \dfrac{k_0\eps}{\eta |k\cdot(k+w)|}\ge 1\\
                    \dfrac{k_0\eps}{\eta |k\cdot(k+w)|}\Big|e^{\frac{iz\eta k\cdot(k+w)}{k_0\eps}}-e^{\frac{iz'\eta k\cdot(k+w)}{k_0\eps}}-\dint_{z'}^ze^{\frac{is\eta k\cdot(k+w)}{k_0\eps}}\mathrm{d}s\Big|,\quad \text{otherwise}\,.
                    \end{cases}
\end{equation*}
After a change of variables $k+w/2\to k$, this gives
\begin{equation*}
    \begin{aligned}
      \sup_{w\in\mathbb{R}^d} \int_{\Rm^d}\hat{R}(k)\Big|\int_{z'}^ze^{\frac{is\eta}{k_0\eps}k\cdot(k+w)}\mathrm{d}s\Big| \mathrm{d}k&\le C\langle z\rangle\sup_{w\in\mathbb{R}^d}\int_{\Rm^d}\hat{R}(k-w)\Big(1\wedge\frac{k_0\eps}{\eta ||k|^2-|w|^2|}\Big)\mathrm{d}k\,.
    \end{aligned}
\end{equation*}
An application of Lemma \ref{lemma:quad_inter} with $\delta=k_0\eps/\eta$ and $f=\hat R$ concludes the proof of the proposition. 
    \end{proof}
\end{proposition}
We will also need the following extension of the previous results:
\begin{lemma}\label{lemma:bound_time_dep_sources}
    Suppose that  $\forall s\in[0,z]$, $L(s)$  is a bounded operator on $\mathcal{M}_B(\mathbb{R}^{pd+qd})$ such that
    \begin{equation*}
       \sup_{0\le z'\le z} \Big\|\int_{z'}^zL(s)\mathrm{d}s\Big\|\le \bar{\kappa}<\infty\,.
    \end{equation*}
    For any family of distributions $C^1(\Rm_+;\mathcal{M}_B(\mathbb{R}^{pd+qd})) \ni \rho(s)=\rho(0)+\int_{0}^s\rho'(s')\mathrm{d}s' $ such that    
    \begin{equation*}
      \|\rho(0)\|+\sup_{0\le z'\le z}\|\rho'(z')\|:= \bar{\rho}<\infty\,,
    \end{equation*}
    then we have:
    \begin{equation*}
         \Big\|\int_{0}^zL(s)\rho(s)\mathrm{d}s\Big\|\le \langle z\rangle\bar{\kappa}\bar{\rho}\,.
    \end{equation*}
    \begin{proof}
        This is a simple verification:        
        \begin{equation*}
            \begin{aligned}
                \Big\|\int_{0}^zL(s)\rho(s)\mathrm{d}s\Big\|&\le\Big\|\int_{0}^zL(s)\rho(0)\mathrm{d}s\Big\|+\Big\|\int_{0}^zL(s_1)\int_{0}^{s_1}\rho'(s_2)\mathrm{d}s_2\mathrm{d}s_1\Big\|\\
                &\le \bar{\kappa}\|\rho(0)\|+\Big\|\int_{0}^z\int_{s_2}^zL(s_1)\rho'(s_2)\mathrm{d}s_1\mathrm{d}s_2\Big\|
                \le\bar{\kappa}\|\rho(0)\|+\bar{\kappa}\int_{0}^z\|\rho'(s_2)\|\mathrm{d}s_2\le \langle z\rangle\bar{\kappa}\bar\rho\,.
            \end{aligned}
        \end{equation*}
    \end{proof}
\end{lemma}
We finally obtain the main estimate of this section:
\begin{corollary}\label{cor:boundE}
  The bound \eqref{eqn:E_bound} on $E_{p,q}^\eps$ holds.
\end{corollary}
   \begin{proof}
   Integrating \eqref{eqn:E} and using Lemma \ref{lemma:L_pq_bound}, we have:
\begin{equation*}
    \|E_{p,q}^\eps(z)\|\le \frac{C_0(p+q)^2}{2\eta^2}\int_{0}^z\|E_{p,q}^\eps(s)\|\mathrm{d}s+\Big\|\int_{0}^z\mathcal{E}_{p,q}^{\eps}(s)\mathrm{d}s\Big\|\,.
\end{equation*}
From~\eqref{eqn:error_source}, the source $\mathcal{E}_{p,q}^{\eps}$ to the evolution equation~\eqref{eqn:E} of $E_{p,q}^\eps$ upon integration in $z$ is bounded by
\begin{equation*}
    \Big\|\int_{0}^z\mathcal{E}_{p,q}^{\eps}(s)\mathrm{d}s\Big\|\le \sum_{\kappa=1}^{K}\sum_{\gamma'\in\bar\Lambda_{\kappa}}
    \Big\|\int_{0}^zU_\eta^{\frac{p+q}{2}}L^1_{\gamma'}\prod_{\gamma\in\Lambda_{\kappa}}(U^\eps_\gamma-\mathbb{I})(s)\mathrm{d}
    s\Big\|+\Big\|\int_{0}^zL^2N_{p,q}^\eps(s)\mathrm{d}s\Big\|\,.
\end{equation*}
Note that thanks to \eqref{eq:commute} and by construction of $\bar\Lambda_{\kappa}$, the product $L^1_{\gamma'}\prod_{\gamma\in\Lambda_{\kappa}}(U^\eps_\gamma-\mathbb{I})$ can be re-arranged so that $L^1_{\gamma'}$ acts on $(U_{\gamma''}-\mathbb{I})$ such that $L^1_{\gamma'}$ and $L^1_{\gamma''}$ do not commute so that Proposition \ref{prop:linear_error} applies.  This gives us
with $m=m(\kappa)$,
\begin{equation*}
    \begin{aligned}
        \Big\|\int_{0}^zU_\eta^{\frac{p+q}{2}}L^1_{\gamma'}\prod_{\gamma\in\Lambda_{\kappa}}(U^\eps_\gamma-\mathbb{I})(s)\mathrm{d}s\Big\|&
          =\Big\|\int_{0}^zL^1_{\gamma'}{U_\eta}(U^\eps_{\gamma''}-\mathbb{I})(s){U_\eta}^{\frac{p+q}{2}-m}\prod_{\gamma''\neq \gamma\in\Lambda_{\kappa}}U_{\eta}(U^\eps_{\gamma}-\mathbb{I})(s)\mathrm{d}s\Big\|.
    \end{aligned}
\end{equation*}
For $\chi\in {\mathcal M}_B(\Rm^{(p+q)d})$ with $\|\chi\|=1$,
define ${\rho}(s)={U_\eta}^{\frac{p+q}{2}-m}\prod_{\gamma''\neq \gamma\in\Lambda_{\kappa}}U_{\eta}(U^\eps_{\gamma}-\mathbb{I})(s)\chi$. We verify that $\rho(0)=0$ and using  \eqref{eqn:U_gamma_z_bound} in Lemma \ref{lemma:U_gamma_bound} obtain that
\begin{equation*}
  \begin{aligned}
    \|\partial_s\rho(s)\|&\le\big(\frac{p+q}{2}-m\big)|L_\eta U_\eta^{\frac{p+q}{2}-m}| \Big\|\prod_{\gamma''\neq \gamma\in\Lambda_{\kappa}}U_{\eta}(U^\eps_{\gamma}-\mathbb{I})(s)\Big\| +U_\eta^{\frac{p+q}{2}-m}\Big\|\partial_s\prod_{\gamma''\neq \gamma\in\Lambda_{\kappa}}U_{\eta}(U^\eps_{\gamma}-\mathbb{I})(s)\Big\|\\
    &\le \big(\frac{p+q}{2}-m\big)\frac{C_0}{\eta^2}e^{-\frac{C_0}{\eta^2}\big(\frac{p+q}{2}-m\big)}\big(1+e^{-\frac{C_0z}{\eta^2}}\big)^{m-1}+\frac{2(m-1)C_0}{\eta^2}e^{-\frac{C_0}{\eta^2}\big(\frac{p+q}{2}-m\big)}\big(1+e^{-\frac{C_0z}{\eta^2}}\big)^{m-1},
    \end{aligned}  
\end{equation*}
which is bounded by $c(p,q)\eta^{-2}$. From Lemma \ref{lemma:bound_time_dep_sources}, this gives
\begin{equation*}
\begin{aligned}
    \Big\|\int_{0}^zL^1_{\gamma'}{U_\eta}(U^\eps_{\gamma''}-\mathbb{I})(s){U_\eta}^{\frac{p+q}{2}-m} \!\!\! \prod_{\gamma''\neq \gamma\in\Lambda_{\kappa}}  \!\! U_{\eta}(U^\eps_{\gamma}-\mathbb{I})(s)\mathrm{d}s\Big\|&\le c_1(p,q)\eta^{-2}\langle z\rangle\sup\limits_{0\le z'\le z} \Big\|\int\limits_{z'}^zL^1_{\gamma'}{U_\eta}(U^\eps_{\gamma''}-\mathbb{I})(s)\mathrm{d}s\Big\|\\
    &\le c_1'(p,q)\langle z\rangle^2\eta^{-6}\mathfrak{C}_1[\frac{C_2\eps}{\eta},\hat \sR,\hat R,d]\,,
\end{aligned}
\end{equation*}
for $c_1(p,q)$ independent of $z$ and $\eps$, where we have used Proposition \ref{prop:linear_error} in the last inequality.
  Similarly, using Corollary \ref{corro:N_bound}, Lemma \ref{lemma:bound_time_dep_sources} and Proposition \ref{prop:quad_error} we have for $c_2(p,q)$ independent of $z$ and $\eps$, that
\begin{equation*}
    \begin{aligned}
        \Big\|\int_{0}^zL^2(s)N_{p,q}^\eps(s)\mathrm{d}s\Big\|&\le c_2(p,q)\langle z\rangle^2\eta^{-4}\mathfrak{C}_2[\frac{k_0\eps}{\eta},\hat{R},d]\,.
    \end{aligned}
\end{equation*}
Since $\eta\le 1$, we have that 
\begin{equation*}
      \Big\|\int_{0}^z\mathcal{E}_{p,q}^{\eps}(s)\mathrm{d}s\Big\|\le c(p,q)\langle z\rangle^2\eta^{-6}\big(\mathfrak{C}_1[\frac{C_2\eps}{\eta},{\hat{\sR}},\hat{R},d]+\mathfrak{C}_2[\frac{k_0\eps}{\eta},\hat{R},d]\big)\,,
\end{equation*}
where $c(p,q)$ is a constant independent of $\eps$ and $z$. The statement of the corollary is now a consequence of Gr\"{o}nwall's inequality. More precisely, we obtained a bound for $\|E_{p,q}^\eps(z)\|$ of the form $c(p,q,z)\eps^\gamma e^{\frac{C(p,q,z)}{\eta^2}}$ for $0<\gamma<\frac12$ and $c,C$ bounded uniformly on compact sets. Controlling this term dictates our choice of $\eta(\eps)$ such that $\eps^\gamma$ is much smaller than $e^{\frac{C(p,q,z)}{\eta^2}}$ is large. Let $L=\ln\eps^{-1}$ with $\eta^{-1}=\ln L$. We find for any $n\geq1$, for instance $n=2$, that
\[
\eps^\gamma e^{\frac{C}{\eta^n}} = e^{-\gamma L} e^{ C (\ln L)^n} \leq e^{-\gamma' L} = \eps^{\gamma'}
\]
when $\gamma'<\gamma$ for any $L$ sufficiently large and hence $\eps$ sufficiently small.
\end{proof}
The above bounds are clearly not uniform in the moment $(p,q)$. However, for each $(p,q)$ fixed, the error-term operator $E^\eps_{p,q}$ converges to $0$ as $\eps\to0$. As a consequence, we obtain that the solution $\hat \mu^\eps_{p,q}(z,v)$ of \eqref{eq:mupqF} is given by
\begin{equation}\label{eq:hatmupqsol}
    \hat \mu^\eps_{p,q}(z,v) = \Pi^\eps_{p,q}(z,v)[U^\eps \hat \mu^\eps_{p,q}(0)](z,v) =\Pi^\eps_{p,q}(z,v) [N^\eps   \hat \mu^\eps_{p,q}(0)](z,v) 
    + \Pi^\eps_{p,q}(z,v) [E^\eps   \hat \mu^\eps_{p,q}(0)](z,v) 
\end{equation}
with $\|\Pi^\eps_{p,q}(z,v) [E^\eps   \hat \mu^\eps_{p,q}(0)](z,v)\|\leq \eps^{\frac13}$ for $0<\eps\leq \eps_0(p,q,z)$.

To derive the results of Theorem \ref{thm:kinetic} for general incident beams as in \eqref{eq:u0plane}, we also need to derive approximations on the centered moments. They are defined in the Fourier variables as
\begin{equation}\label{eq:hatbarmu}
\widehat{\tilde \mu^\eps_{p,q}}(z,v) = \mathbb{E}[\prod\limits_{j=1}^p(\hat{u}^\eps(z,\xi_j)-\hat{\mu}^\eps_{1,0}(z,\xi_j))\prod\limits_{l=1}^q({\hat{u}^{\eps*}}(z,\zeta_l)-{\hat{\mu}^{\eps\ast}_{1,0}}(z,\zeta_l))].
\end{equation}
Let us define the operator
\begin{equation}\label{eqn:N_centered}
    \tilde{N}_{p,q}^\eps(z,v)=\begin{cases}
        0,\quad p\neq q\\
        \dsum_{\pi_p}\prod\limits_{j=1}^pU_\eta(U^\eps_{j,\pi_p(j)}-\mathbb{I}),\quad p=q,
    \end{cases}
\end{equation}
where the sum is over all permutations $\pi_p$ of the integers from $1$ to $p$. This operator plays the same role as $N^\eps_{p,q}$ for centered moments:
\begin{corollary}\label{corro:centred_mom}
Let $\widehat{\tilde \mu^\eps_{p,q}}(z,v)$ be the centered moments defined in \eqref{eq:hatbarmu}. Then we have
\begin{equation*}
    \widehat{\tilde \mu^\eps_{p,q}}(z,v)=\Pi^\eps_{p,q}(z,v)[\tilde{N}_{p,q}^\eps(z)\hat{\mu}^\eps_{p,q}(0)+\tilde{E}_{p,q}^\eps(z)\hat{\mu}^\eps_{p,q}(0)](z,v), \qquad    \sup\limits_{0\le z'\le z}\|\tilde{E}_{p,q}^\eps(z')\|
    \leq\eps^{\frac{1}{3}},
\end{equation*}
for $0<\eps < \eps_0(p,q,z)$.  
\end{corollary}    
\begin{proof}
Note that
\begin{equation*}
\begin{aligned}
    \widehat{\tilde \mu^\eps_{p,q}}(z,v)=\sum\limits_{m=0}^p\sum\limits_{n=0}^q\sum\limits_{S^p_m}\sum\limits_{T^q_n}(-1)^{p+q-m-n}\hat{\mu}^\eps_{m,n}(z,\xi_{j_1},\cdots\xi_{j_m},\zeta_{l_1},\cdots,\zeta_{l_n})\prod\limits_{j\notin S^p_m}\hat{\mu}^\eps_{1,0}(z,\xi_j)\prod\limits_{l\notin T^q_n}\hat{\mu}_{1,0}^{\eps\ast}(z,\zeta_l)\,,
    \end{aligned}
\end{equation*}
where the sums are over all sets $S^p_m=\{j_1,\cdots,j_m\}$ that contain $m$ integers drawn from $1$ to $p$ and all sets $T^q_n=\{l_1,\cdots,l_n\}$ that contain $n$ integers drawn from $1$ to $q$.
Using the decomposition of $U^\eps_{m,n}$ in Theorem \ref{thm:Ueps} for the $m+n$th moment on variables $(\xi_{j_1},\cdots\xi_{j_m},\zeta_{l_1},\cdots,\zeta_{l_n})$ gives 
\begin{equation*}
        \begin{aligned}
            \hat{\mu}^\eps_{m,n}(z,\xi_{j_1},\cdots\xi_{j_m},\zeta_{l_1},\cdots,\zeta_{l_n})&=\Pi^\eps_{m,n}(z)[N_{m,n}^\eps(z)\hat{\mu}^\eps_{m,n}(0)+E_{m,n}^\eps(z)\hat{\mu}^\eps_{m,n}(0)](z,\xi_{j_1},\cdots\xi_{j_m},\zeta_{l_1},\cdots,\zeta_{l_n})\,.
        \end{aligned}
\end{equation*}
Define $\epsilon^{pq}_{mn}=(-1)^{p+q-m-n}$.
For $K(m,n)=\sum_{l=1}^{m\wedge n} \binom{m}{l} \binom{n}{l} \, l! $, using the expansion for $N_{m,n}^\eps$  \eqref{eqn:N_def} gives
\begin{equation*}
    \begin{aligned}        &\sum\limits_{m=0}^p\sum\limits_{n=0}^q\sum\limits_{S^p_m}\sum\limits_{T^q_n}\epsilon^{pq}_{mn}[\Pi^\eps_{m,n}(z)N_{m,n}^\eps(z)\hat{\mu}^\eps_{m,n}(0)](z,\xi_{j_1},\cdots\xi_{j_m},\zeta_{l_1},\cdots,\zeta_{l_n})\prod\limits_{j\notin S^p_m}\hat{\mu}^\eps_{1,0}(z,\xi_j)\prod\limits_{l\notin T^q_n}\hat{\mu}_{1,0}^{\eps\ast}(z,\zeta_l)\\    &=\sum\limits_{m=0}^p\sum\limits_{n=0}^q\sum\limits_{S^p_m}\sum\limits_{T^q_n}\epsilon^{pq}_{mn}\sum\limits_{\kappa=1}^{K(m,n)}\prod\limits_{\gamma\in\Lambda_\kappa}[\hat{\mu}^\eps_{1,1}(z,\xi_{\gamma_1},\zeta_{\gamma_2})-\hat{\mu}_{0,1}^\eps(z,\xi_{\gamma_1})\hat{\mu}_{1,0}^{\eps\ast}(z,\zeta_{\gamma_2})]\prod\limits_{j\notin \Lambda_{\kappa,1}}\hat{\mu}^\eps_{1,0}(z,\xi_j)\prod\limits_{l\notin \Lambda_{\kappa,2}}\hat{\mu}_{1,0}^{\eps\ast}(z,\zeta_l)
\\&+\sum\limits_{m=0}^p\sum\limits_{n=0}^q\sum\limits_{S^p_m}\sum\limits_{T^q_n}\epsilon^{pq}_{mn}\prod\limits_{j=1}^p\hat{\mu}^\eps_{1,0}(z,\xi_j)\prod\limits_{l=1}^q\hat{\mu}_{1,0}^{\eps\ast}(z,\zeta_l)\,.
    \end{aligned}
\end{equation*}
Here, $\Lambda_{\kappa,1}$ denotes the set of all indices in the first position in the pairs in $\Lambda_{\kappa}$ and $\Lambda_{\kappa,2}$ denotes the set of all indices in the second position.  Since $\sum_{m=0}^p \sum_{S^p_m} (-1)^{p-m}=\sum_{m=0}^p \binom{p}{m}(-1)^{p-m}=0$, the last term above vanishes.
Let now $\Lambda_\kappa$ be fixed with $1\le k_1\le p\wedge q$ elements. This also fixes the indices $j\not\in \Lambda_{\kappa,1}$ and $l\not\in \Lambda_{\kappa,2}$ and hence the product term over $\gamma\in\Lambda_{\kappa}$ while $j\not\in \Lambda_{\kappa,1}$ and $l\not\in \Lambda_{\kappa,2}$. When $k_1<p\wedge q$, this product term appears multiplied by
\begin{equation*}
    \sum\limits_{m=k_1}^p\sum\limits_{n=k_1}^q\sum\limits_{S^{p-k_1}_{m-k_1}}\sum\limits_{T^{q-k_1}_{n-k_1}}(-1)^{p+q-m-n}=0
\end{equation*}
for the same reason as above.  Suppose now $k_1=p<q$ (or symmetrically $k_1=q<p$). Then such a product term appears multiplied by
\begin{equation*}
    \sum\limits_{m=k_1}^p\sum\limits_{n=k_1}^q\sum\limits_{S^{p-k_1}_{m-k_1}}\sum\limits_{T^{q-k_1}_{n-k_1}}(-1)^{p+q-m-n}=0
\end{equation*}
as well. Thus only remains the term with $k_1=p=q$, in which case $\Lambda_\kappa$ appears once. Noting that 
\begin{equation*}
    \begin{aligned}
    \sum\limits_{\Lambda_p}\prod\limits_{\gamma\in\Lambda_p}[\hat{\mu}^\eps_{1,1}(z,\xi_{\gamma_1},\zeta_{\gamma_2})-\hat{\mu}_{0,1}^\eps(z,\xi_{\gamma_1})\hat{\mu}_{0,1}^{\eps\ast}(z,\zeta_{\gamma_2})]=\Pi^\eps_{p,q}(z,v)\sum\limits_{\pi_p}\prod\limits_{j=1}^pU_\eta(U^\eps_{j,\pi_p(j)}-\mathbb{I})\hat{\mu}^\eps_{p,q}(0,v)\,,
    \end{aligned}
\end{equation*}
we obtain the form in \eqref{eqn:N_centered} for $\tilde{N}_{p,q}^\eps$. As an application of Theorem \ref{thm:Ueps}, we deduce that the remainder $\tilde E^\eps_{p,q}$ is indeed negligible as indicated in the corollary.
\end{proof}
\subsection{Approximation of moments in physical variables}\label{subsec:mu_pq_phys}
The error analysis was performed in the Fourier domain so far. In order to find approximations for moments in physical space, we need to undo the phase compensation followed by an inverse Fourier transform.

Let $\mu_{p,q}^\eps(0,X,Y)=\mu^\eps_0(X,Y)$ be a sequence of incident conditions such that as a measure in the Fourier variables, $\|\hat\mu_{p,q}^\eps(0)\|\leq C$ uniformly in $\eps$ and let us consider the equation
\begin{equation}\label{eq:mupqgal}
    (\partial_z - {\mathcal L}^\eps_{p,q}) \mu_{p,q}^\eps =0 ,\qquad \mu_{p,q}^\eps(0)=\mu^\eps_0
\end{equation}
with ${\mathcal L}^\eps_{p,q}$ defined in \eqref{eq:mupqpde}. 

This equation is solved in the Fourier variables for $\hat\mu_{p,q}^\eps(z,v)$. Using the decomposition $U^\eps=N_{p,q}^\eps+E_{p,q}^\eps$ and \eqref{eqn:phase_correc},  we directly obtain from Theorem \ref{thm:Ueps} the slight generalization of \eqref{eq:hatmupqsol} for arbitrary initial conditions:
\begin{equation*}
    \hat{\mu}_{p,q}^\eps(z,v)=\Pi^\eps_{p,q}(z,v) 
    \Big([N_{p,q}^\eps(z)\hat{\mu}_{p,q}^\eps(0)](z,v) + [E_{p,q}^\eps(z)\hat{\mu}_{p,q}^\eps(0)](z,v)\Big).
\end{equation*}
After inverse Fourier transform, we thus obtain
\begin{equation}\label{eq:decmueps}
        \begin{aligned}
            \mu_{p,q}^\eps(z,X,Y)&=\mathscr{N}_{p,q}^{\eps}(z,X,Y)+\mathscr{E}_{p,q}^{\eps}(z,X,Y)\,,
        \end{aligned}
\end{equation}
where 
\begin{equation}\label{eqn:N_tilde}
       \mathscr{N}_{p,q}^{\eps}(z,X,Y)= \int_{\Rm^{d(p+q)}} e^{i(\sum_{j=1}^p\xi_j\cdot x_j-\sum_{l=1}^q\zeta_l\cdot y_l)}\Pi^\eps_{p,q}(z,v)
       [N_{p,q}^\eps(z)\hat{\mu}_{p,q}^\eps(0)](z,v) \dfrac{\mathrm{d}\xi_1\cdots\mathrm{d}\xi_p\mathrm{d}\zeta_1\cdots\mathrm{d}\zeta_q}{(2\pi)^{pd+qd}},
\end{equation}
and $\mathscr{E}_{p,q}^{\eps}(z,X,Y)$ similarly defined replacing $N_{p,q}^\eps$ above by $E_{p,q}^\eps$.

We now define the functional giving higher-order moments of complex Gaussian variables in terms of first and second moments:
 \begin{equation}\label{eq:F}
 \mathscr{F}(h_1,\cdots,h_p,h'_1,\cdots,h'_q,g_{1,1},\cdots,g_{p,q})=\sum_{\kappa=1}^{K}\prod_{\gamma\in\Lambda_{\kappa}}[g_{\gamma}-h_{\gamma_1}h'_{\gamma_2}]\prod_{j\notin\Lambda_{\kappa,1}} \! \! h_{j}\prod_{l\notin\Lambda_{\kappa,2}} \! \! h'_{l} +\prod_{j=1}^ph_j\prod_{l=1}^qh'_l.
\end{equation}
Note that $\mathscr{F}$ is a bounded and continuous function of its $p+q+pq$ arguments.

The equivalent functional for centered Gaussian variables is defined as 
\begin{eqnarray}\label{eqn:bar_F}
    \widetilde{\mathscr{F}}(\tilde{g}_{1,1},\cdots,\tilde{g}_{p,q})=\begin{cases}
    0,\quad p\neq q\\
    \dsum_{\pi_p}\prod\limits_{j=1}^p\tilde{g}_{j,\pi_p(j)},\quad p=q\,,
    \end{cases}
\end{eqnarray}
where the sum is again over all permutations $\pi_p$ of the integers from $1$ to $p$.

For a function $f(X,Y)$ we define as usual $\|f\|_\infty=\sup_{(X,Y)\in\Rm^{(p+q)d}} |f(X,Y)|$. Then we have the following estimate for the solutions of \eqref{eq:mupqgal} and \eqref{eq:mupqpde}.

\begin{theorem}\label{thm:mupqphys}
    Let $\mu_{p,q}^\eps(z,X,Y)$ be the solution of \eqref{eq:mupqgal} with incident condition such that $\|\hat\mu_{p,q}^\eps(0)\|\leq C$ uniformly in $\eps$. Then the decomposition \eqref{eq:decmueps} holds with 
\begin{equation}\label{eq:bdmueps}
  \|\mu_{p,q}^\eps(z)\|_\infty \leq c(p,q,z)  \sup_{0<\eps\leq 1} \|\hat \mu_{p,q}^\eps(0)\|,\qquad 
  \|\mathscr{E}_{p,q}^\eps(z)\|_\infty \leq c(p,q,z) \eps^{\frac13}\, \sup_{0<\eps\leq 1}\|\hat \mu_{p,q}^\eps(0)\|.
\end{equation}

    Let us now consider $\mu^\eps_{p,q}(z,X,Y)$ as the solution to \eqref{eq:mupqpde} with incident conditions in the form of a Cartesian product as stated there. Then \eqref{eq:bdmueps} applies to $\mu^\eps_{p,q}(z,X,Y)$ with, moreover,
  \begin{equation*}
  \begin{aligned}
      \mu_{p,q}^{\eps }(z,X,Y)&=\mathscr{F}(\mu^\eps_{1,0}(z,x_1),\ldots,\mu^\eps_{1,0}(z,x_p),\mu^\eps_{0,1}(z,y_1),\ldots,\mu^\eps_{0,1}(z,y_q),\mu^\eps_{1,1}(z,x_1,y_1),\ldots,\mu^\eps_{1,1}(z,x_p,y_q))
      \\&
      +\mathscr{E}_{p,q}^{\eps}(z,X,Y)
  \end{aligned}
      \end{equation*}
  where $\mathscr{F}$ is defined in \eqref{eq:F}.
  
     The central moments defined in~\eqref{eqn:mu_bar_pq} are decomposed as
     \begin{eqnarray}\label{eqn:centred_mu_pq_approx}
         \tilde{\mu}^\eps_{p,q}(z,X,Y)=\widetilde{\mathscr{F}}(\tilde{\mu}^\eps_{1,1}(z,x_1,y_1),\cdots,\tilde{\mu}^\eps_{1,1}(z,x_p,y_q))+\tilde{\mathscr{E}}_{p,q}^\eps(z,X,Y)\,,
     \end{eqnarray}
     with $\widetilde{\mathscr{F}}$ defined in \eqref{eqn:bar_F} and for $0<\eps<\eps_0(z,p,q)$ uniformly bounded on compact sets,
     \begin{eqnarray}
         \sup\limits_{0\le z'\le z}|\tilde{\mathscr{E}}_{p,q}^\eps(z',X,Y)|\le \ \|\hat{\mu}^\eps_{p,q}(0)\|\ \eps^{\frac{1}{3}}.
     \end{eqnarray}
\end{theorem}
\begin{proof} By the decomposition \eqref{eq:decmueps}, we deduce that 
    \begin{equation*}
        \begin{aligned}
           \sup_{(X,Y)\in\mathbb{R}^{pd+qd}} |\mathscr{E}_{p,q}^{\eps}(z',X,Y)|&\le   
            \int_{\Rm^{d(p+q)}} |(E_{p,q}^\eps\hat{\mu}_{p,q}^\eps)(z',v)| \frac{\mathrm{d}v }{(2\pi)^{pd+qd}} \leq 
           \frac{1}{(2\pi)^{pd+qd}} \|E_{p,q}^\eps(z')\| \|\hat\mu^\eps_{p,q}(0) \| ,
          \end{aligned}
    \end{equation*}
    so that \eqref{eqn:E_bound} leads to the second estimate in \eqref{eq:bdmueps} both for $\eta(\eps)=1$ and $\eta^{-1}=\ln\ln\eps^{-1}$. The first estimate in \eqref{eq:bdmueps} is then a direct consequence of the uniform bound $\|N_{p,q}^\eps\|\leq c(p,q)$ in \eqref{eq:bdNeps}. 

    Let us now consider the setting of $\mu^\eps_{p,q}$ with an incident condition written as a Cartesian product form as in \eqref{eq:mupqpde}. Still denoting  $\mu^\eps_{p,q}=\mathscr{N}_{p,q}^{\eps}+\mathscr{E}_{p,q}^{\eps}$, then \eqref{eq:bdmueps} holds for $\mu^\eps_{p,q}$. Moreover, using the expansion for $N_{p,q}^\eps$ in~\eqref{eqn:N_def} and the definition of $\Pi_{p,q}^{\eps}$, we obtain the decomposition
    \begin{equation*}
        \begin{aligned}
            \mathscr{N}_{p,q}^{\eps}(z,X,Y)&=\sum_{\kappa=1}^{K}\int_{\Rm^{d(p+q)}}  \prod_{\gamma\in\Lambda_{\kappa}}e^{i(\xi_{\gamma_1}\cdot x_{\gamma_1}-\zeta_{\gamma_2}\cdot y_{\gamma_2})}
            e^{-\frac{iz\eta}{2k_0\eps}(|\xi_{\gamma_1}|^2-|\zeta_{\gamma_2}|^2)}
            U_{\eta}(U_{\gamma}^{\eps}-\mathbb{I})\\
            &\quad\times\prod_{j\notin\Lambda_{\kappa,1}}e^{i\xi_{j}\cdot x_{j}}e^{-\frac{iz\eta}{2k_0\eps}|\xi_{j}|^2}U_{\eta}^{\frac{1}{2}}\prod_{l\notin\Lambda_{\kappa,2}}e^{-i\zeta_{l}\cdot y_{l}}e^{\frac{iz\eta}{2k_0\eps}|\zeta_{l}|^2}U_{\eta}^{\frac{1}{2}}\hat{\mu}^\eps_{p,q}(0,v)\frac{dv}{(2\pi)^{pd+qd}} \\
            &\quad +\int_{\Rm^{d(p+q)}}  \prod_{j=1}^pe^{i\xi_{j}\cdot x_{j}}e^{-\frac{iz\eta}{2k_0\eps}|\xi_{j}|^2}U_{\eta}^{\frac{1}{2}}\prod_{l=1}^qe^{-i\zeta_{l}\cdot y_{l}}e^{\frac{iz\eta}{2k_0\eps}|\zeta_{l}|^2}U_{\eta}^{\frac{1}{2}}\hat{\mu}^\eps_{p,q}(0,v) \frac{dv}{(2\pi)^{pd+qd}}\,.
        \end{aligned}
    \end{equation*}
    When $\hat{\mu}^\eps_{p,q}(0,v)$ has the separable structure given in \eqref{eq:mupqpde}, then $\mathscr{N}_{p,q}^{\eps}$ takes the form
    \begin{equation*}
        \begin{aligned}
           \mathscr{N}_{p,q}^{\eps}(z,X,Y)=\sum_{\kappa=1}^{K}\int_{\Rm^{d(p+q)}}  \prod_{\gamma\in\Lambda_{\kappa}}e^{i(\xi_{\gamma_1}\cdot x_{\gamma_1}-\zeta_{\gamma_2}\cdot y_{\gamma_2})}[\hat{\mu}_{1,1}^\eps(z,\xi_{\gamma_1},\zeta_{\gamma_2})-\hat{\mu}^\eps_{1,0}(z,\xi_{\gamma_1}){\hat{\mu}_{0,1}^\eps}(z,\zeta_{\gamma_2})] \\
            \times\!\! \prod_{j\not\in\Lambda_{\kappa,1}} \!\! e^{i\xi_{j}\cdot x_{j}}\hat{\mu}_{1,0}^\eps(z,\xi_{j})\prod_{l\not\in\Lambda_{\kappa,2}}
            {\hat{\mu}_{0,1}^\eps}(z,\zeta_l)  \frac{e^{-i\zeta_{l}\cdot y_{l}}dv}{(2\pi)^{pd+qd}}
            +\int_{\Rm^{d(p+q)}}  \prod_{j=1}^pe^{i\xi_{j}\cdot x_{j}}
            \hat{\mu}^\eps_{1,0}(z,\xi_{j})\prod_{l=1}^q {\hat{\mu}_{0,1}^\eps}(z,\zeta_l)  \frac{e^{-i\zeta_{l}\cdot y_{l}}dv}{(2\pi)^{pd+qd}}\\
            =\mathscr{F}(\mu^\eps_{1,0}(z,x_1),\ldots,\mu^\eps_{1,0}(z,x_p),\mu^\eps_{0,1}(z,y_1),\ldots,\mu^\eps_{0,1}(z,y_q),\mu^\eps_{1,1}(z,x_1,y_1),\ldots,\mu^\eps_{1,1}(z,x_p,y_q))\,,
        \end{aligned}
    \end{equation*}
    where $\mu_{0,1}^\eps=\mu_{1,0}^{\eps*}$. The decomposition in~\eqref{eqn:centred_mu_pq_approx} is derived similarly using Corollary~\ref{corro:centred_mom}.
  \end{proof}
\begin{remark}(Gaussian summation rule) \label{rem:GSR}
If $\mathbf{Z}=(Z_1,\cdots,Z_N)$ is a circularly symmetric Gaussian random vector, all its moments can be completely described using second moments as~\cite{reed1962moment} 
\begin{equation}\label{eqn:Z_mom}
    \mathbb{E}[\prod_{j=1}^pZ_{s_j}\prod_{l=1}^qZ^\ast_{t_l}]=\begin{cases}
        0,\quad p\neq q\\
        \sum_{\pi_p}\prod_{j=1}^p\mathbb{E}[Z_{s_j}Z^\ast_{t_{\pi_p(j)}}],\quad p=q\,.
    \end{cases}
\end{equation}
Here, $s_j$ and $t_l$ are integers drawn from $\{1,\cdots,N\}$ for $1\le j\le p, 1\le l\le q$ and the sum is over all permutations $\pi_p$ of $\{1,\cdots,p\}$. This precisely corresponds to the arrangement in~\eqref{eqn:bar_F}.    
\end{remark}
We conclude this section with the following corollaries of Theorem \ref{thm:mupqphys}:
\begin{corollary}\label{cor:duhamel}
    In addition to the hypotheses of the above theorem, let $S^\eps(s,X,Y)$ be a measurable function such that $\hat S^\eps(s,v) \in C^0(\Rm_+;{\mathcal M}_B)(\Rm^{(p+q)d})$ with $\sup_{s\geq0}\|\hat S^\eps(s)\|\leq C$ uniformly in $\eps$. Let $\mu^\eps(z,X,Y)$ be the solution of
     \begin{equation}\label{eq:mupqgals}
    (\partial_z - {\mathcal L}^\eps_{p,q}) \mu^\eps =S^\eps(z) ,\qquad \mu^\eps(0)= \mu^\eps_0.
\end{equation}
Then we have
\begin{equation}\label{eq:bdduhamel}
 \|\mu^\eps(z)\|_\infty \leq c(p,q,z)  \big(\sup_{0<\eps\leq1}\|\hat \mu^\eps_0\|+\sup_{0\leq s\leq z; 0<\eps\leq1}\|\hat S^\eps(s)\| \big).
\end{equation}
\end{corollary}
Indeed, as a direct application of the Duhamel principle, we have
\begin{equation*}
\mu^\eps(z)=\tilde{U}^\eps(z,0)\mu^\eps_0+\int_{0}^z\tilde{U}^\eps(z,s)S^\eps(s)\mathrm{d}s\,,
\end{equation*}
where $\tilde{U}^\eps(z,s)$ is the solution operator to the homogeneous problem with source placed at $z=s$. This is bounded as
$\|\mu^\eps(z)\|_\infty\le\|\tilde{U}^\eps(z,0)\mu^\eps_0\|_\infty+z\sup_{0\le s\le z}\|\tilde{U}^\eps(z,s)S^\eps(s)\|_\infty$ and
Theorem \ref{thm:mupqphys} then provides the necessary bound.

The uniform bound on the moments $\mu^\eps_{p,q}(z,X,Y)$ of $u^\eps(z,x)$ directly translate into uniform bounds on the moments $m^\eps(z,X,Y)$ of $\phi^\eps(z,r,x)$ given in \eqref{eq:phieps}. We thus obtain from Theorem \ref{thm:mupqphys} the following corollary on the $p+q$th moments of the random vector $\Phi^\eps$.
\begin{corollary}\label{coro:M_pq_moments}
For the random vector $\Phi^\eps$ with elements $\{\phi^\eps_j\}_{j=1}^N$,   define the tensor
    \begin{equation*}
        \mathbf{M}_{p,q}^\eps(z,r,X)=\mathbb{E}[\underbrace{\Phi^\eps\otimes\cdots\otimes\Phi^\eps}_{p \text{ terms}}\otimes\underbrace{{\Phi^{\eps\ast}}\otimes\cdots\otimes{\Phi^{\eps\ast}}}_{q \text{ terms}}](z,r,X)\,.
    \end{equation*}
    Then, the elements of $\mathbf{M}_{p,q}^\eps$ are given by
    \begin{equation*}
        \mathbb{E}[\prod_{j=1}^p\phi^\eps_{s_j}\prod_{l=1}^q{\phi^{\eps\ast}_{t_l}}]=\mathscr{F}\big(m^\eps_{1,0}(z,{x_{s_j}})_j,m^\eps_{0,1}(z,{x_{t_l}})_l,(m^\eps_{1,1}(z,{x_{s_j}},{x_{t_l}}))_{j,l}\big) + \mathcal{O}(c(p,q,z)\eps^{\frac13})\,,
    \end{equation*}
    where $s_j$ and $t_l$ are integers in $\{1,\cdots,N\}$. Moreover, the tensor of central moments
        \begin{equation*}
        \widetilde{\mathbf{M}}_{p,q}^\eps(z,r,X)=\mathbb{E}[\underbrace{\tilde{\Phi}^\eps\otimes\cdots\otimes\tilde{\Phi}^\eps}_{p \text{ terms}}\otimes\underbrace{{\tilde{\Phi}^{\eps\ast}}\otimes\cdots\otimes{\tilde{\Phi}^{\eps\ast}}}_{q \text{ terms}}](z,r,X)
    \end{equation*}
    has elements given by
    $\mathbb{E}[\prod_{j=1}^p\tilde{\phi}^\eps_{s_j}\prod_{l=1}^q{\tilde{\phi}^{\eps\ast}_{t_l}}]=\widetilde{\mathscr{F}}(\tilde{m}^\eps_{1,1}(z,{x_{s_j}},{x_{t_l}})_{j,l}) + \mathcal{O}(c(p,q,z)\eps^{\frac13})$.
\end{corollary}

\section{Statistics in the scintillation regime}\label{sec:final}
We are now in a position to prove the main theorems stated in section \ref{sec:process}.
    \begin{proof}[Proof of Theorem \ref{thm:kinetic}.]
    As the second moments of the random variables $\{\tilde{\Phi}^\eps\}$ are uniformly bounded in $\eps$ thanks to Lemma \ref{lem:limmts}, the family of distributions $\{\tilde{\Phi}^\eps\}$ is tight \cite{billingsley2017probability}. From Corollary \ref{coro:M_pq_moments}, Lemma \ref{lem:limmts}, and the continuity of $\widetilde{\mathscr{F}}$, we deduce that the moments of $\tilde{\Phi}^\eps$ converge to those of $\tilde{\Phi}$ in all settings considered (kinetic and diffusive regime, incident profiles \eqref{eq:u0} and \eqref{eq:u0plane}, as well as $\beta=1$ and $\beta>1$).

The limiting vector-valued variable $\tilde{\Phi}$ has complex Gaussian distribution. As an application of a Carleman criterion \cite{billingsley2017probability}, such variables are characterized by their moments since the latter of order $p+q$ do not grow faster than $C^{p+q}e^{(p+q)\ln(p+q)}$. This shows that the limiting distribution $\tilde{\Phi}$ is uniquely characterized as circularly symmetric complex Gaussian and that the whole sequence $\tilde{\Phi}^\eps$ converges in distribution to $\tilde{\Phi}$ \cite{billingsley2017probability}.

The results of Remark \ref{rem:cv} are proved similarly. For an incident beam profile \eqref{eq:u0}, all moments $m^\eps_{p,q}$ in \eqref{eq:mupq} converge to their limit as $\eps\to0$. We may thus replace $\tilde{\Phi}^\eps$ by ${\Phi}^\eps$ and use the continuity of $\mathscr{F}$to obtain the results stated in the remark.
\end{proof}

\begin{proof}[Proof of Theorem \ref{thm:diffusive}]
        The proof is similar to that of Theorem \ref{thm:kinetic}. The mean value $\mathbb{E}[\Phi^\eps]$ drops to zero when $\eta=(\ln\ln(1/\eps))^{-1}\ll1$ from~\eqref{eqn:mu_1_limit} and we have that the limiting distribution of $\Phi^\eps$ is indeed a mean zero circularly symmetric Gaussian random vector.
    \end{proof}
The following results are a direct consequence of Theorem~\ref{thm:diffusive}.

\begin{proof}[Proof of Corollary \ref{cor:scint}]
From Theorem \ref{thm:diffusive} and noting that $\sum\limits_{\pi_p}=p!$, we have
\begin{equation*}
    \mathbb{E}[I^\eps(z,r,x)^p]=\mathbb{E}[|\phi^\eps(z,r,x)|^{2p}]\to p!M_{1,1}^p(z,r,x,x)\,.
\end{equation*}
This immediately gives that $I^\eps(z,r,x)$ converges to an exponentially distributed random variable $I(z,r)$ as stated in the Corollary. In particular, the scintillation index of the random vector $\phi^\eps$ given by
\begin{equation*}
    \sS^\eps(z,r,x)=\frac{\mathbb{E}[I^\eps(z,r,x)^2]-\mathbb{E}[I^\eps(z,r,x)]^2}{\mathbb{E}[I^\eps(z,r,x)]^2}\to \frac{\mathbb{E}[I(z,r)^2]-\mathbb{E}[I(z,r)]^2}{\mathbb{E}[I(z,r)]^2}:=\sS(z,r)\,.
\end{equation*}
Noting that $\mathbb{E}[I(z,r)^2]=2\mathbb{E}[I(z,r)]^2>0$ strictly positive for $z>0$ as a solution of a heat equation with non-negative sources, we have that $\sS(z,r)=1$ identically.
\end{proof}

    \begin{proof}[Proof of Corollary \ref{cor:selfaver}]

   {\color{black} We have that}
    \begin{equation*}
    \begin{aligned}
                \mathbb{E}[|I^{\eps}_D(z,r,x)|^2]&=\frac{1}{|D|^2}\int_{D\times D}\mathbb{E}[I^\eps(z,r+r_1,x)I^\eps(z,r+r_2,x)]\mathrm{d}r_1\mathrm{d}r_2\\
                &=\frac{1}{|D|^2}\int_{D\times D}\mathbb{E}[I^\eps(z,r,x+\eps^{-\beta}\eta^{-1}r_1)I^\eps(z,r,x+\eps^{-\beta}\eta^{-1}r_2)]\mathrm{d}r_1\mathrm{d}r_2{\color{black}\,.}
    \end{aligned}
    \end{equation*}
       
    {\color{black} Note that from Corollary~\ref{coro:M_pq_moments}, 
    \begin{equation*}
    \begin{aligned}
      &\mathbb{E}I^\eps(z,r,x+\eps^{-\beta}\eta^{-1}r_1)I^\eps(z,r,x+\eps^{-\beta}\eta^{-1}r_1)\\
      &=\mathbb{E}[|\phi^\eps(z,r,x+\eps^{-\beta}\eta^{-1}r_1)|^2|\phi^\eps(z,r,x+\eps^{-\beta}\eta^{-1}r_1)|^2]= m_{1,1}^\eps(z,r+r_1,x,x)m_{1,1}^\eps(z,r+r_2,x,x)\\
      &+m_{1,1}^\eps(z,r,x+\eps^{-\beta}\eta^{-1} r_1,x+\eps^{-\beta}\eta^{-1} r_2)m_{1,1}^\eps(z,r,x+\eps^{-\beta}\eta^{-1} r_2,x+\eps^{-\beta}\eta^{-1} r_1)+\mathcal{O}(\eps^{\frac13})\,,
    \end{aligned}
    \end{equation*}
    with the cross term $m_{1,1}^\eps(z,r,x+\eps^{-\beta}\eta^{-1} r_j,x+\eps^{-\beta}\eta^{-1} r_l)\to 0$ as $\eps^{-\beta}\eta^{-1}|r_j-r_l|\to\infty$ for $j\neq l$. This gives 
   \begin{equation*}
   \begin{aligned}
         \mathbb{E}[|I^{\eps}_D(z,r,x)|^2] \xrightarrow{\eps\to 0} &\lim_{\eps\to 0}\frac{1}{|D|^2}\int_{D\times D} 
 m_{1,1}^\eps(z,r+r_1,x,x)m_{1,1}^\eps(z,r+r_2,x,x) \mathrm{d}r_1\mathrm{d}r_2\\
 =&\lim_{\eps\to 0}\Big(\frac{1}{|D|}\int_{D} m_{1,1}^\eps(z,r+r_1,x,x)\mathrm{d}r_1\Big)^2.
   \end{aligned}
   \end{equation*}
   When $a_\eps=a$, this gives 
   \begin{equation*}
       \mathbb{E}[|I^{\eps}_D(z,r,x)|^2] \xrightarrow{\eps\to 0}\Big(\frac{1}{|D|}\int_{D} \mathbb{E}I(z,r+r_1)\mathrm{d}r_1\Big)^2.
   \end{equation*}
   Now suppose $a_\eps\to 0$. We show the proof only for sources of the form~\eqref{eq:u0} with $\beta=1$. The other cases can be treated in a similar manner. From~\eqref{eqn:meps11_f0},
   \begin{equation*}
       \begin{aligned}
           &\frac{1}{|D|}\int_{D} m_{1,1}^\eps(z,r+r_1,x,x)\mathrm{d}r_1\\
           &=\int_{\mathbb{R}^{2d}}\Big(\frac{1}{|D|}\int_{D}e^{ir_1\cdot(\xi-\zeta)}\mathrm{d}r_1\Big)e^{ir\cdot(\xi-\zeta)}e^{i\eta\eps x\cdot(\xi-\zeta)}e^{-\frac{iz\eta\eps}{2k_0}(|\xi|^2-|\zeta|^2)}\hat{\sff}(\xi)\hat{\sff}^\ast(\zeta)\mathcal{Q}(0,\eta(\xi-\zeta))\frac{\mathrm{d}\xi\mathrm{d}\zeta}{(2\pi)^{2d}}\\
           &=\int_{\mathbb{R}^{2d}}\Big(\prod_{j=1}^d \mathrm{sinc}\big(\frac{a_\eps(\xi_j-\zeta_j)}{2}\Big)e^{ir\cdot(\xi-\zeta)}e^{i\eta\eps x\cdot(\xi-\zeta)}e^{-\frac{iz\eta\eps}{2k_0}(|\xi|^2-|\zeta|^2)}\hat{\sff}(\xi)\hat{\sff}^\ast(\zeta)\mathcal{Q}(0,\eta(\xi-\zeta))\frac{\mathrm{d}\xi\mathrm{d}\zeta}{(2\pi)^{2d}}\,,
           \end{aligned}
   \end{equation*}
   where $\mathrm{sinc}(x)=\frac{\sin x}{x}$. Since for fixed $(\xi_j,\zeta_j)$, $\lim_{a_\eps\to 0}\mathrm{sinc}\big(\frac{a_\eps(\xi_j-\zeta_j)}{2}\Big)\to 1$, the proof is completed by noting that
   \begin{equation*}
       \frac{1}{|D|}\int_{D} m_{1,1}^\eps(z,r+r_1,x,x)\mathrm{d}r_1 \xrightarrow{\eps\to 0}\mathbb{E}I(z,r)\,.
   \end{equation*}
   }
\end{proof}
We now come to the proof of tightness and stochastic continuity of the process $x\mapsto\phi^\eps(z,r,x)$.
\begin{proof}[Proof of Theorem \ref{thm:tightness}.]
We break the derivation of the result into several steps.
\paragraph{Equation for $\phi^\eps$. }
For $r$ fixed, we introduce the dilation-translation operator
\[
   \phi^\eps(z,r,x) = u^\eps(z,\eps^{-\beta}r+\eta x) = ({\mathcal S}^\eps u^\eps) (z,x).
\]
The moments of $\phi^\eps$ and $u^\eps$ are then related by
\[
 m_{p,q}^\eps(z,r,X,Y) =: {\mathcal S}^\eps_{p,q} \mu_{p,q}^\eps(z,X,Y),
\]
which defines ${\mathcal S}^\eps_{p,q}$ implicitly.
We observe that in the Fourier domain,
\begin{equation}\label{eq:mum}
  \|\hat m_{p,q}^\eps(z,r,\cdot)\| = \|\hat \mu_{p,q}^\eps(z,\cdot)\|.
\end{equation}
Let us define
\begin{equation}\label{eq:scalingL}
  {\mathfrak L}^\eps_{p,q} = {\mathcal S}^\eps_{p,q}  {\mathcal L}^\eps_{p,q}  ({\mathcal S}^\eps_{p,q})^{-1}
\end{equation}
where ${\mathcal L}^\eps_{p,q}$ is the operator in \eqref{eq:mupqpde} defining the moments $\mu^\eps_{p,q}$.
As a consequence, we deduce from Corollary \ref{cor:duhamel} that the solution of 
\begin{equation}\label{eq:meps}
  (\partial_z - {\mathfrak L}^\eps_{n,n} )m^\eps = S^\eps,\quad m^\eps(0)=m_0
\end{equation}
for any $n\geq1$ admits a unique solution that satisfies
\begin{equation}\label{eq:bdmeps}
  \|m^\eps(z,\cdot)\|_\infty + \|\hat m^\eps(z,\cdot)\| \leq C (\|\hat m_0\| + \sup_{0\leq s\leq z} \|\hat S^\eps(s)\|)
\end{equation}
with a bound uniform in $\eps$ both in the kinetic and diffusive regimes.   This is simply recasting the results for $\mu^\eps_{p,q}$ in section \ref{sec:highermoments} to moments $m^\eps_{p,q}$ of the process $\phi^\eps$. 

\paragraph{Equations for finite differences.}
  
Let $h\in\Rm^d$ with $|h|\leq1$ and $n\geq1$. We denote throughout
\[
  \delta f(x) \equiv \delta_h f(x) := f(x+h) - f(x),\qquad f_h(x) := f(x+h).
\]  
Let $X=(x_1,\ldots,x_{2n}) \in \Rm^{2nd}$. To simplify notation, we denote $\varphi^\eps_j(z,x_j)=\phi^\eps(z,r,x_j)$ with $\epsilon_j=1$ for $1\leq j\leq n$ and $\varphi_j^\eps(z,x_j)={\phi^{\eps*}}(z,r,x_j)$ with $\epsilon_j=-1$ for $n+1\leq j\leq 2n$. 

From \eqref{eq:uSDEd} conjugated by ${\mathcal S}^\eps$, we deduce the equation (with $k_0$ set to $1$ to simplify notation)
\begin{equation}\label{eq:varphi}
  d\varphi^\eps_j = \big(\frac{i\epsilon_j}2\frac{\eta^3}{\eps} \Delta_{x_j} -\frac{R^\eps(0)}{8} \big)\varphi^\eps_j dz + \frac i2 \epsilon_j \varphi^\eps_j dB^\eps
\end{equation}
where we have defined (dropping the irrelevant shift by $\eps^{-\beta}r$ in $B^\eps$ since it will not appear in $R^\eps$)
\[ 
  B^\eps(x)= \frac 1\eta B(\eta x),\qquad R^\eps(x)=\frac{1}{\eta^{2}}R(\eta x).
\] 
With this notation we obtain for the finite differences that 
\begin{equation}\label{eq:deltavarphi}
 d\delta\varphi^\eps_j =  \big(\frac{i\epsilon_j}2 \frac{\eta^3}{\eps} \Delta_{x_j} -\frac{R^\eps(0)}{8} \big)  \delta\varphi^\eps_j dz + \frac {i\epsilon_j}2 (\delta\varphi^\eps_j dB^\eps + \varphi^\eps_{jh} d \delta B^\eps).
\end{equation}
It seems difficult to analyze $\E|\delta \varphi^\eps(z,x)|^{2n}$ directly. We can, however, see the latter object as a $2n-$ moment in $2n-$ variables  $X$. 

Let $\sigma$ be a permutation of the $2n$ variables in $X$. We define the products
\begin{eqnarray}\label{eqn:del_Psi}
       \delta_h^q\Psi^\eps_\sigma(z,X,h) := \prod_{j\in J_1} \varphi_j^\eps(z,\sigma(x_j)) \prod_{l\in J_2} \delta\varphi_l^\eps (z,\sigma(x_l)),
\end{eqnarray}
with $J_1=\{1,\ldots,p\}$ and $J_2=\{p+1,\ldots,p+q\}$ for $0\leq q\leq 2n$ and $p+q=2n$. In other words, $p$ terms in the above product are of the form $\phi^\eps$ or ${\phi^{\eps*}}$ while the remaining $q$ terms are of the form  $\delta \phi^\eps$ or $\delta {\phi^{\eps*}}$. 
We observe that for $\sigma$ the identity, we have our primary object of interest
\[
   \delta_h^{2n}\Psi^\eps(z,(x,\ldots, x),h) =  |\delta \phi^\eps(z,r,x)|^{2n}.
\]
We will devise an equation for $\E\delta_h^{2n}\Psi^\eps(z,X,h)$, which, however, will involve source terms $\E\delta_h^q\Psi^\eps_\sigma(z,X,h)$. The latter terms then also satisfy similar equations with source terms involving similar terms with different values of $q$ and $\sigma$. After an inductive process, we will arrive at \eqref{eq:tightness}, at least in the kinetic regime $\eta(\eps)=1$. In the diffusive regime, we will obtain the $\eta-$dependent bound \eqref{eq:tighteta} and a further step will be necessary.

To simplify notation, we drop the dependence in $\sigma$, i.e., replace $\sigma(x_j) \to x_j$. We assume that the coefficients $\epsilon_j$ are permuted accordingly when looking for equations involving $\varphi^\eps_j(z,x_j)$. We next compute 
\begin{eqnarray} \nonumber
  d \delta_h^q\Psi^\eps &=& \dsum_{j\in J_1} d \varphi^\eps_j(z,x_j)  \prod_{j'\in \tilde J_1} \varphi_{j'}^\eps(z,x_{j'}) \prod_{l\in J_2} \delta\varphi_l^\eps (z,x_l) + \sum_{l\in J_2} d  \delta\varphi^\eps_l(z,x_l)  \prod_{j\in  J_1} \varphi_j^\eps(z,x_j) \prod_{l'\in \tilde J_2} \delta\varphi_{l'}^\eps (z,x_{l'})
  \\[0mm] \nonumber
  & + & \dfrac{1}2 \dsum_{j\neq j'\in J_1} d \varphi^\eps_j (z,x_j)d\varphi^\eps_{j'} (z,x_{j'}) \prod_{j''\in \tilde J_1} \varphi_{j''}^\eps(z,x_{j''}) \prod_{l\in J_2} \delta\varphi_l^\eps (z,x_l) \\[0mm]\nonumber 
  & + & \dfrac{1}2\dsum_{l\neq l'\in J_2} d \delta\varphi^\eps_l(z,x_l)d \delta\varphi^\eps_{l'}(z,x_{l'})   \prod_{j\in J_1} \varphi_j^\eps(z,x_j) \prod_{l''\in \tilde J_2} \delta\varphi_{l''}^\eps (z,x_{l''})\\
&+&  \dsum_{j\in J_1} \dsum_{l\in J_2} d\varphi^\eps_j(z,x_j)d\delta \varphi^\eps_l(z,x_l)   \prod_{j'\in \tilde J_1} \varphi_{j'}^\eps(z,x_{j'}) \prod_{l'\in \tilde J_2} \delta\varphi_{l'}^\eps (z,x_{l'}),
\label{eq:dPsi}
\end{eqnarray}
where $\tilde J_l$ involves the set deleting any index that already appeared to the left in any product of $p+q$ terms (and hence could be empty in which case the corresponding term vanishes). We have used here the It\^o calculus anticipating from \eqref{eq:varphi} and \eqref{eq:deltavarphi} that all the above terms possibly contribute.

It\^o calculus then provides an equation for the moments of $\delta_h^q\Psi^\eps$ as follows.  We first compute
\[\begin{array}{rcl}
  \E dB^\eps(x_j) dB^\eps(x_l)&=&R^\eps(x_l-x_j){\color{black}\mathrm{d}z} = \eta^{-2} R(\eta(x_l-x_j)){\color{black}\mathrm{d}z} \\[1mm] 
  \E dB^\eps(x_j) d\delta B^\eps(x_l) &=& \delta_{h}R^\eps_{lj}(X;h){\color{black}\mathrm{d}z}:= {\color{black}[}R^\eps(x_l+h-x_j)-R^\eps(x_l-x_j) {\color{black}]\mathrm{d}z}, \\[1mm]
  \E d\delta B^\eps(x_j) d\delta B^\eps(x_l) &=& \delta^2_h R^\eps_{jl}(X;h){\color{black}\mathrm{d}z} := {\color{black}[}2R^\eps(x_l-x_j) - R^\eps(x_l+h-x_j)-R^\eps(x_l-x_j-h){\color{black}]\mathrm{d}z}.
\end{array}\]
All terms in \eqref{eq:dPsi} of the form $d\varphi^\eps_j$, $d\delta\varphi^\eps_j$, $d\varphi^\eps_jd\varphi^\eps_l$ contribute terms that are the same as in the derivation of the operator ${\mathcal L}^\eps_{n,n}$ in \eqref{eq:mupqpde}.  The remaining terms produce the following equation with source terms:
\begin{equation}\label{eq:deltameps}\begin{array}{rcl}
   (\partial_z- {\mathfrak L}^\eps_{n,n}) \delta_h^q m^\eps &=& \dfrac{-1}8 \dsum_{l\neq l'\in J_2}  \epsilon_l\epsilon_{l'} \delta^2_h R^\eps_{ll'} \   \E \varphi^\eps_{lh}(z,x_l)\varphi^\eps_{l'h}(z,x_{l'}) \prod_{j\in  J_1}\varphi_j^\eps(z,x_j) \prod_{l''\in \tilde J_2}  \delta\varphi_{l''}^\eps (z,x_{l''})\\
&& -\dfrac{1}4  \dsum_{j\in J_1} \dsum_{l\in J_2}   \epsilon_j\epsilon_l \delta_h R^\eps_{lj} \   \E   \varphi^\eps_{j}(z,x_j)\varphi^\eps_{lh}(z,x_l) \prod_{j'\in \tilde J_1} \varphi_{j'}^\eps(z,x_{j'}) \prod_{l'\in \tilde J_2} \delta\varphi_{l'}^\eps (z,x_{l'})\\
&& -\dfrac{1}4  \dsum_{l\neq l'\in J_2}  \epsilon_l\epsilon_{l'} \delta_h R^\eps_{ll'} \   \E   \varphi^\eps_{lh}(z,x_l) \delta\varphi^\eps_{l'}(z,x_{l'}) \prod_{j\in  J_1} \varphi_{j}^\eps(z,x_{j}) \prod_{l''\in \tilde J_2} \delta\varphi_{l''}^\eps (z,x_{l''}),
\end{array}\end{equation}
where we have defined $\delta_h^q m^\eps = \E \delta_h^q\Psi^\eps$.

This is an equation for $\delta_h^qm^\eps$ involving a finite sum of source terms, which up to a permutation of the variables $X$ and a shift of some of them by $h$ that do not change the TV norm in the Fourier variables, have either the form
\[
   \delta_h R^\eps_{12}  \delta_h^{q-1}m^\eps,\quad \delta_h^{q-1}m^\eps = \E \prod_{j\in \tilde J_1}\varphi_j^\eps(z,x_j) \prod_{l\in \tilde J_2}  \delta\varphi_l^\eps (z,x_l),
\]
or the form
\[
  \delta^2_h R^\eps_{12}  \delta_h^{q-2}m^\eps,\quad \delta_h^{q-2}m^\eps = \E \prod_{j\in \check J_1}\varphi_j^\eps(z,x_j) \prod_{l\in \check J_2}  \delta\varphi_l^\eps (z,x_l).
\]
In the above expressions, $|\check J_2|=|J_2|-2$ (unless $|J_2|\leq1$ and this term does not appear) while $|\check J_1|=|J_1|+2$. Similarly, $|\tilde J_2|=|J_2|-1$ (unless $|J_2|=1$ and this term does not appear) while $|\tilde J_1|=|J_1|+1$.

\paragraph{Control of source terms.}   Each one of the source terms in \eqref{eq:deltameps} is of the form $\delta^j_h R^\eps(x-y)$ for $j=1,2$ times a function of the form $\varphi(x,y,X')$ with $X'$ a set of variables in $\Rm^{2(n-1)d}$. These products are controlled thanks to the following lemma.
Define
\begin{equation}\label{eq:deltahR}\begin{array}{rcl}
  \delta_h R^\eps(\tau;h) &=& \dfrac1{\eta^2} ( R(\eta(\tau+h)) - R(\eta \tau))
\\
  \delta^2_h R^\eps(\tau;h) &=& \dfrac1{\eta^2} ( 2R(\eta\tau) - R(\eta(\tau+h)) - R(\eta (\tau-h))).
\end{array}\end{equation}
We consider the operators which to $\varphi(x,y)$ associate $\delta^j_h R^\eps(x-y;h)\varphi(x,y)$ for $j=1,2$, and  show that they are bounded in the TV sense in dual variables. 
\begin{lemma}\label{lem:bddhR}
  Let $\psi_1(x,y)=\delta_h R^\eps(x-y;h)\varphi(x,y)$ and $\psi_2(x,y)=\delta_h^2 R^\eps(x-y;h)\varphi(x,y)$.
  
  Assume that for $\alpha\in (0,1]$, $\aver{k}^{2\alpha} \hat R(k) \in \sL^1(\Rm^d)$. Then 
  \[
     \|\hat \psi_1\|  \leq C \frac{|h|^\alpha}{\eta^{2-\alpha}} \|\hat\varphi\|,\qquad \|\hat\psi_2\|  \leq C \frac{|h|^{2\alpha}}{\eta^{2-2\alpha}} \|\hat\varphi\|,
  \]
  with $C$ uniform in $\eps\in(0,1]$ and $h\in(0,1]$.
\end{lemma}
\begin{proof}
  We first observe that 
  \[
     |e^{ih\cdot k} -1| \leq 2 \wedge C |h| |k|\quad \mbox{ and hence } \quad  |e^{ih\cdot k} -1| \leq C |h|^{\alpha} |k|^{\alpha},
  \]
  \[
     |e^{ih\cdot k} + e^{-ih\cdot k} -2| \leq 4 \wedge C|h|^2 |k|^2\quad \mbox{ and hence } \quad  |e^{ih\cdot k} + e^{-ih\cdot k} -2|  \leq C |h|^{2\alpha} |k|^{2\alpha}.
  \]
  We compute in the Fourier domain
  \[
    \hat \psi_j(\xi,\zeta) = \dint_{\Rm^d}  \widehat{\delta_h^jR^\eps}(k) \hat\varphi(\xi-k,\zeta-k) \dfrac{dk}{(2\pi)^d},
  \]
  for $j=1,2$, so that as already used earlier,
  $ \|\hat \psi_j \| \leq C \|\widehat{\delta_h^jR^\eps}\|_{1} \|\hat \varphi\|$. Now
  \[
     |\widehat{\delta_hR^\eps}(k) |=| e^{ih\cdot k}-1| \eta^{-2} \eta^{-d} \hat R(\eta^{-1}k) \leq C \eta^{-2+\alpha} |h|^\alpha \eta^{-d}  |\eta^{-1} k|^{\alpha} \hat R(\eta^{-1}k).
  \]
  When $\eta=1$, we deduce that the above is integrable and independent of $\eps$ with bound of order $|h|^\alpha$. When $\eta\to0$, the bound cannot be independent of $\eta$ even if we choose $\alpha=1$. 
  Similarly,
  \[
     |\widehat{\delta_h^2R^\eps}(k)| =|2-e^{ih\cdot k}-e^{-ih\cdot k}| \eta^{-2} \eta^{-d} \hat R(\eta^{-1}k) \leq C \eta^{-2+2\alpha} |h|^{2\alpha} \eta^{-d}  |\eta^{-1}k|^{2\alpha} \hat R(\eta^{-1}k).
  \]
  This yields the result after integration in $k\in\Rm^d$.
\end{proof}
\paragraph{Iterative stability.} Combining Lemma \ref{lem:bddhR} and the stability result \eqref{eq:bdmeps} for solutions of \eqref{eq:meps} with the equation \eqref{eq:deltameps}, we deduce that for $q\geq1$ and $\delta^q_h m^\eps_\sigma := \E \delta^q_h\Psi^\eps_\sigma$, then
\begin{equation}\label{eq:stabmts}
 \sup_{\sigma,\,0\leq s\leq z} \| \widehat{\delta^q_h m ^\eps_\sigma}(s)\| \leq C(z)   \sup_\sigma 
 \Big(  \frac{(\eta|h|)^\alpha}{\eta^2} \sup_{s\in[0,z]} \| \widehat{\delta_h^{q-1} m_\sigma^\eps}(s)\| + \frac{(\eta|h|)^{2\alpha}}{\eta^2} 
 \sup_{s\in[0,z]} \| \widehat{\delta_h^{q-2} m_\sigma^\eps}(s)\| +  \| \widehat{\delta^q_h\mu^\eps_\sigma}(0)\| \Big).
\end{equation}
Iterating over $q$ from $2n$ to $0$ gives the worst-case estimate:
\begin{equation}\label{eq:stabmtsn}
 \sup_{\sigma,\,0\leq s\leq z} \|  \widehat{\delta^{2n}_h m^\eps_\sigma}(s)\| \leq \frac{C(z)}{\eta^{2n(2-\alpha)}} \Big(   \sup_\sigma
\dsum_{q=0}^{2n} |h|^{(2n-q)\alpha}  \| \widehat{\delta_h^q m_\sigma^\eps}(0)\| + |h|^{2n\alpha}  \sup_{s\in[0,z] }\|\hat m^\eps(s)\|\Big). 
\end{equation}
We deduce from Theorems \ref{thm:kinetic} and \ref{thm:diffusive} and \eqref{eq:mum} that $\sup_{s\in[0,z] }\|\hat m^\eps(s)\|\leq C\|\hat u_0^\eps\|^{2n}$ is bounded uniformly in $\eps$.
It remains to analyze the incident beam moment $\| \widehat{\delta_h^q m_\sigma^\eps}(0)\|$. Consider one of the terms in the incident beam $\phi^\eps(0,r,x)$ given by $\phi^\eps_\sm(x):=\sff_\sm(r+\theta x)e^{i \sk_\sm\cdot (\eps^{-\beta}r+\eta x)}$ where we use the shorthand $\theta=\eta\eps^\beta$ to simplify. We find
\[ \hat \phi^\eps_\sm(k) = e^{i{k}\cdot r \theta^{-1}} \theta^{-d} \hat \sff_\sm(\theta^{-1}(k-\eta \sk_\sm)),\]
so that 
\[ \widehat{\delta_h \phi^\eps_\sm}(k) = (e^{ik\cdot h}-1) e^{i{k}\cdot r \theta^{-1}} \theta^{-d} \hat \sff_\sm(\theta^{-1}(k-\eta \sk_\sm)).\]
For $\theta\leq1$ and $\eta \sk_\sm$ uniformly bounded, we find that 
\[ \aver{k}^{2\alpha} \leq C \aver{k-\eta \sk_\sm}^{2\alpha} \leq  C\aver{\theta^{-1}(k-\eta \sk_\sm)}^{2\alpha}\]
so that, using that $|e^{ik\cdot h}-1|\aver{k}^{-{2\alpha}}\leq C|h|^\alpha$,
\[ | \widehat{\delta_h \phi^\eps_\sm}(k)| \leq C |h|^{\alpha} \theta^{-d} \aver{\theta^{-1}(k-\eta \sk_\sm)}^{2\alpha} |\hat \sff_\sm(\theta^{-1}(k-\eta \sk_\sm))|. \]
By assumption on the incident beam that $\| \aver{k}^{2\alpha}\hat \sff_\sm \|\leq C_\alpha$, we deduce that 
\[
\| \widehat{\delta_h \phi^\eps_\sm} \| \leq C C_\alpha |h|^{\alpha}.
\]
Since $\widehat{\delta_h^q m_\sigma^\eps}(0)$ involves a finite sum of products including $q$ terms of the form $\widehat{\delta_h \phi^\eps_\sm}$ as above, we therefore obtain that
\[
   \sup_\sigma  \| \widehat{\delta_h^q m_\sigma^\eps}(0)\| \leq C |h|^{q\alpha} . 
\]
These estimate combined with \eqref{eq:stabmtsn} show that 
\begin{equation}\label{eq:tighteta}
 \sup_{s\in[0,z]} \| \delta^{2n}_h m^\eps(s)\|_\infty \leq \sup_{s\in[0,z]} \| \widehat{\delta^{2n}_h m^\eps}(s)\|  \leq C \frac{|h|^{2n\alpha}}{\eta^{2n(2-\alpha)}}.
\end{equation}
When $\eta=1$ in the kinetic regime, the proof of the theorem is complete since \eqref{eq:tighteta} is exactly \eqref{eq:tightness} in that case as we may choose $n$ such that $2n\alpha\geq d+2\alpha_- n$, ensuring stochastic continuity and tightness of $\phi^\eps$ as a field with values in $C^{0,\alpha_-}(\Rm^d)$ \cite{kunita1997stochastic}.

\paragraph{Approximate Gaussianity for sufficiently large $h$.} When $\eta\to0$, the loss of a term $\eta$ in the analysis of $\delta_hR^\eps$ implies that the bound is useful only for $|h|$ sufficiently small compared to $\eta$. {\color{black} So a proof as in the previous section leading to~\eqref{eq:tighteta} will not be sufficient to show tightness when $\eta\ll 1$. However, heuristically the asymptotic limit is Gaussian, which is smooth in $x$.} To complete the proof of tightness in the diffusive regime, we now use the almost complex Gaussian statistics of $\phi^\eps$ to estimate \eqref{eq:tightness} when $h$ is sufficiently large. 

Indeed, for $(r,s)$ fixed, we proved in Theorem \ref{thm:mupqphys} that $\mu^\eps_{p,q}$ uniformly, or equivalently $m^\eps_{p,q}$ uniformly, was given by a functional of its first and second moments up to a term bounded by $C\eps^{\frac13}$. Since \eqref{eq:tightness} is composed of a sum with $2^{2n}$ terms of products of moments of the form $m^\eps_{n,n}$, we deduce from complex Gaussian statistics \cite{reed1962moment} that 
\[ \E |\phi^\eps(s,r,x+h) - \phi^\eps(s,r,x)|^{2n} = n! \Big( \E |\phi^\eps(s,r,x+h) - \phi^\eps(s,r,x)|^{2}\Big)^n + O(\eps^{\frac13}).\]
In the diffusive regime,  we have
\[ \E |\phi^\eps(z,r,x+h) - \phi^\eps(z,r,x)|^{2} = m^\eps_{1,1}(z,r,x+h,x+h) + m^\eps_{1,1}(z,r,x,x) - m^\eps_{1,1}(z,r,x,x+h)-m^\eps_{1,1}(z,r,x+h,x). \]

We start from the simpler setting of an incident beam of the form \eqref{eq:u0}, where
\begin{align*}
    m^\eps_{1,1}(z,r,x,x+h)= \int_{\mathbb{R}^{2d}} & \hat\sff(\xi)\hat\sff^\ast(\zeta)e^{ir\cdot(\xi-\zeta)}e^{i\eps^\beta\eta x\cdot(\xi-\zeta)}e^{-i\eps^{\beta}\eta h\cdot\zeta}e^{-\frac{i\eta z}{2k_0}\eps^{2\beta-1}(|\xi|^2-|\zeta|^2)}
    \\ & \times e^{\frac{k_0^2z}{4\eta^2}\int\limits_{0}^1Q\big(\eta h +\eta\eps^{\beta-1}\frac{(\xi-\zeta)sz}{k_0}\big)\mathrm{d}s}\frac{\mathrm{d}\xi\mathrm{d}\zeta}{(2\pi)^{2d}}.
\end{align*}
We treat the harder case $\beta=1$ to simplify notation and leave the case $\beta>1$ to the reader.  We write
\begin{equation}\label{eq:decphase}
    e^{-i\eps^{\beta}\eta h\cdot\zeta} = 1 \ + \ (e^{-i\eps^{\beta}\eta h\cdot\zeta}-1)
\end{equation}
with the latter term bounded by $\eps \eta |h||\zeta|$, which provides a negligible contribution to the above integral over $\Rm^{2d}$ assuming $\aver{\xi}\hat\sff(\xi)$ integrable. A similar argument applies to $m^\eps_{1,1}(z,r,x+h,x)$ and $m_{1,1}^\eps(z,r,x+h,x+h)$ as well. The exponential factor involving $|\xi|^2-|\zeta|^2$ is independent of $h$ and of modulus (bounded by) $1$. Define $\sC=\frac{k_0^2 z}{4}$ and $\sth=\frac{(\xi-\zeta)z}{k_0}$.
A bound on $\E|\delta\phi^\eps|^2$ is thus obtained by bounding the term
\[
  T := 2e^{\frac{\sC}{\eta^2} \int_0^1 Q(\eta\sth s) ds} - e^{\frac{\sC}{\eta^2} \int_0^1 Q(\eta(\sth s+h)) ds} - e^{\frac{\sC}{\eta^2} \int_0^1 Q(\eta(\sth s-h)) ds}.
\]
In the diffusive regime, we assume that $\aver{\xi}^2 \hat R(\xi)\in \sL^1(\Rm^d)$, which implies that $\nabla^2 Q(x) = \nabla^2 R(x)$ is uniformly bounded. We also recall that $\nabla Q(0)=0$ since $R(x)$ is maximal at $x=0$. By assumption, we have
\[Q(\eta \sth s+\eta h) = Q(\eta\sth s) + \eta h\cdot\nabla Q(\eta \sth s) + \frac12 \eta^2 (h \cdot\nabla)^2 Q(\eta \sth s +\tau \eta h)\]
for some $0\leq \tau\leq1$. By uniform bound on $\nabla^2 Q(x)$, we may approximate $e^{\frac C2 (h\cdot\nabla)^2 Q(\eta \sth s +\tau \eta h)}$ by $1$ up to an error of order $O(|h|^2)$ and thus obtain (uniformly in $\sth$)
\[T = e^{\frac \sC{\eta^2}\int_0^1 Q(\eta \sth s)ds} (2- e^{\frac \sC{\eta} h\cdot\int_0^1 \nabla Q(\eta \sth s)ds} -e^{-\frac \sC{\eta} h\cdot\int_0^1 \nabla Q(\eta\sth s)ds}) + O(|h|^2).\]

We next use that for $a\in\Rm$, then $|2-e^a-e^{-a}|\leq Ca^2 e^{|a|}$ and since $Q(x)\leq0$ that we have $Q(\eta\theta s) \pm \eta h\cdot\nabla Q(\eta \sth s)\leq O(\eta^2|h|^2)$, again uniformly in $\sth$. This shows that 
\[ |T| \leq C \Big| \frac{1}{\eta} h\cdot \int_0^1 \nabla Q(\eta\sth s)\Big|^2 + C |h|^2.\]
Since $\nabla Q$ is globally Lipschitz and $\nabla Q(0)=0$, we deduce from $|\nabla Q(\eta\sth s)|\leq C\eta |\sth|s$ that 
\[ |T| \leq C |h|^2 (1+  |\xi|^2+|\zeta|^2).\]
Upon integrating the above inequality in $(\xi,\zeta)\in\Rm^{2d}$, we obtain the bound:
\begin{equation}\label{eq:bdsquareu0}
    \E |\phi^\eps(x+h) - \phi^\eps(x)|^{2} \leq C( |h|^2 + |h|\eps).
\end{equation}

The incident condition \eqref{eq:u0plane} is treated similarly with the following extensions. We recall that $m^\eps_{1,1}(z,r,x,y)$ is given in \eqref{eqn:m_11_plane_wave}-\eqref{eqn:I_mn} where we have for $\beta=1$:
     \[ \begin{aligned}
    \mathcal{I}_{\sm,\sn}^\eps(z,r,x,y)
    = & \int_{\mathbb{R}^{2d}}
    \hat{\sff}_\sm(\xi)\, \hat{\sff}_{\sn}^*(\zeta)
    e^{\frac{-iz\eta}{2k_0\eps}(|\eps\xi+\sk_\sm|^2-|\eps\zeta+\sk_{\sn}|^2)}  e^{ir\cdot(\xi-\zeta)} e^{i\frac{r}{\eps}\cdot(\sk_\sm-\sk_\sn)} e^{i\eps\eta (\xi\cdot x-\zeta\cdot y)}
    \\ &\quad \times  
    e^{i \eta (\sk_\sm \cdot x- \sk_\sn\cdot y)}
    \mQ\big(\eta(y-x), \eta(\xi-\zeta) + \frac \eta \eps (\sk_\sm-\sk_\sn)\big)
    \frac{\mathrm{d}\xi \mathrm{d}\zeta}{(2\pi)^{2d}}.
    \end{aligned}\]
   
    Start with the case $\sm=\sn$. Up to a negligible term as in \eqref{eq:decphase}, we may replace $e^{i\eps\eta (\xi\cdot x-\zeta\cdot y)}$ by $1$. The quadratic terms still do not depend on $h$.  Then, the above term $T$ is being replaced by
    \[
    T_\sm = 2e^{\frac{\sC}{\eta^2} \int_0^1 Q(\eta\sth s) ds} - e^{\frac{\sC}{\eta^2} \int_0^1 Q(\eta(\sth s+h)) ds} e^{-i\eta \sk_\sm\cdot h} - e^{\frac{\sC}{\eta^2} \int_0^1 Q(\eta(\sth s-h)) ds}e^{i\eta \sk_\sm\cdot h}.
    \]
    We have the approximation
    \[T_\sm = e^{\frac \sC{\eta^2}\int_0^1 Q(\eta \sth s)ds} (2- e^{\frac \sC{\eta} h\cdot\int_0^1 \nabla Q(\eta \sth s)ds} e^{-i\eta \sk_\sm\cdot h}-e^{-\frac \sC{\eta} h\cdot\int_0^1 \nabla Q(\eta\sth s)ds}e^{i\eta \sk_\sm\cdot h}) + O(|h|^2).\]
    The linear terms in $h$ still cancel out so that $|T_\sm| \leq C |h|^2 (1+|\xi|^2+|\zeta|^2)$ as for $T$. 

    It remains to consider $\sm\not=\sn$ with terms of the form
    \[\begin{aligned}
    \mathcal{I}_{\sm,\sn}^\eps(z,r,x,y)
    = & \int_{\mathbb{R}^{2d}}
    \hat{\sff}_\sm(\xi)\, \hat{\sff}_{\sn}^*(\zeta)
    \ \breve\Phi\ 
    \mQ\big(\eta(y-x), \eta(\xi-\zeta) + \frac \eta \eps (\sk_\sm-\sk_\sn)\big)
    \frac{\mathrm{d}\xi \mathrm{d}\zeta}{(2\pi)^{2d}},
    \end{aligned}\]
    where $\breve\Phi\in\C$ satisfies $|\breve\Phi|=1$.

    Consider the domain $D$ 
    where $|\xi|+|\zeta|\leq \eps^{-\frac12}|\sk_\sm-\sk_\sn|$ and its complementary $\Rm^{2d}-D$. Since $|\mQ|\leq1$ and $|\xi|\hat{\sff}_\sm(\xi)\in \sL^1(\Rm^d)$, the contribution to $\mathcal{I}_{\sm,\sn}^\eps$ on $\Rm^{2d}-D$ is of order $\eps^{\frac12} |\sk_\sm-\sk_\sn|^{-1}$ using that $1\leq \eps^{\frac12} |\sk_\sm-\sk_\sn|^{-1} (|\xi|+|\zeta|)$ on $\Rm^{2d}-D$.

    On $D$, we have $|\eta(\xi-\zeta) + \frac \eta \eps (\sk_\sm-\sk_\sn)|\geq C \eta \eps^{-\frac12}$ for $C>0$.
    We thus observe from 
    \[ \mQ(\tau_1,\tau_2)=e^{-\frac{k_0^2 z}{4\eta^2} R(0)} e^{\frac{k_0^2 z}{4\eta^2} \int_0^1 R(\tau_1+\frac{zs}{k_0} \tau_2) ds} =e^{-\frac{k_0^2 z}{4\eta^2} R(0)} e^{\frac{k_0^3}{4\eta^2|\tau_2|} \int_0^{\frac{|\tau_2|z}{k_0}} R(\tau_1+\hat\tau_2 s) ds} \]
    for $|\tau_2|\geq C_1\eta\eps^{-\frac12}\geq C \eps^{-\frac13}$  for $C>0$ and from the line integrability of $R(x)$ that 
    \[ \mQ\big(\eta(y-x), \eta(\xi-\zeta) + \frac \eta \eps (\sk_\sm-\sk_\sn)\big) = e^{-\frac{k_0 z}{4\eta^2} R(0)} + O_{\|\cdot\|_\infty}(\eps^{\frac13}).\]
    Up to a term of order $\eps^{\frac13}$, we may thus replace $\mQ$ by a constant term in the analysis of a term $T_{\sm,\sn}$ generalizing the term $T$ above. The only remaining dependence in $h$ is a phase term treated as \eqref{eq:decphase}. The contribution to $\E |\phi^\eps(z,r,x+h) - \phi^\eps(z,r,x)|^{2}$ stemming from the cross-terms $\sk_\sm\not=\sk_\sn$ is therefore bounded by $C\eps^{\frac13}$.

Summarizing the above, we obtain for the incident beam in \eqref{eq:u0plane} an estimate of the form
\begin{equation}\label{eq:bdsquareu0plane}
    \E |\phi^\eps(z,r,x+h) - \phi^\eps(z,r,x)|^{2} \leq C( |h|^2 + |h|\eps + \eps^{\frac13}).
\end{equation}

We deduce from the above and \eqref{eq:tighteta} with $\alpha=1$ (since we require that both $\aver{\xi}^2\hat \sff_\sm$ and $\aver{\xi}^2\hat R$ be integrable), that for an incident beam of the form \eqref{eq:u0plane} or \eqref{eq:u0}, we have
\[ \E |\phi^\eps(s,r,x+h) - \phi^\eps(s,r,x)|^{2n} \leq C_n \Big(\big[(|h|^{2}+|h|\eps + \eps^{\frac13})^n + \eps^{\frac13}\big] \wedge \frac{|h|^{2n}}{\eta^{2n}}\Big).\]
It remains to choose $h_0=\eps^{\frac{1}{6n}}$, use the first estimate for $|h|\geq h_0$ and the second estimate for $|h|\leq h_0$ to observe that since $\eta^{-1}=\ln\ln\eps^{-1}$,
\[
\E |\phi^\eps(s,r,x+h) - \phi^\eps(s,r,x)|^{2n} \leq C_g |h|^{2n-g}
\]
for any $g>0$ with $C_g$ independent of $h\in(0,1]$ and $\eps$ sufficiently small.

This concludes the derivation of \eqref{eq:tightness} and the proof of the theorem for the process $\phi^\eps$.
The deterministic function $\E\phi^\eps(z,r,x)$ is also clearly uniformly bounded and appropriately H\"older continuous (since the initial conditions are as we just saw and evolution consists of a simple multiplication by a $z-$dependent term) so that the continuity and tightness properties we obtained for $\phi^\eps$ extend to the centered process $\phi^\eps-\E\phi^\eps$.
This concludes the proof of the Theorem.
\end{proof}
%
%
\paragraph{Comparison with existing work.} The authors in \cite{gu2021gaussian} analyze the kinetic regime of the scintillation scaling for the same It\^o-Schr\"odinger paraxial model as in this work. They analyze a compensated wave field in the {\em Fourier} variables, which in the kinetic regime takes the form $(z,\zeta)\to\hat u^\eps(z,\xi+\eps\zeta)e^{\frac{i}{\eps k_0}z|\xi+\eps\zeta|^2}$. Using the martingale structure of the random wave process and a reasonable amount of diagrammatic expansions, they show that this process is tight on continuous functions $(z,\zeta)\in\Rm_+\times\Rm^d$ and converges in distribution as $\eps\to0$ to a limit $\hat\phi$ such that $\hat\phi-\E\hat\phi$ is complex Gaussian. The main difference with our current work is in the regime of incident beams, which in \cite{gu2021gaussian} takes the form of \eqref{eq:u0} with $\beta=0$. Such narrow incident beams experience significantly more dispersion than for larger values of $\beta$. It is not clear whether complex Gaussian structures in the {\em physical} variables as we displayed in Theorems \ref{thm:kinetic} and \ref{thm:diffusive} may be obtained for narrow beams corresponding to small values of $\beta$. The information they obtain on $\hat u^\eps(z,\xi+\eps\zeta)e^{\frac{i}{\eps k_0}z|\xi+\eps\zeta|^2}$ is, however, sufficient to infer statistical stability of mesoscopic kinetic quantities defined in phase space.

We note that analysis of $(z,\xi)\to\hat u^\eps(z,\xi)e^{\frac{i}{\eps k_0}z|\xi|^2}$ has been carried out in a general paraxial model in \cite{bal2011asymptotics}. The convergence as $\eps\to0$ to a limit $\hat\phi$ such that $\hat\phi-\E\hat\phi$ is complex Gaussian is then established using diagrammatic expansions.

\medskip

As we indicated in the introduction section, the kinetic regime of the scintillation scaling for the It\^o-Schr\"odinger paraxial model for broad beams is analyzed in detail in a series of papers  \cite{garnier2014scintillation,garnier2016fourth,garnier2018noninvasive,garnier2022scintillation}. Their theory focuses on the analysis of the partial differential equations \eqref{eq:mupqF} for first moments, second moments  $p=q=1$, and fourth moments $p=q=2$ and their limits as $\eps\to0$ when augmented with initial conditions of the form \eqref{eq:u0} or \eqref{eq:u0plane} with $\beta=1$. In particular, they prove that (in our notation) for the {\em physical} variables $X=(x_1,x_2)$ and $Y=(y_1,y_2)$
\[
  \mu^\eps_{2,2}(X,Y)= \mu^\eps_{1,1}(x_1,y_1)\mu^\eps_{1,1}(x_2,y_2) + \mu^\eps_{1,1}(x_1,y_2)\mu^\eps_{1,1}(x_2,y_1) - \mu^\eps_{1,0}(x_1)\mu^\eps_{1,0}(x_2)\mu^\eps_{0,1}(y_1)\mu^\eps_{0,1}(y_2) + o(1)
\]
with an error term $o(1)\to0$ in the uniform sense. The above estimate is readily derived from Theorem \ref{thm:kinetic}. One of our main contributions is to show that this error term up to log factors is of order $O(\eps)$ when $d\geq2$, which allows us to consider the diffusive regime, and to extend the analysis to arbitrary moments $\mu^\eps_{p,q}$. While the derivations are slightly different, many of our convergence results in section \ref{sec:firstmoments} appear in a similar form in \cite{garnier2014scintillation,garnier2016fourth,garnier2018noninvasive,garnier2022scintillation}.  This allows these works to show that in the scintillation regime $\eps\to0$ and formally sending the rescaled distance $z$ to infinity, then the scintillation index \eqref{eq:scintindex} indeed converges to $1$ as obtained in Corollary \ref{cor:scint}. This result provides significant evidence for the complex Gaussianity of the limiting process we obtained in Theorem \ref{thm:diffusive} and the corresponding exponential distribution for the energy density (irradiance) as shown in Corollary \ref{cor:scint}. We refer to these aforementioned works for many explicit expressions for second moments and scintillation indices associated to $u^\eps$ in different contexts. See also \cite{hislop2019transport} for an analysis of second moments of similar equations beyond the scintillation regime.

\paragraph{Natural extensions.} The results stated in Theorems \ref{thm:kinetic} and \ref{thm:diffusive} show that $u^\eps$ converges to a complex Gaussian process after appropriate rescaling for {\em coherent} incident beams $u_0$. However, the theory presented in section \ref{sec:highermoments} is more general.  In particular,  the estimates \eqref{eq:bdmueps} for solutions of \eqref{eq:mupqgal} in Theorem \ref{thm:mupqphys} hold for arbitrary incident conditions at $z=0$, which do not need to be of the form given in \eqref{eq:mupqpde}. The theory of this paper allows us to analyze partially coherent beams and obtain regimes where the scintillation index is significantly smaller than unity, as a violation of the results of  Theorems \ref{thm:kinetic} and \ref{thm:diffusive}{\color{black}~\cite{bal2025long}}.  The analysis of the partial differential equations satisfied by the statistical moments also extends to random media $B(z,x;t)$ that depend slowly on a macroscopic parameter $t$ and allow us to consider time-averaged irradiance measurements. {\color{black} Also see} \cite{garnier2018noninvasive,garnier2022scintillation} for analyses of the scintillation index in these physical settings.
As supported by extensive experimental studies and optimization strategies~\cite{gbur2014partially, nelson2016scintillation, nair2023scintillation}, partially coherent beams are widely acknowledged to have a superior performance in mitigating the effects of turbulence. 

We note also that the limiting diffusive model of Theorem \ref{thm:diffusive} required the spatial correlation function $R(x)$ to be sufficiently smooth at $x=0$. Such an assumption is not necessary to establish convergence of finite moments in Theorem \ref{thm:kinetic} or tightness in Theorem \ref{thm:tightness}. It is possible that an algebraic decay $\hat R(k)\sim |k|^{-\gamma}$ as $|k|\to\infty$ leads to another limit than the standard diffusion limit given in section \ref{sec:firstmoments}. Such a regime will be analyzed elsewhere. 
 
When the medium has long range correlations, an It\^{o}-Schr\"{o}dinger approximation with the standard Brownian motion as discussed here is no longer valid. However we note that it is possible that the results presented here can be extended to cases as in~\cite{borcea2023paraxial}, where a similar limiting model has been obtained after an appropriate random phase transformation.

\section*{Acknowledgments} This work was supported in part by NSF Grant DMS-2306411. {\color{black} We also thank the two anonymous referees for their valuable comments in improving the presentation of this paper.}
\section*{Data Availability Statement} Data sharing is not applicable to this article as no datasets were generated or analyzed during the current study.

\bibliographystyle{siam}
\bibliography{Reference}

\end{document}